\documentclass[reqno,english,12pt]{amsart}

\usepackage[margin=1in]{geometry}

\usepackage{amstext,amsthm,amsfonts,amssymb,euscript,mathrsfs,graphics,color,xcolor,amsmath,amssymb,latexsym}
\usepackage{amsthm}
\newtheorem{definition}{Definition}
\usepackage{hyperref}

\makeatletter

\newcommand{\lyxmathsym}[1]{\ifmmode\begingroup\def\b@ld{bold}
	\text{\ifx\math@version\b@ld\bfseries\fi#1}\endgroup\else#1\fi}

\numberwithin{equation}{section}
\numberwithin{figure}{section}

\theoremstyle{plain}

\newtheorem{thm}{\protect\theoremname}[section]
\theoremstyle{plain}
\newtheorem{cor}[thm]{\protect\corollaryname}
\newtheorem{rem}[thm]{\protect\remarkname}
\theoremstyle{plain}
\newtheorem{lem}[thm]{\protect\lemmaname}
\theoremstyle{plain}
\newtheorem{prop}[thm]{\protect\propositionname}

\makeatother

\def\tofill{\vskip30pt $\cdots$ To fill in $\cdots$ \vskip30pt}

\newcommand{\comment}[1]{\vskip.3cm
	\fbox{%
		\parbox{0.93\linewidth}{\footnotesize #1}}
	\vskip.3cm}

\def\wtF{\widetilde{{\mathcal{F}}}}
\def\whF{\widehat{{\mathcal{F}}}}
\def\jx{\langle x \rangle}
\def\jy{\langle y \rangle}
\def\jk{\langle k \rangle}

\def\wt{\widetilde}

\def\R{\mathbb{R}}
\def\C{\mathbb{C}}

\def\K{\mathcal{K}}
\def\e{\epsilon}

\usepackage{babel}
\providecommand{\corollaryname}{Corollary}
\providecommand{\lemmaname}{Lemma}
\providecommand{\propositionname}{Proposition}
\providecommand{\remarkname}{Remark}
\providecommand{\theoremname}{Theorem}

\begin{document}
	
	\title[Cubic NLS with a potential]{The $1$d nonlinear Schr\"odinger equation with \\ a weighted $L^1$ potential}

	\author{Gong Chen and Fabio Pusateri}
	
	\date{\today}
	
	\address{University of Toronto, Department of Mathematics, 40 St George Street,
		Toronto, ON, M5S 2E4, Canada.}
	
	\email{gc@math.toronto.edu}
	
	\email{fabiop@math.toronto.edu}

	\begin{abstract}
		We consider the $1d$ cubic nonlinear Schr\"odinger equation with a large external potential $V$ with no bound states.
		We prove global regularity and quantitative bounds for small solutions under mild assumptions on $V$.
		In particular, we do not require any differentiability of $V$, 
		and make spatial decay assumptions that are weaker than those 
		found in the literature (see for example \cite{Del,N,GPR}).
		We treat both the case of generic and non-generic potentials,
		with some additional symmetry assumptions in the latter case.
		
		Our approach is based on the combination of three main ingredients: 
		the Fourier transform adapted to the Schr\"odinger operator,
		basic bounds on pseudo-differential operators that exploit the structure of the Jost function, 
		and improved local decay and smoothing-type estimates.
		An interesting aspect of the proof is an ``approximate commutation'' identity
		for a suitable notion of a vectorfield, which allows us to simplify the previous approaches 
		and extend the known results to a larger class of potentials.
		%
		%
		%
		%
		Finally, under our weak assumptions we can include the interesting physical case of a barrier potential
		as well as 
		recover the result of \cite{MMS} for a delta potential.
		%
	\end{abstract}
	
	\maketitle
	
	\setcounter{tocdepth}{1}
	
	\tableofcontents

	\section{Introduction}
	In this paper we study the global-in-time behavior 
	of small solutions to the one-dimensional cubic nonlinear Schr\"odinger equation with  a large external potential
	\begin{equation}\label{eq:NLSE}
		i\partial_{t}u-\partial_{xx}u + V(x)u \pm \left|u\right|^{2}u = 0,\qquad u(t=0,x)=u_{0}(x),
	\end{equation}
	for an unknown $u: \R_t\times \R_x \longrightarrow \C$, with initial data
	$u_{0}$ in a weighted $L^2$ (or Sobolev) space. 
	The aims of this paper are twofold.
	First, we enlarge the class of potentials $V$ for which one can obtain
	global-in-time bounds and long-time asymptotics for \eqref{eq:NLSE};
	second, we provide an alternative approach, compared to recent works on this and similar problems
	(see for example \cite{N,Del,GPR,MMS}), 
	that we believe could be helpful to treat open problems on 
	the stability of solitons for nonlinear Schr\"odinger and related models.
	
	Here are the assumptions on the class of potentials that we are going treat:
	
	\medskip
	\subsubsection*{Assumptions on $V$}
	It is well-known that for $V\in L^1$ the spectrum of the Schr\"odinger operator $H=-\partial_{xx}+V$
	consists of absolutely continuous spectrum on $(0,\infty)$ plus a finite number of eigenvalues
	(see for example Reed-Simon \cite{Reed-Simon}).
	
	In this paper we assume that 
	\begin{equation}\label{noboundstates}
		H \quad \mbox{has no eigenvalues.}
	\end{equation}
	This is the case if, for example, $V\geq 0$. 
	One important aspect of our results is that we are able to treat both {\it generic} and {\it non-generic}
	classes of potentials. See Definition \ref{def:generic}.
	We make the following mild assumptions on the decay of $V$ at infinity:
	
	\setlength{\leftmargini}{1.5em}
	\begin{itemize}
		
		\item If $V$ is generic we assume that
		\begin{equation}\label{Vgeneric}
			\jx^\gamma V\in L^1 \,\,\,  \text{with} \,\,\, \gamma>5/2.
		\end{equation} 
		
		\item If $V$ is non-generic we assume that
		\begin{equation}\label{Vnongeneric}
			\jx^\gamma V\in L^1 \,\,\,  \text{with} \,\,\, \gamma>7/2
		\end{equation}
		and, in addition, that 
		\begin{equation}\label{nongenericresonance}
			\widetilde{\mathcal{F}}[e^{it(-\partial_{xx}+V)}u](0)=0, \qquad 
			\widetilde{\mathcal{F}}[e^{it(-\partial_{xx}+V)}|u|^2u](0)=0,
		\end{equation}
		where $\widetilde{{\mathcal{F}}}$ is the distorted Fourier transform associated to $H$. 
		See Subsection \ref{ssecDFT} for its definition. 
		
		We remark that the condition \eqref{nongenericresonance} can be easily satisfied by imposing, 
		for example, that the potential and the zero energy resonance 
		are even and that the initial datum $u_0$ is odd (so that the solution $u$ is odd for all times). 
		See Lemma \ref{rem0} and Remark \ref{lemtildeF}.
	\end{itemize}
	
	Notice that in either cases we do not require any differentiability for $V$.
	
	%

	\medskip
	\subsection{Background and previous results}
	First, recall that sufficiently regular solutions of \eqref{eq:NLSE} conserve the $L^{2}$ norm
	\[M(u) := \int\left|u\right|^{2}\,dx \]
	and the total energy (Hamiltonian):
	\[
	H\left(u\right) := \frac{1}{2}\left|\partial_xu\right|^{2}+\frac{1}{2}V\left|u\right|^{2}\pm\frac{1}{4}\left|u\right|^{4}.
	\]
	The Cauchy problem for \eqref{eq:NLSE} with $V=0$ - we will refer to this as the ``free'' or ``flat'' case -
	is globally well-posed in $L^{2}$, see for example
	Cazenave-Weissler \cite{CW} and Cazenave-Naumkin \cite{CN}. 
	Our main interest is the global-in-time bounds and asymptotic behavior as $|t|\rightarrow \infty$. 
	The main feature of the cubic nonlinearity is its criticality with respect to scattering:
	linear solutions of the Schr\"odinger equation decay at best like $|t|^{-1/2}$ in $L^\infty_x$, 
	so that, when evaluating the nonlinearity on linear solutions, one see that $|u|^{2}u \sim |t|^{-1}u$;
	the non-integrability of $|t|^{-1}$ 
	results in a ``Coulomb''-type contribution of the nonlinear terms.
	See \cite{SS93} and 
	references therein for works on the classical question of asymptotic completeness 
	for the linear many-body problem with a Coulomb potential.
	
	
	In the case $V=0$ the problem is well understood. 
	Solutions of \eqref{eq:NLSE} with $V=0$ and initial data $u|_{t=0}\in H^{1}\cap L^{2}(x^{2}dx)$
	(i.e. bounded energy and variance) are known to exhibit modified scattering as time goes to infinity:
	they decay at the same rate of linear solutions 
	but their asymptotic behavior differs from linear solutions by a logarithmic phase correction. 
	Using complete integrability this was proven in the seminal work of Deift-Zhou \cite{DZ};
	see also \cite{DZ2} on nonlinear perturbations of the defocusing cubic NLS.
	Without making use of complete integrability, and
	restricting the analysis to small solutions, 
	proofs of modified scattering were given by Hayashi-Naumkin \cite{HN}, Lindblad-Soffer \cite{LS},
	Kato-Pusateri \cite{KP} and Ifrim-Tataru \cite{IT}.
	
	Recently, the above results for small solutions have been extended to the full problem with potential \eqref{eq:NLSE}
	in the works of Naumkin \cite{N}, Delort \cite{Del} and Germain-Pusateri-Rousset \cite{GPR}.
	These works treat 
	potentials of sufficient regularity and decay, and establish modified scattering results similar to 
	the flat case.
	More precisely, \cite{N} and \cite{GPR} both assume that the potential is generic;
	\cite{N} assumes that $V \in W^{1,1}$ and $\jx^\gamma V\in L^1$ with $\gamma>5/2$; 
	similar but slightly stronger assumptions are made in \cite{GPR}.
	The work \cite{Del} also considers non-generic cases\footnote{More precisely 
		the case of the so-called ``very exceptional'' potentials, that is,
		\[\int V(x)m_+(x,0)\,dx=0 \qquad \text{and} \qquad \int V(x)xm_+(x,0)\,dx=0;\]
		see \eqref{mpm} for the definition of $m_+$.
	}
	under symmetry assumptions (even potential, odd data) akin to \eqref{nongenericresonance}.
	The recent work of Masaki-Murphy-Segata \cite{MMS} treats the special case of a delta potential; 
	compared to the works cited above, in \cite{MMS} the potential has no regularity
	but, on the other hand, it has the advantage of being explicitly calculable\footnote{In the sense that
		one has explicit formulas for the generalized eigenfunctions associated to the Schr\"odinger operator 
		$-\partial_{xx} + \delta$, and for the scattering matrix (transmission and reflection coefficients).}.

	
	\medskip
	\subsection{Ideas of the proof}
	Our proof is based on three main ingredients:
	the Fourier transform adapted to the Schr\"odinger operator $H=-\partial_{xx}+V$,
	local decay and smoothing estimates, and $L^2$-bounds on pseudo-differential operators
	whose symbols are given by the Jost functions for $H$.
	
	Our first step  is to use the Fourier Transform adapted to $H$ -
	the so-called ``Distorted Fourier Transform'' - to rewrite \eqref{eq:NLSE} in (distorted) Fourier space.
	For the sake of this brief introduction, it suffices
	to admit for the moment the existence of ``generalized plane waves''   $\mathcal{K}(x,k)$ such that
	one can define an $L^2$ unitary transformation $\wtF$ by
	\begin{align}\label{tildeFint}
		\wtF[f](k) := \widetilde{f}(k) := \int \overline{\mathcal{K}(x,k)}f(x)\,dx,
		\quad \mbox{with} \quad \wtF^{-1}\left[\phi\right](x)=\int\mathcal{K}(x,k)\phi(k)\,dk.
	\end{align}
	See \eqref{matK}, \eqref{psipm} and \eqref{defT},
	for the precise definition of $\mathcal{K}(x,k)$ and its relation
	with the generalized eigenfunctions of $H$.
	The distorted transform $\wtF$ diagonalizes the Schr\"odinger operator:
	$-\partial_{xx}+V=\wtF^{-1}k^{2}\wtF$.

	Given a solution $u$ of \eqref{eq:NLSE} we let $f = e^{-itH} u$ be its linear profile, so that 
	$\widetilde{f}(t,k) = e^{-itk^2}\wt{u}(t,k)$.
	One can prove the basic linear estimate 
	\begin{align}\label{intlinest}
		{\| u(t,\cdot) \|}_{L^\infty_x} \lesssim \frac{1}{|t|^{1/2}} {\| \wt{f}(t) \|}_{L^\infty_k}
		+ \frac{1}{|t|^{3/4}} {\| \partial_k \wt{f}(t) \|}_{L^2_k}
	\end{align}
	which is the analogue of the standard linear estimate for the case $V=0$
	(where one can replace ${\| \partial_k \wt{f}(t) \|}_{L^2}$ by 
	a standard weighted norm ${\| xf(t) \|}_{L^2} = {\| Ju(t) \|}_{L^2}$, with $J = x+2it\partial_x$).
	To obtain the sharp pointwise decay of $|t|^{-1/2}$ it then suffices to control 
	$\wt{f}$ uniformly in $k$ and $t$ and the $L^2$-norm of $\partial_k \wt{f}(t)$
	with a small growth in $t$. 
	Both of these bounds are achieved by studying the equation in the distorted Fourier space as we briefly described below.
	
	In the distorted Fourier space the Duhamel's formula associated to \eqref{eq:NLSE} becomes
	\begin{align}\label{introD}
		\begin{split}
			\tilde{f}(t,k) & = \tilde{f}(0,k) \pm i \, \mathcal{N}_\mu[f,f,f](t,k) 
			\\ 
			\mathcal{N}_\mu[f,f,f](t,k) & := \int_{0}^{t} \iiint e^{is(-k^2+\ell^2-m^2+n^2)}
			\tilde{f}(s,\ell)\overline{\tilde{f}(s,m)}\tilde{f}(s,n)\mu(k,\ell,m,n) \,dndmd\ell ds
		\end{split}
	\end{align}
	where we have defined
	the {\it nonlinear spectral distribution}
	\begin{align}\label{intromu}
		\mu(k,\ell,m,n):=\int\overline{\mathcal{K}(x,k)}\mathcal{K}(x,\ell)\overline{\mathcal{K}(x,m)}\mathcal{K}(x,n)\,dx.
	\end{align}
	To obtain the desired bounds on $\wt{f}$ we need to understand the structure of $\mu$.
	To do this we first decompose 
	$\K = \K_S + \K_R$ where:
	$\K_S$ is linear combination of exponentials $e^{\pm i xk}$ whose coefficients depend on the sign of $k$ and $x$,\footnote{The
		distinction between $x$ positive or negative is done by inserting smooth cutoffs,
		but let us be more informal at this stage.}
	and therefore resembles a (flat) plane wave;
	$\K_R$ is the component arising from the interaction with the potential
	and has strong localization in $x$ and is uniformly regular in $k$.
	See \eqref{eq:decomK} and \eqref{Ksing}-\eqref{KR} for the precise formulas.
	
	According to this basic decomposition, 
	we propose a splitting of $\mu$ into two pieces: $\mu = \mu_S + \mu_R$,
	where $\mu_S$ only contains the interaction of the four $\K_S$ functions having argument $x$ of the same sign,
	and $\mu_R$ is all the rest. 
	We call $\mu_S$ the ``singular' part of $\mu$
	and $\mu_R$, the``regular'' part of $\mu$.
	We then define $\mathcal{N}_S = \mathcal{N}_{\mu_S}$, respectively $\mathcal{N}_R := \mathcal{N}_{\mu_R}$, 
	to be the singular, respectively, the regular, part of the nonlinear terms $\mathcal{N}_\mu$ in \eqref{introD}.
	
	We treat the two components $\mathcal{N}_S$ and $\mathcal{N}_R$ separately by relying on two main observations:
	a {\it commutation} property with $\partial_k $ for the singular part, 
	and the {\it localization} property of the regular part.
	More precisely, with a simple explicit calculation, we show that the multilinear commutator
	between $\partial_k$ and $\mathcal{N}_S$ satisfies
	\begin{align}\label{introcomm}
		[\partial_k,\mathcal{N}_S] = \mathcal{N}_S^{'}
	\end{align}
	where $\mathcal{N}_S^{'}$ 
	is a localized term of the form $a(x) |u|^2 u$, for a Schwartz function $a$. 
	This last term is then very easy to handle using the localized decay estimates
	that we present below, and that are used to estimate $\mathcal{N}_R$ as well.
	
	The regular part $\mathcal{N}_R$ can be thought of as 
	the (flat) transform of a nonlinear term of the form  $\jx^{-\rho}|u|^2u$, 
	for some $\rho >0$ related to the decay of $V$. 
	More precisely, we can view it as (the transform of) a localized trilinear term
	whose inputs are pseudo-differential operators applied to the solution $u$ that satisfy
	$L^2$ and $L^\infty$ type estimates similar to those satisfied by $u$ itself.
	
	It is not hard to see that applying $\partial_k$ to $\mathcal{N}_R$ essentially amounts to multiplying it by 
	a factor of $tk$.
	Then, we are reduced to estimating the $L^2_x$ norm of an expression of the form 
	\begin{align*}
		\int_0^t s \, \langle \partial_x \rangle \jx^{-\rho} |u(s)|^2 u(s)\, ds.
	\end{align*}
	This naturally leads us to investigate {\it local decay and smoothing-type estimates} for 
	$u=e^{itH}f$.
	Importantly, we need to do this under the sole assumptions that we can control (up to some small growth in time)
	$\partial_k \wt{f}$ in $L^2_k$.
	
	Examples of the estimates that we prove are the localized improved decay
	\begin{align}\label{introloc1}
		{\Big\| \jx^{-1} e^{iHt} f \Big\|}_{L^\infty_x}
		\lesssim |t|^{-\frac{3}{4}}{\big\| \tilde{f} \big\|}_{H^1_k}, 
	\end{align}
	and the localized $L^2$ smoothing/improved decay estimate
	\begin{align}\label{introloc2}
		{\Big\| \jx^{-1} \partial_x e^{iHt}f \Big\|}_{L^2}
		\lesssim |t|^{-\frac{3}{4}} {\big\| \tilde{f} \big\|}_{H^1_k}. 
	\end{align}
	These hold for all potentials 
	(also non-generic ones) provided $\wt{f}(0)=0$;
	see Lemmas \ref{lemlocdecinftyng} and \ref{lem:locl2ng} for the precise statements.
	For generic potentials \eqref{introloc2} 
	can be improved by replacing $|t|^{-3/4}$ by $|t|^{-1}$, and
	this allows us to put weaker assumption on $V$ in the generic case.
	\eqref{introloc1} can be similarly improved for generic $V$. 
	Although we do not use this estimate, we still provide its proof in Appendix \ref{sec:localdecay}, 
	since we believe this is of independent interest,
	and will be useful when dealing with perturbations of solitons and solitary waves.\footnote{Also note
		that \eqref{introloc1} with a power of $|t|^{-1}$
		is at the same ``scaling'' of the linear estimate \eqref{KS}
		which appears in Krieger and Schlag \cite{KS}.}

	Finally, let us discuss another important aspect our proofs that is related to the assumptions on the potential $V$,
	and the fact that we are able to not impose any bound on any of its derivatives.
	In the proof of the linear estimates mentioned above,
	and of the nonlinear estimates for the regular part of the distribution $\mu_R$,
	one naturally faces the need to bound {\it pseudo-differential operators} (PDOs)
	with symbols given by the functions $\K=\K(x,k)$, or their components (such as $\K_R$),
	as well as their first order derivatives in $x$, and up to second order derivatives in $k$.
	Typically, derivatives in $k$ can be controlled by the decay of $V$,
	while derivatives in $x$ require regularity.
	In particular, one should notice that the symbols under consideration are bounded in $x$ and $k$,
	but, for potentials which are only in a weighted $L^1$ space,
	they are not (uniformly) regular in $x$: their second derivatives in $x$ grow with $k$. 
	Therefore, 
	these symbols do not belong to standard symbol classes
	for which one has boundedness of the associated operator on $L^2$ via known results.
	We then make the simple observation that these symbols solve some basic integral equations,
	and, therefore, are quite explicit.
	We then exploit the precise structure of these equations 
	and obtain the necessary PDO bounds under our weak hypotheses.

	\medskip
	\subsection{Results}
	Here is our main result on the global existence and sharp pointwise decay
	for solutions of \eqref{eq:NLSE}.

	\begin{thm}\label{thm:main1}
		Consider the nonlinear Schr\"odinger equation with a potential
		\begin{equation}\label{eq:NLSE-1}
			i\partial_{t}u-\partial_{xx}u\pm\left|u\right|^{2}u+V(x)u=0,\ u\left(0\right)=u_{0},
		\end{equation}
		where $V$ and $u$ satisfy the assumptions \eqref{noboundstates}, 
		\eqref{Vgeneric} in the generic case and \eqref{noboundstates}, 
		\eqref{Vnongeneric}, \eqref{nongenericresonance} in the non-generic case. 
		Then we have the following:
		
		
		There exists $0<\epsilon_{0}\ll 1$ 
		such that for all $\eta\leq\epsilon_{0}$ and 
		\begin{align}\label{datasmall}
			\left\Vert u_{0}\right\Vert _{H^{1,1}}=
			\left\Vert u_{0}\right\Vert _{H^{1}}+\left\Vert xu_{0}\right\Vert _{L^{2}}=\eta
		\end{align}
		the equation \eqref{eq:NLSE-1} has a unique
		global solution $u\in C(\R,H^1(\R))$, with $u(0,x)=u_{0}(x)$, and satisfying the sharp decay rate
		\begin{align}\label{main1fdecay}
			\sup_{t\in\mathbb{R}}\left\Vert u(t)\right\Vert_{L^\infty_x}
			\lesssim \frac{\eta}{\left(1+\left|t\right|\right)^{\frac{1}{2}}}. 
		\end{align}
		
		Moreover, if we define the profile of the solution $u$ as
		\begin{align}
			\label{main1prof}
			f\left(t,x\right):=e^{-it\left(-\partial_{xx}+V\right)}u\left(t,x\right),
			\qquad \tilde{f}\left(t,k\right):=e^{-itk^{2}}\tilde{u}\left(t,k\right),
		\end{align}
		where $\tilde{g}=\widetilde{\mathcal{F}}g$ denotes the distorted Fourier transform, 
		then 
		\begin{align}\label{main1fbounds}
			{\big\| \tilde{f}(t) \big\|}_{L_k^\infty}
			+ (1+|t|)^{-\alpha} {\| \partial_{k}\tilde{f}(t) \|}_{L_k^2} \lesssim\eta
		\end{align}
		for some $\alpha=\alpha(\gamma)>0$ small enough.
		
		Finally, we have the following asymptotics:
		there exists $W_{+\infty}\in L^{\infty}$ such that
		\begin{align}\label{mainasy}
			\left|\tilde{f}\left(t,k\right)\exp\left(\frac{i}{2}\int_{0}^{t}\left|\tilde{f}\left(s,k\right)\right|^{2}
			\frac{ds}{s+1}\right)-W_{+\infty}(k)\right| \lesssim \eta \, t^{-\beta}
		\end{align}
		for some $\beta\in(0,\alpha)$ as $t\rightarrow\infty$.
	\end{thm}
	
	Combining \eqref{mainasy} above, and the linear asymptotic formula
	\begin{align}\label{linearasy}
		u\left(t,x\right)=\frac{e^{i\frac{x^{2}}{4t}}}{\sqrt{-2it}}
		\tilde{f}\left(t,-\frac{x}{2t}\right)+\mathcal{O}\left(t^{-\frac{1}{2}-\alpha}\right),\ \ t\gg1,
	\end{align}
	one can also derive the following asymptotic formula for $u$ in physical space:
	\begin{align}\label{nonlinearasy}
		u\left(t,x\right)=\frac{e^{i\frac{x^{2}}{4t}}}{\sqrt{-2it}}
		\exp\left(-\frac{i}{2}\left|W_{+\infty}\left(-\frac{x}{2t}\right)\right|^{2}\log t\right)
		W_{+\infty}\left(-\frac{x}{2t}\right)+\mathcal{O}\left(t^{-\frac{1}{2}-\alpha}\right),\ \ t\gg1.
	\end{align}
	
	
	For $t \rightarrow -\infty$ we can use the time-reversal symmetry as follows.
	
	\begin{rem}[Asymptotics for negative times]\label{remnegativetimes}
		Define $v(t,x)$ by
		\[
		\overline{v}(s,x) = u(t,x), \,\, s=-t.
		\]
		Note that $v$ is also a solution of the equation \eqref{eq:NLSE}.
		Then, the asymptotic behavior of $u$ as $t\rightarrow -\infty$ 
		is given by the (complex conjugate of the) asymptotics of $v$ as $s \rightarrow \infty$. 
		This in turn can be calculated from the (transform of the) profile of $v$
		according to \eqref{linearasy} and \eqref{mainasy},
		leading to the analog of \eqref{nonlinearasy} for negative times.
		
		More precisely, if the profile of $u(t,x)$ is $f$ in \eqref{main1prof},
		then the profile of $v$ is 
		\begin{align}\label{gfbar}
			g(s,x):=e^{-is\left(-\partial_{xx}+V\right)}v(s,x) = \overline{f}(-s,x),
			\qquad \widetilde{g}(s,k)=\widetilde{\overline{f}}(-s,x)
		\end{align}
		and we express this in terms of $\widetilde{f}(-s,k)$.
		By direct computations, starting from the definition of the distorted transform \eqref{tildeF},
		and using \eqref{eq:f_+-1},  \eqref{eq:f_--1}, one can see that,
		%
		for $k<0$,
		\begin{align}\label{fbartilda}
			\left(\begin{array}{c}
				\widetilde{\overline{f}}(k)
				\\
				\widetilde{\overline{f}}(-k)
			\end{array}\right) 
			& =\left(\begin{array}{cc}
				T(k) & R_+(k)
				\\
				R_{-}(k) & T(k)
			\end{array}\right)
			\left(\begin{array}{c}
				\overline{\widetilde{f}(-k)}
				\\
				\overline{\widetilde{f}(k)}
			\end{array}\right)
			:= S(k) \overline{\left(\begin{array}{c}
					\widetilde{f}(-k)
					\\
					\widetilde{f}(k)
				\end{array}\right)}.
		\end{align}
		where $S$ is the scattering matrix associated to the potential $V$; see \eqref{eq:scatMatrix}.
		
		From \eqref{linearasy}, it follows that we can write
		\[ v(s,x)=\frac{e^{ix^2/4s}}{\sqrt{-2is}}\widetilde{g}\left(s,-\frac{x}{2s}\right)
		+ \mathcal{O}\left(s^{-1/2-\alpha}\right), \qquad s \gg 1.\]
		hence, for $t \ll -1$,
		\begin{align}\label{negf0}
			\begin{split}
				u(t,x) 
				& 
				= \frac{e^{ix^{2}/4t}}{\sqrt{-2it}}
				\overline{\widetilde{\overline{f}}\left(t,\frac{x}{2t}\right)}+ \mathcal{O}\left(|t|^{-1/2-\alpha}\right).
			\end{split} 
		\end{align}
		Then, by \eqref{fbartilda}, for $x\leq0$ and $k_0:=x/2t>0$, we can express
		\begin{subequations}\label{negf1}
			\begin{align}
				\overline{\widetilde{\overline{f}}(t,k_0)}= T(k_0) \widetilde{f}(t,-k_0) + R_{-}(k_0)\widetilde{f}(t,k_0) 
			\end{align}
			and, for $x\geq0$,
			\begin{align}
				\overline{\widetilde{\overline{f}}(t,k_0)}= T(-k_0) \widetilde{f}(t,-k_0) + R_{+}(-k_0)\widetilde{f}(t,k_0).
			\end{align}
		\end{subequations}
		
		As in \eqref{mainasy}, one has that in terms of $\tilde{g}$, there exists $U_{+\infty}\in L^{\infty}$ such that
		\begin{align}\label{negg}
			\left|\tilde{g}\left(s,k\right)\exp\left(\frac{i}{2}\int_{0}^{s}\left|\tilde{g}\left(\tau,k\right)\right|^{2}
			\frac{d\tau}{\tau+1}\right) - U_{+\infty}(k)\right| \lesssim \eta s^{-\beta}
		\end{align}
		for some $\beta\in(0,\alpha)$ as $s\rightarrow\infty$.
		For $k\leq0$, define 
		\[
		G\left(s,k\right):=\left(\begin{array}{c}
			\tilde{g}(s,k)\\
			\tilde{g}(s,-k)
		\end{array}\right),
		\quad F\left(-s,-k\right):=\left(\begin{array}{c}
			\tilde{f}(-s,-k)\\
			\tilde{f}(-s,k)
		\end{array}\right),
		\\
		\quad 
		\mathrm{U}_{+\infty}\left(k\right) := \left(\begin{array}{c}
			U_{+\infty}(k)\\
			U_{+\infty}(-k)
		\end{array}\right).
		\]
		Then \eqref{fbartilda} reads
		\[
		G\left(s,k\right)=S\left(k\right)\overline{F\left(-s,-k\right)},
		\]
		and defining the operator $\mathcal{M}$ mapping a column vector to a $2\times2$ matrix by
		\[
		\mathcal{M}\left(\begin{array}{c}
			q_{1}\\
			q_{2}
		\end{array}\right):=\mathrm{diag}\left(\left|q_{1}\right|^{2},\left|q_{2}\right|^{2}\right),
		\]
		and taking complex conjugate in \eqref{negg},
		we get
		\begin{align}\label{negf4}
			\left|\overline{S\left(k\right)} F\left(-s,-k\right)-\exp\left(\frac{i}{2}\int_{0}^{s}
			\mathcal{M}\overline{S\left(k\right)} F\left(-s,-k\right)\frac{d\tau}{\tau+1}\right)
			\overline{\mathrm{U}_{+\infty}(k)}\right|\lesssim \eta s^{-\beta}.
		\end{align}
		It follows that 
		\[
		\Big| \big| \overline{S\left(k\right)}
		F\left(t,-k\right) \big| - \left| \mathrm{U}_{+\infty}(k)\right| \Big|
		\lesssim \eta |t|^{-\beta}, \qquad t \ll -1,
		\]
		and, since $S$ is unitary, $\tilde{f}(t,k)$ is bounded for negative times and \eqref{negf4} implies
		\begin{align}\label{negf5}
			\left| F\left(t,-k\right)- S(k)^T \exp\left(\frac{i}{2}
			\mathcal{M} \mathrm{U}_{+\infty}(k) \log |t| \right)
			\overline{\mathrm{U}_{+\infty}(k)}\right|\lesssim \eta s^{-\beta}.
		\end{align}
		
		Eventually one can put together \eqref{negf0}, \eqref{negf1} and \eqref{negf5}
		to derive the formula for the asymptotics in physical space as $t \rightarrow -\infty$
		analogous to \eqref{nonlinearasy}.

		%
		
	\end{rem}

	In this paper, we will mostly focus our attention 
	on the proof of the global bounds \eqref{main1fbounds}, 
	and, in particular, on the weighted-$L^2$ bound.
	As a corollary of our estimate we will also obtain the Fourier-$L^\infty$ bound for 
	most of the interactions (the ones corresponding to the regular measure $\mu_R$ introduced after \eqref{intromu}).
	Finally we will also give some details of the derivation of the $L^\infty$ bound 
	for the singular part of the measure. 

	\smallskip
	We now give some consequences of our results and corollaries of our proofs.
	
	
	\medskip
	\begin{enumerate}
		\item\label{remextracoeff} {\bf Non-constant coefficients}.
		Our analysis can be naturally generalized to the case of the equation
		\begin{align}\label{eqextracoeff}
			i\partial_{t}u-\partial_{xx}u+Vu+a(x)|u|^{2}u=0
		\end{align}
		where $a(x)\to1$ as $|x|\to\infty$ sufficiently fast. 
		Of course, one needs to assume the proper analogue of \eqref{nongenericresonance} in the non-generic case;
		e.g. $a$ is even in the case of an odd solution and an even zero energy resonance.
		
		One can carry out the proof in the exact same way by adding the term $(a(x)-1)|u|^2u$ 
		to the contribution from the regular part of the measure $\mu_R$, as defined in \eqref{eq:muR}.
		More precisely, one adds to the analysis a term like \eqref{musum} with a factor of $a(x)-1$ included in the integrand,
		and then notices that this has the same properties of the measure $\mu_{R,1}$ defined in \eqref{muR1},
		provided that $|a(x)-1|\lesssim \jx^{-\kappa}$, for $\kappa$ large enough.
		Our proof shows that $\kappa > 2$ suffices.

		\medskip
		\item {\bf Non-gauge invariant nonlinearities}.
		We can obtain the same result of Theorem \ref{thm:main1} also for models 
		with localized non-gauge invariant non-linear terms, that is, for equations of the form
		\begin{align}\label{eqextranon}
			i\partial_{t}u-\partial_{xx}u+Vu = |u|^{2}u + a_1(x)|u|^{2}u + a_2(x)u^3 + a_3(x)\bar{u}^3 + a_4(x)|u|^2\bar{u}
		\end{align}
		where $a_i$, $i=1,\dots 4$, are $C^1$ coefficients that decay fast enough, $|a_i(x)| \lesssim \jx^{-\kappa}$, 
		for $\kappa>2$.
		
		Equation \eqref{eqextranon} is an example of a simplified model
		that can arise from studying the stability of (topological) solitons in various settings;
		see for example the introduction of \cite{KGV}.
		The extension of our analysis from \eqref{eq:NLSE-1} to \eqref{eqextranon} is immediate
		(again assuming the proper analogue of \eqref{nongenericresonance} in the non-generic case).
		Indeed, all the cubic terms with a decaying coefficient can be included in the treatment of the 
		nonlinear terms corresponding to the regular part of the measure in Section \ref{sec:estimateregular},
		just by noticing that the argument there is insensitive to presence of $u$ or $\bar{u}$.
		
		
		\medskip
		\item {\bf The case $V=0$}.
		As a particular case of a non-generic potential, one can choose $V=0$.
		Then, the distorted Fourier transform coincides with the regular one and conditions
		\eqref{nongenericresonance} can be obtained by simply assuming that the data is odd, $u_0(x)=-u_0(-x)$.
		Moreover, in view of the \eqref{remextracoeff} above,
		we can also add any even decaying coefficient $b\in C^1$ in front of the nonlinearity and consider the model
		\begin{align}\label{eqflat}
			i\partial_{t}u-\partial_{xx}u = \pm |u|^{2}u + b(x)|u|^{2}u.
		\end{align}
		This gives us a result analogous to the recent result of \cite{LLS}
		on the Klein-Gordon equation (see also \cite{LS} and \cite{Sterb}). 
		
		
		
		\medskip
		\item {\bf Barrier potentials and the delta potential}.
		Given non-negative $K$ and $L$, consider \eqref{eq:NLSE-1} with a potential 
		\begin{equation}\label{eq:Potential}
			V(x) = V(K,L) :=
			\begin{cases}
				K & -L\leq x\leq L
				\\
				0 & \text{otherwise}
			\end{cases}.
		\end{equation}
		$V(K,L)$ represents a barrier of width $2L$ and height $K$. 
		For discussions about the physical importance of this potential see, for instance, 
		\cite{AGHH}.
		Note that since $K\geq0$, $V$ has no bound states. Moreover, it is a generic potential
		and therefore Theorem \ref{thm:main1} can be applied. 
		
		Furthermore, we observe that the potential $V(\epsilon^{-1},\epsilon)$ converges, 
		as $\epsilon \rightarrow 0$, to $2\delta_{0}$
		in the sense of distributions, where $\delta_0$ is Dirac's delta.
		Since $V(\epsilon^{-1},\epsilon)$ has uniformly bounded weighted $L^1$-norms,
		so do the corresponding solutions of NLS, in the space that we are considering. 
		
		One can also check that 
		the Jost functions for the barrier potentials converge nicely to the Jost functions 
		associated to the $2\delta_0$ potential. 
		All of these can be computed explicitly; see for example \cite{LLbook1},
		and also Kulaev-Shabat \cite{KuSh} for a discussion of more general distributional potentials. 
		In view of the convergence of the Jost functions, 
		the spectral distributions will also converge and, in the case of \eqref{eq:Potential}, 
		the regular part will disappear when passing to the limit.
		
		Using our result and a limiting argument one can then obtain global-in-time bounds and asymptotics
		for the cubic NLS with a repulsive delta potential $V = 2\delta_0$.
		This provides an alternative proof of the recent result of Masaki-Murphy-Segata \cite{MMS}. 
		
	\end{enumerate}
	


	\medskip
	\subsection{Organization}
	In Section \ref{sec:JostSpec} we recall basic properties of 
	the generalized eigenfunction for the Schr\"odinger operator 
	and the distorted Fourier transform;
	we then establish a series of elementary estimates for pseudo-differential operators that have 
	symbols related to the generalized eigenfunctions (or Jost solutions).
	In Section \ref{sec:Linear} we first present the decomposition of the eigenfunctions
	into a flat-looking and a localized components,
	and then establish linear estimates, including pointwise decay and local improved decay. 
	In Sections \ref{sec:Cubic} and \ref{sec:estimateregular},
	we give the nonlinear estimates for the `singular' and `regular' parts of the cubic terms. 
	In Section \ref{sec:pointwisebound} we discuss the pointwise bound for the 
	distorted Fourier transform of the profile.
	Finally, in Appendix \ref{sec:localdecay}, we present an improved local decay estimate for generic potentials.
	
	\medskip
	\subsection*{Notation}
	As usual, \textquotedblleft $A:=B\lyxmathsym{\textquotedblright}$
	or $\lyxmathsym{\textquotedblleft}B=:A\lyxmathsym{\textquotedblright}$
	is the definition of $A$ by means of the expression $B$. We use
	the notation $\langle x\rangle := (1+|x|^{2})^{\frac{1}{2}}$.
	For positive quantities $a$ and $b$, we write $a\lesssim b$ for
	$a\leq Cb$ where $C$ is a universal constant,
	and $a\simeq b$ when $a\lesssim b$ and $b\lesssim a$. 
	We denote $u_{t}:=\frac{\partial}{\partial_{t}}u$, $u_{xx}:=\frac{\partial^{2}}{\partial x^{2}}u$.
	The (regular) Fourier transforms is defined as
	\begin{equation}
		\hat{h}\left(\xi\right)=\mathcal{F}\left[h\right]\left(\xi\right)=\frac{1}{\sqrt{2\pi}}
		\int e^{-ix\xi}h(x)\,dx.\label{eq:FT}
	\end{equation}


	\medskip
	\subsection*{Acknowledgments}
	F.P. was supported in part by a startup grant from the University of Toronto, a
	Connaught Fund New Researcher grant, and the NSERC grant RGPIN-2018-06487.

	\bigskip
	\section{Jost functions, spectral theory and PDO estimates \label{sec:JostSpec}}
	
	In this section, we collect some basic facts about Jost functions and the distorted Fourier transform.
	
	\subsection{Basic properties of general Jost functions}
	In this subsection, we recall some basic properties of Jost functions.
	The facts provided here hold for all potentials, see 
	\cite{DT} and 
	\cite{GPR}.
	
	The Jost functions $\psi_+(x,k)$ and $\psi_-(x,k)$
	are defined as solutions to 
	\begin{align}\label{psipm}
		H\psi_{\pm}(x,k)=\left(-\partial_{xx}+V\right)\psi_{\pm}(x,k)=k^{2}\psi_{\pm}(x,k)
	\end{align}
	such that
	\begin{align}\label{psipmlim}
		\lim_{x\rightarrow\infty}\left|e^{-ikx}\psi_+(x,k)-1\right|=0,
		\qquad  \lim_{x\rightarrow-\infty}\left|e^{ikx}\psi_-(x,k)-1\right|=0. 
	\end{align}
	We let 
	\begin{align}\label{mpm}
		m_{\pm}(x,k)=e^{\mp ikx}\psi_{\pm}(x,k).
	\end{align}
	Then for fixed $x$, $m_{\pm}$ is analytic in $k$ for $\Im k>0$
	and continuous up to $\Im k\geq0$.
	
	We define
	\begin{align}\label{defW_+-}
		\mathcal{W}_{+}^{s}(x)=\int_x^\infty \jy^{s}\left|V(y)\right|\,dy,
		\qquad \mathcal{W}_{-}^{s}(x)=\int_{-\infty}^{x}\jy^{s}\left|V(y)\right|\,dy.
	\end{align}
	Note that if $V(y)$ decays fast enough, then $\mathcal{W}_{\pm}^{s}(x)$
	also decay as $x\rightarrow\pm\infty$ respectively.
	
	\begin{lem}\label{lem:Mestimates}
		For every $s\ge0$, we have the estimates: 
		\begin{align}\label{Mestimates1}
			\begin{split}
				& \left|\partial_{k}^{s}\left(m_{\pm}(x,k)-1\right)\right|\lesssim\frac{1}{\jk}\mathcal{W}_{\pm}^{s+1}(x),\qquad \pm x\geq-1,
				\\
				& \left|\partial_{k}^{s}\left(m_{\pm}(x,k)-1\right)\right|\lesssim\frac{1}{\jk}\jx^{s+1}, \qquad \pm x\leq1.
			\end{split}
		\end{align}
		Also,
		\begin{align}\label{Mestimates1'}
			\begin{split}
				& \left|\partial_{k} \left(m_{\pm}(x,k)-1\right)\right|\lesssim\frac{1}{|k|}\mathcal{W}_{\pm}^{1}(x),\qquad \pm x\geq-1.
			\end{split}
		\end{align}
		Moreover
		\begin{align}\label{Mestimates2}
			\begin{split}
				& \left|\partial_{k}^{s} \partial_x m_{\pm}(x,k) \right| \lesssim \mathcal{W}_{\pm}^{s}(x), \qquad \pm x\geq-1,
				\\
				& \left|\partial_{k}^{s} \partial_x m_{\pm}(x,k) 
				\right|\lesssim\jx^{s}, \qquad \pm x\leq1.
			\end{split}
		\end{align}
	\end{lem}
	
	\begin{proof}
		The proofs of these estimates follow from analyzing the Volterra equation satisfied by $m_\pm$,
		that is,
		\begin{align}\label{minteq}
			m_\pm(x,\lambda)=1\pm\int_x^{\pm\infty} D_{\lambda}(\pm(y-x))V(y)m_\pm(y,\lambda)\,dy,
			\qquad D_{\lambda}(x)=\frac{e^{2i\lambda x}-1}{2i\lambda};
		\end{align} 
		as in Deift-Trubowitz \cite{DT}, Weder \cite{We1} or \cite[Appendix A]{GPR}. 
		The estimates \eqref{Mestimates1'} follow from Lemma 2.1 in \cite{We2}.
	\end{proof}

	Denote $T(k)$ and $R_{\pm}(k)$ the transmission
	and reflection coefficients associated to the potential $V$ respectively.
	For more details, see Deift-Trubowitz \cite{DT}. With these coefficients,
	one can write
	\begin{equation}\label{eq:f_+-1}
		\psi_{+}(x,k)=\frac{R_{-}(k)}{T(k)}\psi_{-}(x,k)+\frac{1}{T(k)}\psi_{-}\left(x,-k\right)
	\end{equation}
	\begin{equation}\label{eq:f_--1}
		\psi_{-}(x,k)=\frac{R_{+}(k)}{T(k)}\psi_{+}(x,k)+\frac{1}{T(k)}\psi_{+}\left(x,-k\right).
	\end{equation}
	Using these formulae, one has
	\[
	\psi_{+}(x,k)\sim\frac{R_{-}(k)}{T(k)}e^{-ikx}+\frac{1}{T(k)}e^{ikx}\ \ \ x\rightarrow-\infty,
	\]
	\[
	\psi_{-}(x,k)\sim\frac{R_{+}(k)}{T(k)}e^{ikx}+\frac{1}{T(k)}e^{-ikx}\ \ \ x\rightarrow\infty.
	\]
	Moreover, these coefficients are given explicitly by
	\begin{align}\label{defT}
		\frac{1}{T(k)}=1-\frac{1}{2ik}\int V(x)m_{\pm}(x,k)\,dx, 
	\end{align}
	\begin{align}\label{defR}
		\frac{R_{\pm}(k)}{T(k)}=\frac{1}{2ik}\int e^{\mp2ikx}V(x)m_{\mp}(x,k)\,dx, 
	\end{align}
	and satisfy
	\begin{align}\label{TRid}
		\begin{split}
			& T\left(-k\right)=T(k),\ \ \ R_{\pm}\left(-k\right)=\overline{R_{\pm}(k),}
			\\
			& \left|R_{\pm}(k)\right|^{2}+\left|T(k)\right|^{2}=1,\ \ T(k)\overline{R_{-}(k)}+\overline{T(k)}R_{+}(k)=0.
		\end{split}
	\end{align}
	
	\begin{definition}\label{def:generic}
		$V$ is defined to be a ``generic'' potential if
		\[
		\int V(x)m_{\pm}\left(x,0\right)\,dx\neq0.
		\]
	\end{definition}
	
	If a potential is generic, then by the relation given above, we know that
	\[
	T\left(0\right)=0,\ \ \ R_{\pm}\left(0\right)=-1.
	\]
	Given $T(k)$ and $R_{\pm}(k)$ as above, we
	can define the scattering matrix associated to the potential by
	\begin{equation}
		S(k):=\left(\begin{array}{cc}
			T(k) & R_{+}(k)
			\\
			R_{-}(k) & T(k)
		\end{array}\right),\ \ S^{-1}(k):=\left(\begin{array}{cc}
			\overline{T(k)} & \overline{R_{-}(k)}
			\\
			\overline{R_{+}(k)} & \overline{T(k)}
		\end{array}\right).\label{eq:scatMatrix}
	\end{equation}
	
	We have the following lemma on the coefficients.

	\begin{lem}\label{estiTR}
		Assuming that $\jx^{2}V\in L^1$, we have the uniform estimates for $k\in\mathbb{R}$:
		\[
		\left|\partial_{k}T(k)\right|+\left|\partial_{k}R_{\pm}(k)\right|\lesssim\frac{1}{\jk}.
		\]
	\end{lem}
	
	Moreover, for a generic potential, the associated transmission and
	reflection coefficients have the following Taylor expansions near
	$k\sim0$. For a detailed proof, see page 144 in Deift-Trubowitz \cite{DT}.
	
	\begin{lem}\label{lem:estiTRTaylor}
		Assuming that $\jx^2 V\in L^1$ and $V$
		is generic, then
		\[
		T(k)=\alpha k+\mathcal{O}(k^2),\ \alpha\neq0,\ \text{as}\ k\rightarrow0,
		\]
		and
		\[
		1+R_{\pm}(k)=\alpha_{\pm}k+\mathcal{O}(k^2),\ \text{as}\ k\rightarrow0.
		\]
	\end{lem}

	\medskip
	\subsection{Pseudo-differential bounds}
	Starting from the estimates for the Jost functions in Lemma \ref{lem:Mestimates}, 
	we want to obtain $L^2\mapsto L^2$ bounds for pseudo-differential operators whose symbols are given by $m_\pm(x,\lambda)$
	and their derivatives.
	One possibility to obtain this type of bounds is to apply a criterion 
	of Hwang \cite{Hwang}, also used in \cite{GPR}, requiring the essential boundedness of 
	the first derivatives (in $x$ and $\lambda$) and the second mixed derivative of the symbol.
	However, the use of this tool as a black-box requires sufficiently strong conditions on the potentials,
	both in terms of decay and regularity.
	The approach we present below appears to be more general, since it requires less assumptions on the potential
	(no Sobolev regularity and less spatial decay)
	and at the same time more direct, as it relies on the explicit structure of the Jost functions and their derivatives.
	
	
	
	\begin{lem}\label{lem:pseeasy}
		Suppose $\jx^{\gamma}V\in L^1$ with $\gamma>3/2$. 
		Then
		\begin{equation}\label{eq:pseeasy}
			\left\Vert \mathrm{1}_{x\geq-1}\int e^{i\lambda x}m_+(x,\lambda)g(\lambda)\,d\lambda
			\right\Vert_{L^2_x}\lesssim{\| g \|}_{L^2}.
		\end{equation}
		Similar estimates hold for $m_{-}$ restricted on $x\leq1$.	
		%
		
		In addition, if we assume $\gamma>\beta+3/2$, then
		\begin{equation}\label{eq:mPDO+}
			\left\Vert \mathrm{1}_{x\geq-1}\jx^{\beta}
			\int e^{i\lambda x}\left(m_+(x,\lambda)-1\right)g(\lambda)\,d\lambda\right\Vert_{L^2_x}
			\lesssim \left\Vert g\right\Vert_{L^{2}}.
		\end{equation}
		Similar estimates hold for $m_{-}-1$ restricted on $x\leq1$.
	\end{lem}
	
	\begin{proof}
		We only prove the bound for $m_+$ since the one for $m_-$ follows identically.
		Recall that $m_+$ satisfies the integral equation \eqref{minteq}
		\begin{align}\label{m_+inteq}
			m_+=1+\int_x^\infty D_{\lambda}(y-x)V(y)m_+(y,\lambda)\,dy,
			\qquad D_{\lambda}(x)=\frac{e^{2i\lambda x}-1}{2i\lambda}.
		\end{align} 
		It suffices to bound $L^{2}$ norm of
		\[
		\int e^{i\lambda x}\left(\int_x^\infty D_{\lambda}(y-x)V(y)m_+(y,\lambda)\,dy\right)g(\lambda)\,d\lambda
		\]
		
		We split the integral into  $|\lambda|\leq1$ and $|\lambda|\geq1$.
		For the second region, using the boundedness of Jost function, see Lemma \ref{lem:Mestimates},
		$|D_\lambda|\lesssim |\lambda|^{-1}$, and Cauchy-Schwarz:
		\begin{align*}
			\left\Vert \mathrm{1}_{x\geq-1}\int_{|\lambda|\ge1}e^{i\lambda x}\left(\int_x^\infty D_{\lambda}(y-x)V(y)m_+(y,\lambda)\,dy\right)g(\lambda)\,d\lambda\right\Vert_{L^2_x}\\
			\lesssim \left\Vert \mathrm{1}_{x\geq-1}\int_{|\lambda|\ge1}\int_x^\infty \frac{1}{|\lambda|}\left|V(y)m_+(y,\lambda)\right|\,dy\left|g(\lambda)\right|\,d\lambda\right\Vert_{L^2_x}\\
			\lesssim \left\Vert \mathrm{1}_{x\geq-1}\int_x^\infty \left|V(y)\right|\,dy\right\Vert_{L^{2}_{x}} {\| g \|}_{L^2}
			\lesssim {\| g \|}_{L^2}.
		\end{align*}
		
		For the region $|\lambda|\leq1$, we use $|D_{\lambda}(z)| \lesssim|z|$
		and Cauchy-Schwarz, to obtain
		\begin{align*}
			\left\Vert \mathrm{1}_{x\geq-1}\int_{|\lambda|\leq1}e^{i\lambda x}
			\left(\int_x^\infty D_{\lambda}(y-x)V(y)m_+(y,\lambda)\,dy\right)g(\lambda)\,d\lambda\right\Vert_{L^2_x}
			\\
			\lesssim\left\Vert \mathrm{1}_{x\geq-1}\int_{|\lambda|\leq1}\int_x^\infty |y|\left|V(y)m_+(y,\lambda)\right|\,dy\left|g(\lambda)\right|\,d\lambda\right\Vert_{L^2_x}\\
			\lesssim\left\Vert \mathrm{1}_{x\geq-1}\int_x^\infty |y|\left|V(y)\right|\,dy\right\Vert_{L^{2}_{x}} {\| g \|}_{L^2}\lesssim{\| g \|}_{L^2}.
		\end{align*}
		having used the assumption $\jx^\gamma V \in L^1$ for $\gamma > 3/2$.
		
		To show \eqref{eq:mPDO+}, we can proceed as above and simply notice that, for $\gamma-\beta>\frac{3}{2}$
		\begin{align*}
			\left\Vert \mathrm{1}_{x\geq-1}\jx^{\beta}\int_{|\lambda|\ge1}e^{i\lambda x}\left(\int_x^\infty D_{\lambda}(y-x)V(y)m_+(y,\lambda)\,dy\right)g(\lambda)\,d\lambda\right\Vert_{L^2_x}\\
			\lesssim\left\Vert \mathrm{1}_{x\geq-1}\jx^{-\gamma+\beta}\int_x^\infty \jy^{\gamma}\left|V(y)\right|\,dy\right\Vert_{L^{2}_{x}} {\| g \|}_{L^2}
			\lesssim{\| g \|}_{L^2}
		\end{align*}
		and
		\begin{align*}
			\left\Vert \mathrm{1}_{x\geq-1}\jx^{\beta}\int_{|\lambda|\leq1}e^{i\lambda x}
			\left(\int_x^\infty D_{\lambda}(y-x)V(y)m_+(y,\lambda)\,dy\right)g(\lambda)\,d\lambda\right\Vert_{L^2_x}
			\\
			\lesssim\left\Vert \mathrm{1}_{x\geq-1}\jx^\beta\int_{|\lambda|\leq1}\int_x^\infty |y|\left|V(y)m_+(y,\lambda)\right|\,dy\left|g(\lambda)\right|\,d\lambda\right\Vert_{L^2_x}
			\\
			\lesssim\left\Vert \mathrm{1}_{x\geq-1}\jx^{-\gamma+\beta+1}\int_x^\infty \jy^{\gamma}\left|V(y)\right|\,dy\right\Vert_{L^{2}_{x}} {\| g \|}_{L^2}\lesssim{\| g \|}_{L^2}.
		\end{align*}
	\end{proof}

	By duality, we have the following:
	\begin{lem}\label{lem:pseeasy-1}
		Suppose $\jx^{\gamma}V\in L^1$ with $\gamma>3/2$. Then one has the following estimate
		\begin{equation}\label{eq:pseeasy-1}
			\left\Vert \int e^{i\lambda x}m_+(x,\lambda)\mathrm{1}_{x\geq-1}h(x)\,dx\right\Vert_{L_{\lambda}^{2}}
			\lesssim\left\Vert h\right\Vert_{L^2}.
		\end{equation}
		Similar estimates hold for $m_{-}$ restricted on $x\leq1$.
		
		For $\gamma>\beta+3/2$,
		\begin{equation}\label{eq:mPDO+-1}
			\left\Vert \int e^{i\lambda x}\left(m_+(x,\lambda)-1\right)\mathrm{1}_{x\geq-1}h(x)\,dx\right\Vert_{L_{\lambda}^{2}}
			\lesssim\left\Vert \jx^{-\beta}h\right\Vert_{L^2_x}.
		\end{equation}
		Similar estimates hold for $m_{-}-1$ restricted to $x\leq1$.
	\end{lem}

	We also need estimates for pseudo-differential operators whose symbols involve of derivatives of $m_{\pm}$.
	
	\begin{lem}\label{lem:psehard1}
		Suppose $\jx^{\gamma}V\in L^1$ with $\gamma>\max(\beta/2+3/4,\beta)$.
		Then
		\begin{equation}\label{eq:psehard}
			\left\Vert \jx^{\beta} \mathrm{1}_{x\geq-1}\int e^{i\lambda x}\partial_xm_+(x,\lambda)g(\lambda)\,d\lambda\right\Vert_{L^2_x}\lesssim{\| g \|}_{L^2}.
		\end{equation}
		Similar estimates hold for $\partial_xm_{-}$ by restricting to $x\leq1$.
	\end{lem}
	
	\begin{proof}
		From \eqref{m_+inteq} we can write
		\begin{align}\label{eq:split}
			\begin{split}
				\partial_xm_+(x,\lambda) & =\int_x^\infty e^{2i\lambda(y-x)}V(y)m_+(y,\lambda)\,dy
				\\
				& = \int_x^\infty e^{2i\lambda(y-x)}V(y)\left(m_+(y,\lambda)-1\right)\,dy
				+\int_x^\infty e^{2i\lambda(y-x)}V(y)\,dy.
			\end{split}
		\end{align}
		From the estimate of the Jost functions, see Lemma \ref{lem:Mestimates}, we know that
		\[ |m_+(y,\lambda)-1| \lesssim \frac{1}{\langle\lambda\rangle}\mathcal{W}_{+}^{1}(y), \quad y\geq-1,\]
		see the definition \ref{defW_+-}.
		Then
		\begin{align*}
			\left\Vert \mathrm{1}_{x\geq-1} \jx^{\beta} \int
			e^{i\lambda x}\left(\int_x^\infty e^{2i\lambda(y-x)}V(y)
			\left(m_+(y,\lambda)-1\right)\,dy\right)g(\lambda)\,d\lambda\right\Vert_{L^2_x}
			\\
			\lesssim \left\Vert \mathrm{1}_{x\ge-1} \jx^{\beta} \int_x^\infty \left|V(y)
			\right|\mathcal{W}_{+}^{1}(y)\,dy
			\, \Big( \int \frac{1}{\langle\lambda\rangle} \left|g(\lambda)\right|\,d\lambda \Big) \, \right\Vert_{L^2_x}
			\\
			\lesssim\left\Vert \mathrm{1}_{x\ge-1}\jx^{\beta} \int_x^\infty |V(y)| \, 
			\mathcal{W}_{+}^{1}(y)\,dy\right\Vert_{L^2_x} {\| g \|}_{L_{\lambda}^{2}}
		\end{align*}
		where in the last inequality we applied Cauchy-Schwarz in $\lambda$.
		Note that if $\jx^{\gamma}V\in L^1$ with $\gamma\geq1$
		then, for $y \geq -1$ we have $|\mathcal{W}_{+}^{1}(y)| \lesssim {\jy}^{1-\gamma} {\| \jx^{\gamma}V \|}_{L^1}$
		so that
		\begin{align*}
			\left\Vert \mathrm{1}_{x\geq-1} \jx^{\beta}\int_x^\infty |V(y)| \, \mathcal{W}_{+}^{1}(y)\,dy \right\Vert_{L^2_x}
			\lesssim \left\Vert \jx^{\beta-2\gamma+1} \int \jy^\gamma |V(y)| \,dy \right\Vert_{L^2_x}
			\left\Vert \jx^{\gamma}V\right\Vert_{L^1}.
		\end{align*}
		This is finite if $2\gamma>\beta +3/2$, which is the case under our assumptions.
		
		
		
		
		It remains to bound the integral involving the last term in \eqref{eq:split}. 
		Exchanging the order of integration we notice that
		\begin{align*}
			\int e^{i\lambda x}\left(\int_x^\infty e^{2i\lambda(y-x)}V(y)\,dy\right)g(\lambda)\,d\lambda
			= \sqrt{2\pi} \int_x^\infty V(y) \, \whF^{-1}(g)\left(2y-x\right)\,dy.
		\end{align*}
		It follows from the Young inequality that
		\begin{align*}
			{\left\| \jx^\beta \int e^{i\lambda x}\left(\int_x^\infty e^{2i\lambda(y-x)}V(y)\,dy\right)g(\lambda)\,d\lambda \right\|}_{L^2_x}
			\\ 
			\lesssim {\left\| \int_{-\infty}^\infty \jy^\beta |V(y)| \, \big| \whF^{-1}(g)(2y-x) \big| \,dy \right\|}_{L^2_x}
			\\
			\lesssim {\| \jx^\beta V \|}_{L^1} {\| g \|}_{L^2}.
		\end{align*}
		
		
	\end{proof}

	By duality we can obtain the following:
	\begin{lem}\label{lem:psehard2} 
		Suppose $\jx^{\gamma}V\in L^1$ with $\gamma>\max(\beta/2+3/4,\beta)$. 
		Then
		\begin{equation}\label{eq:mPDOhard+2}
			\left\Vert \int e^{i\lambda x}\partial_xm_+(x,\lambda)\mathrm{1}_{x\geq-1}h(x)\,dx\right\Vert_{L_{\lambda}^{2}}
			\lesssim\left\Vert \jx^{-\beta}h\right\Vert_{L^2}.
		\end{equation}
		Similar estimates hold for $\partial_xm_{-}$ restricted on $x\leq1$.
	\end{lem}

	A result similar to Lemmas \ref{lem:psehard1} 
	can be obtained for the operator with the symbol $\partial_{\lambda}\partial_xm_+(x,\lambda)$
	More precisely we have
	
	\begin{lem}\label{lem:psehardk1}
		Suppose $\jx^{\gamma}V\in L^1$ with $\gamma>\max(\beta/2+3/4,\beta)+1$.
		Then
		\begin{equation}\label{eq:psehard-1}
			\left\Vert \jx^\beta
			\mathrm{1}_{x\geq-1}\int e^{i\lambda x}\partial_{\lambda}\partial_xm_+(x,\lambda)g(\lambda)\,d\lambda\right\Vert_{L^2_x}
			\lesssim{\| g \|}_{L^2}
		\end{equation}
		Similar estimates hold for $\partial_{\lambda}\partial_xm_{-}$ restricted on $x\leq1$.
	\end{lem}
	
	\begin{proof}
		The proof of \eqref{eq:psehard-1} can be done in the same way as the proof of \eqref{eq:psehard}.
		We first use the identity
		\begin{align}\label{dldxm}
			\begin{split}
				\partial_{\lambda}\partial_xm_+(x,k) & 
				=\int_x^\infty e^{2i\lambda(y-x)}V(y)\partial_{\lambda}m_+(y,\lambda)\,dy\\
				& +\int_x^\infty 2i(y-x)e^{2i\lambda(y-x)}V(y)m_+(y,\lambda)\,dy,
			\end{split}
		\end{align}
		see \eqref{eq:split}. We can then estimate the first term on the right-hand side of \eqref{dldxm}
		as in Lemma \ref{lem:psehard1} replacing  $m_+-1$ with $\partial_\lambda m_+$ and using \eqref{Mestimates1}
		(which now requires one more weight).
		The last term in \eqref{dldxm} is estimated similarly, absorbing the extra weights thanks to the 
		restriction $\gamma>\max(3/2,\beta+1)$.
	\end{proof}

	To deal with the localized decay, we also need to estimate PDOs with $\partial_\lambda m_+$ as the symbol.
	
	\begin{lem}\label{lem:m_k}
		Suppose $\jx^{\gamma}V\in L^1$ with $\gamma \geq 2$.
		If $\gamma>\beta + 5/2$, then
		\begin{equation}\label{eq:pdomk}
			\left\Vert \mathrm{1}_{x\geq-1} \jx^\beta 
			\int e^{i\lambda x}\partial_{\lambda}m_+(x,\lambda)g(\lambda)\,d\lambda\right\Vert_{L^2_x}\lesssim{\| g \|}_{L^2}.
		\end{equation}
		Similar estimates hold for $\partial_{\lambda}m_{-}$ restricted onto $x\leq1$.
		
		By duality
		\begin{equation}\label{eq:pdomkd}
			\left\Vert \int e^{i\lambda x}\partial_{\lambda}m_+(x,\lambda)\mathrm{1}_{x\geq-1}h(x)\,dx\right\Vert_{L^2_\lambda}
			\lesssim\left\Vert \jx^{-\beta}h\right\Vert_{L^2_x}.
		\end{equation}
		Similar estimates hold for $\partial_{\lambda}m_{-}$ restricted onto $x\leq1$.
		
	\end{lem}

	\begin{proof}
		Using the estimate for the Jost function from Lemma \ref{lem:Mestimates} we have,
		for $x\geq-1$, 
		\begin{align}
			\left|\partial_{\lambda}m_+(x,\lambda)\right| \lesssim \frac{\mathcal{W}_{+}^{2}(x)}{\langle\lambda\rangle}
			\lesssim  \frac{1}{\langle\lambda\rangle} \jx^{-\gamma+2} {\| \jx^\gamma V\|}_{L^1}.
		\end{align}
		Then we apply Cauchy-Schwarz in $\lambda$ 
		to obtain that, for fixed $x\geq-1$,
		\[
		\left|\jx^\beta \int e^{i\lambda x}\partial_{\lambda}m_+(x,\lambda)g(\lambda)\,d\lambda\right|
		\lesssim \jx^{\beta-\gamma+2}{\| g \|}_{L^2}.
		\]
		and taking the $L^2_x$ norm gives \eqref{eq:pdomk}.
	\end{proof}

	Finally, we record an additional lemma which is useful for the localized $L^2$
	estimates for the derivative of the Schr\"odinger flow, see Lemma \ref{lem:localEn}.
	
	\begin{lem}\label{lem:PDOmxkk}
		For $\jx^{\gamma}V\in L^1$ with $\gamma > \beta + 5/2$, we have
		\begin{equation}\label{eq:mPDOhard+31}
			\left\Vert \jx^{\beta}\mathrm{1}_{x\geq-1}\int_{|\lambda|\lesssim1}e^{i\lambda x}
			\partial_{\lambda}^2\partial_xm_+(x,\lambda)g(\lambda)\,d\lambda\right\Vert_{L^2_x}
			\lesssim{\| g \|}_{L^2}
		\end{equation}
		Similar estimates hold for $\partial_\lambda^2\partial_xm_{-}$ restricted
		on $x\leq1$.
		
	\end{lem}
	
	\begin{proof}
		The proof is similar to those above. It suffices notice that, by Lemma \ref{lem:Mestimates}, for $x\geq-1$, 
		\[
		\left|\partial_{\lambda}^2\partial_xm_+(x,\lambda)\right|\lesssim\mathcal{W}_{+}^{2}(x)
		\lesssim \jx^{-\gamma+2} {\|\jy^{\gamma}V \|}_{L^1}.
		\]
		We can then apply Cauchy-Schwarz (here $|\lambda| \lesssim 1$) and obtain \eqref{eq:mPDOhard+31} 
		since $\gamma > \beta + 5/2$.
	\end{proof}

	\medskip
	\subsection{Distorted Fourier transform}\label{ssecDFT}
	We recall some basic properties of the distorted Fourier transform
	with respect to the perturbed Schr\"odinger operator.
	First, recall that the standard Fourier transform is defined, for $\phi\in L^2$, as
	\[
	\mathcal{F}\left[\phi\right](\lambda):=\hat{\phi}(\lambda)=\frac{1}{\sqrt{2\pi}}\int e^{-i\lambda x}\phi(x)\,dx
	\]
	with its inverse as
	\[
	\mathcal{F}^{-1}\left[\phi\right](x):=\frac{1}{\sqrt{2\pi}}\int e^{i\lambda x}\phi(\lambda)\,d\lambda.
	\]
	Given the Jost functions $\psi_{\pm}$ from \eqref{psipm}, we set
	\begin{align}\label{matK}
		\mathcal{K}(x,\lambda):=\frac{1}{\sqrt{2\pi}}
		\begin{cases}
			T(\lambda) \psi_+(x,\lambda) & \lambda\geq0
			\\
			T(-\lambda) \psi_-(x,-\lambda) & \lambda<0
		\end{cases},
	\end{align}
	and define the ``distorted Fourier transform'' for $f\in \mathcal{S}$ by
	\begin{align}\label{tildeF}
		\wtF\left[\phi\right](\lambda) = \widetilde{f}(\lambda) := \int \overline{\mathcal{K}(x,\lambda)} f(x)\,dx.
	\end{align}

	\begin{lem}\label{lemtildeF}
		In our setting, one has
		\[
		{\big\| \wtF\left[f\right] \big\|}_{L^{2}}=\left\Vert f\right\Vert _{L^{2}},\,\,\forall f\in L^{2}
		\]
		and
		\begin{align}\label{tildeF-1}
			\wtF^{-1}\left[\phi\right](x)=\int\mathcal{K}(x,\lambda)\phi(\lambda)\,d\lambda.
		\end{align}
		Also, if $D:=\sqrt{-\partial_{xx}+V}$, 
		\begin{align}
			m(D)=\wtF^{-1}m(\lambda)\wtF.
		\end{align}
		so that in particular
		$\left(-\partial_{xx}+V\right)=\wtF^{-1}\lambda^{2}\wtF$.
		
		\medskip
		We also have the following properties:
		
		\setlength{\leftmargini}{2em}
		\begin{itemize}
			
			\item[(i)] If $\phi\in L^1$, then $\widetilde{\phi}$ is a continuous, bounded function. 
			
			\medskip
			\item[(ii)] If the potential $V$ is generic $\tilde{\phi}\left(0\right)=0$.
			
			\medskip
			\item[(iii)] If the potential $V$ is non-generic with $\psi_+(x,0)$ even  (resp. odd),
			then $\tilde{\phi}\left(0\right)=0$ if $\phi$ is odd  (resp. even).
			
			\medskip
			\item[(iv)] There exists $C>0$ such that one has
			\begin{equation*}
				\left\Vert \lambda\widetilde{u}\right\Vert _{L^{2}}
				\leq C\left(1+\left(\left\Vert V\right\Vert_{L^1}\right)^{\frac{1}{2}}\right)\left\Vert u\right\Vert _{H^{1}},
			\end{equation*}
			and
			\begin{align}\label{eq:weiF}
				\left\Vert \partial_{\lambda}\widetilde{u}\right\Vert _{L^{2}}\leq C\left\Vert \jx u\right\Vert _{L^{2}}.
			\end{align}
		\end{itemize}
		
	\end{lem}
	
	\begin{proof}
		See for example Section 6 in \cite{Agmon}, 
		\cite{DS,Yaf}, 
		and \cite[Lemma 2.4]{GPR}.

	\end{proof}
	
	\begin{rem}[About the zero frequency]\label{rem0}
		From (ii) in Lemma \ref{lemtildeF} above 
		we see that for generic potentials the solution $u$ of \eqref{eq:NLSE}
		automatically satisfies $\widetilde{u}(0) = 0$. 
		From (iii)  we see that this condition still holds true if, for example, we 
		assume that the zero energy Jost function is even and we can ensure that the solution is odd (for all times).
		This latter property holds true if we assume that the potential $V$ is even and that $u(t=0)$ is odd,
		so that $u(t)$ will be an odd function for all times.
		
		This `momentum' condition may also be preserved under the flow if the equation
		has a suitable structure at low frequencies, without the need to impose symmetries.
		
		The vanishing behavior at zero frequencies plays an important role in obtaining improved local
		decay estimates in the next section.
	\end{rem}

	\bigskip
	\section{Linear estimates and improved local decay\label{sec:Linear}}
	
	In this section we prove various linear estimates for the Schr\"odinger operator. 
	Before establishing linear estimates for the perturbed operator $e^{itH}$
	we first analyze more in detail the generalized eigenfunctions. 

	\subsection{Decomposition of the generalized eigenfunctions}\label{subsec:DecK}
	In this subsection we decompose the generalized  eigenfunctions 
	defined by \eqref{mpm} 
	into two parts.
	A proper decomposition of the generalized eigenfunctions is needed in order to have a good
	understanding of the nonlinear spectral distribution.
	A different decomposition than the one we present below appeared in \cite{GPR}.
	There, the authors decompose the generalized eigenfunctions into three pieces:
	a singular part, a singular part with low frequency improved behavior, and a regular part;
	accordingly they decompose $\mu$ into as many pieces.
	Here, thanks to our simplified approach - in particular, the ``approximate commutation'' identity
	in Lemma \ref{lem:algetri} -
	we only need to decompose the eigenfunctions into two parts (a singular and a regular one)
	and will not need to worry about the exact behavior of the nonlinear terms when the input frequencies are close to zero.
	
	First, from the definition of $\mathcal{K}$ in \eqref{matK} and $m_\pm$ in \eqref{mpm} 
	we have 
	\begin{align}
		\begin{split}\label{matKk>00}
			\mbox{for $k>0$,} \qquad \sqrt{2\pi}\mathcal{K}(x,k) & = T(k)m_+(x,k)e^{ixk}
		\end{split}
		\\
		\begin{split}\label{matKk<00}
			\mbox{for $k<0$,} \qquad \sqrt{2\pi}\mathcal{K}(x,k) & = T(-k)m_-(x,-k)e^{ixk}.
		\end{split}
	\end{align}
	Note now the expression \eqref{matKk>00} is bounded for $x>0$ but unbounded when $x\rightarrow -\infty$
	(and viceversa for  \eqref{matKk<00}).
	
	Let $\Phi$ be a smooth, non-negative function, which is one in a neighborhood of $0$, vanishes outside
	of $[-2,2]$, and such that $\int\Phi\,dx=1$. We define $\chi_+$ and $\chi_-$ by
	\begin{align}\label{chipm}
		\chi_+(x) = \int_{-\infty}^{x}\Phi(y)\,dy,
		\qquad \text{and} \qquad  \chi_+(x)+\chi_-(x)=1.
	\end{align}
	Using $\chi_{\pm}(x)$, the definition of $\mathcal{K}(x,k)$ in \eqref{matK},
	and the properties \eqref{eq:f_+-1}-\eqref{eq:f_--1} we can write
	\begin{align*}
		\begin{split}
			\mbox{for $k>0$,} \qquad \sqrt{2\pi}\mathcal{K}(x,k) & = \chi_+(x)T(k)\psi_+(x,k)
			+\chi_-(x)\left[\psi_-(x,-k) + R_-(k) \psi_-(x,k) \right],
		\end{split}
		\\
		\begin{split}
			\mbox{for $k<0$,} \qquad \sqrt{2\pi}\mathcal{K}(x,k) & =\chi_-(x)T(-k) \psi_-(x,-k)
			+\chi_+(x)\left[ \psi_+(x,k) + R_+(-k)\psi_+(x,-k) \right].
		\end{split}
	\end{align*}
	Then, 
	with the definition of $m_{\pm}(x,k)$ in \eqref{mpm}, we can write:
	\begin{align}
		\begin{split}\label{matKk>0}
			\mbox{for $k>0$,} \qquad \sqrt{2\pi}\mathcal{K}(x,k) & =\chi_+(x)T(k)m_+(x,k)e^{ixk}
			\\
			& +\chi_-(x)\left[m_-(x,-k)e^{ikx}+R_-(k)m_-(x,k)e^{-ikx}\right],
		\end{split}
		\\
		\begin{split}\label{matKk<0}
			\mbox{for $k<0$,} \qquad \sqrt{2\pi}\mathcal{K}(x,k) & =\chi_-(x)T(-k)m_-(x,-k)e^{ixk}\\
			& +\chi_+(x)\left[m_+(x,k)e^{ikx}+R_+(-k)m_+(x,-k)e^{-ikx}\right].
		\end{split}
	\end{align}
	We decompose in a ``singular'' and ``regular'' part
	\begin{equation}\label{eq:decomK}
		\sqrt{2\pi}\mathcal{K}(x,k) = \mathcal{K}_{S}(x,k) 
		+ \mathcal{K}_{R}(x,k)
	\end{equation}
	where the singular part is
	\begin{align}\label{Ksing}
		\mathcal{K}_{S}(x,k) = \chi_+(x)\mathcal{K}_+(x,k) + \chi_-(x)\mathcal{K}_-(x,k)
	\end{align}
	with
	\begin{align}\label{Ksing+}
		\mathcal{K}_+(x,k) & = 
		\begin{cases}
			T(k) e^{ikx} & k\geq0
			\\
			e^{ikx} + R_+(-k)e^{-ikx}  & k<0
		\end{cases},
		\\
		\label{Ksing-}
		\mathcal{K}_-(x,k) & = 
		\begin{cases}
			e^{ikx} + R_-(k) e^{-ikx} & k\geq0
			\\
			T(-k) e^{ikx} & k<0
		\end{cases},
	\end{align}
	and the regular part is
	\begin{align}\label{KR}
		\mathcal{K}_{R}(x,k) :=
		\begin{cases}
			\chi_+(x)T(k)\left(m_+(x,k)-1\right)e^{ixk}
			\\
			+ \chi_-(x)\left[\left(m_-(x,-k)-1\right)e^{ixk}
			+ R_-(k) \left(m_-(x,k)-1\right)e^{-ixk}\right] & k\geq0,
			\\
			\\
			\chi_-(x)T(-k)\left(m_-(x,-k)-1\right)e^{ixk}
			\\
			+\chi_+(x)\left[\left(m_+(x,k)-1\right)e^{ixk}+R_+(-k)\left(m_+(x,-k)-1\right)e^{-ixk}\right] 
			& k<0.
		\end{cases}
	\end{align}
	Note now the singular part $\mathcal{K}_S$ is a just combination of exponentials with coefficients that
	are smooth except at $k = 0$.
	The regular part $\mathcal{K}_R$ instead enjoys localization in $x$ in view of the estimates
	for the Jost functions in Lemma \ref{lem:Mestimates}.
	In particular, using also the estimates for the $T,R$ coefficients \eqref{estiTR}, we have:
	
	\begin{lem}\label{lem:Kregular}
		Let $\mathcal{K}_R$ be defined as in \eqref{KR}, and recall the definition \eqref{defW_+-}: then
		\begin{align}
			\label{KRdxi}
			& \left|\partial_{k}^{s}\mathcal{K}_{R}(x,k)\right|\lesssim\frac{1}{\jk}\mathcal{W}_{\pm}^{s+1}(x),
			\\
			& \label{KRdxidx}
			\left|\partial_x\partial_{k}^{s}\mathcal{K}_{R}(x,k)\right|\lesssim\mathcal{W}_{\pm}^{s+1}(x).
		\end{align}
	\end{lem}
	According to \eqref{eq:decomK}, given $\phi\in \mathcal{S}$ we can write
	\begin{align}\label{decompphi}
		\begin{split}
			\phi = \phi_S+\phi_R, \qquad \phi_\ast(x) :=\frac{1}{\sqrt{2\pi} }\int \mathcal{K}_\ast(x,k) \tilde{\phi}(k) \, dk.
		\end{split}
	\end{align}
	We can then extend this definitions to $L^2$ as in Lemma \ref{lemtildeF}.
	We also set 
	\begin{align}\label{decompphinot}
		\phi_0 := \phi, \qquad \mathcal{K}_0:=\sqrt{2\pi} \mathcal{K}.
	\end{align}
	For use in later estimates it is also useful to rewrite \eqref{Ksing+}-\eqref{Ksing-} as 
	a linear combination of exponentials and coefficients as follows:
	with the notation $\mathrm{1}_{+} = \mathrm{1}_{[0,\infty)}$ and $\mathrm{1}_{-} = \mathrm{1}_{(-\infty,0)}$
	for the indicator functions, we have
	\begin{align}\label{Ksing+'}
		\begin{split}
			& \mathcal{K}_+(x,k) = a_+^+(k) e^{ikx} + a_+^-(k) e^{-ikx},
			\\
			& \mbox{with} \qquad a^+_+(k) := \mathrm{1}_{+}(k)T(k) + \mathrm{1}_{-}(k), \qquad a_+^-(k) := \mathrm{1}_{-}(k)R_+(-k),
		\end{split}
	\end{align}
	and
	\begin{align}\label{Ksing-'}
		\begin{split}
			& \mathcal{K}_-(x,k) = a_-^+(k) e^{ikx} + a_-^-(k) e^{-ikx},
			\\
			& \mbox{with} \qquad a_-^+(k) = \mathrm{1}_{+}(k) + \mathrm{1}_{-}(k)T(-k), \qquad a^-_-(k) = \mathrm{1}_{+}(k)R_-(k).
		\end{split}
	\end{align}
	One of the convenient features of these expressions is given below by  Remark \ref{remcoeff}.

	\medskip
	\subsection{Dispersive decay}
	First of all, let us recall the linear dispersive estimate for the free flow:
	
	\begin{lem}\label{lem:linearest}
		The linear free Schr\"odinger flow has the following
		dispersive estimate: for $t\geq0$
		\begin{equation}
			\left(e^{-it\partial_{xx}}h\right)(x)
			= \frac{1}{\left(-2it\right)^{\frac{1}{2}}}e^{i\frac{\left|x\right|^{2}}{4t}}\hat{h}\left(-\frac{x}{2t}\right)
			+\frac{1}{t^{\frac{1}{2}+b}}\mathcal{O}\left(\left\Vert h\right\Vert _{H^{0,c}}\right)\label{eq:linearasy-1}
		\end{equation}
		for $x\in\mathbb{R}$ and $c \geq \frac{1}{2}+2b$. 
		As a consequence, for $t\geq0$
		\begin{equation}\label{eq:linearpoinwise0}
			{\big\| e^{-it\partial_{xx}}h \big\|}_{L^\infty_x}
			\lesssim \frac{1}{\sqrt{t}} {\big\| \hat{h} \big\|}_{L^\infty_k}
			+\frac{1}{t^{\frac{3}{4}}} {\big\| \partial_{k}\hat{h} \big\|}_{L^2_k}.
		\end{equation}
	\end{lem}
	
	The above free dispersive estimate can be extended to the perturbed
	flow after projecting onto the continuous spectrum, see Goldberg-Schlag
	\cite{GSch} and Germain-Pusateri-Rousset \cite{GPR}. 
	In particular, they can be extended to the case of 
	\[H:=-\partial_{xx}+V\]
	under the assumption that there are no bound states.
	
	\begin{lem}\label{lem:pointwiseH}
		Suppose $\jx^\gamma V(x)\in L^1$ with $\gamma \geq 1$. The perturbed Schr\"odinger flow has the following dispersive estimate:
		for $t\geq0$
		\begin{equation}\label{eq:linearpoinwiseH}
			{\big\| e^{itH}h \big\|}_{L^\infty_x} \lesssim\frac{1}{\sqrt{t}}{\big\| \tilde{h} \big\|}_{L^\infty_k}
			+\frac{1}{t^{\frac{3}{4}}}{\big\| \partial_k\tilde{h} \big\|} _{L^2_k}.
		\end{equation}
	\end{lem}
	We recall a Lemma from \cite{GPR}. 
	\begin{lem}
		\label{lemstat}
		Consider a function $a(x,k)$ defined on $I \times \mathbb{R}_{+}$ and    such that
		\begin{equation}
			\label{hypest}
			| a(x, k)| +|k| |\partial_{k}a(x,k)| \lesssim 1, \quad \forall x \in I, \, \forall k \in \mathbb{R}_{+}
		\end{equation}
		and for every  $X \in \mathbb{R}$, consider the oscillatory  integral
		$$
		I(t,X,x)=\int_{0}^{+\infty} e^{it(k-X)^2} a(x,k) \widetilde h(k) \, dk, \quad t>0, \, x \in I.
		$$
		Then, we have the  estimate
		\begin{equation}
			\label{decaylem} | I(t,X,x)|  \lesssim \frac{1}{\sqrt t} \| \widetilde{h}(t) \|_{L^\infty} + \frac{1}{t^{\frac{3}{4}}} \| \partial_k \widetilde{h} \|_{L^2}
		\end{equation}
		which is uniform in $X\in \mathbb{R}$, $t >0$ and $x\in I$.
	\end{lem}

	Recall the definition of the regular part from \eqref{decompphi} and \eqref{eq:decomK}.
	We have the following $L^\infty$ decay for the regular part of  the Schr\"odinger flow as a corollary 
	of Lemmas \ref{lem:pointwiseH} and \ref{lemstat}.
	
	\begin{cor}\label{cor:regularLinfty}
		If $\jx^\gamma V(x)\in L^1$ with $\gamma\geq\beta+1$, then
		\begin{align}\label{regularinfty}
			\left\Vert \jx^\beta \big( e^{iHt} h \big)_R \right\Vert_{L^{\infty}_x}
			= \left\Vert \jx^\beta \int\mathcal{K}_R(x,k)e^{ik^{2}t}\tilde{h}(k)\,dk\right\Vert_{L^{\infty}_x}
			\lesssim\frac{1}{\sqrt{t}}{\big\| \tilde{h} \big\|}_{L^\infty_k}
			+\frac{1}{t^{\frac{3}{4}}}{\big\| \partial_k\tilde{h} \big\|} _{L^2_k}.
		\end{align}
	\end{cor}
	\begin{proof}
		We only  analyze the case $k\geq0$ since the case $k<0$ is similar.
		From the definition of the regular part see \eqref{eq:decomK} and \eqref{decompphi},
		\begin{align}\label{locdecpr12ng0}
			\begin{split}
				& \int_{k\geq0} \mathcal{K}_R(x,k)e^{ik^2 t} \tilde{h}(k)\,dk = J_+(x) + J_-(x),
				\\
				& J_+(x) := 
				\chi_+(x)\int_{k\geq0}e^{ik^2t} T(k) \big( m_+(x,k) - 1) e^{ikx} \tilde{h}(k)\,dk,
				\\
				& J_-(x) := \chi_-(x)\int_{k\geq0} e^{ik^2t} 
				\big((m_-(x,-k) - 1)e^{ikx} + R_-(k)(m_-(x,k)-1)e^{-ikx}\big) \tilde{h}(k)\,dk.
			\end{split}
		\end{align}
		From Lemma \ref{lem:Mestimates} and Lemma \ref{estiTR} we know that 
		\begin{align}
			& \label{m+}
			\jx^\beta|m_\pm(x,k) - 1 | \lesssim \frac{1}{\jk}\jx^\beta\mathcal{W}^1_\pm, \quad x \geq \mp 1,
			\\
			& \label{m+2}
			\jx^\beta| \partial_{k}m_\pm(x,k)| 
			\lesssim \frac{1}{|k|}\jx^\beta\mathcal{W}^{1}_\pm, \quad x \geq  \mp 1,
			\\
			& \label{Tk2}  | \partial_{k} T(k) | + | \partial_{k}R_{+} (k)|  \lesssim \frac{1}{\jk}.
		\end{align}
		We then write $\jx^\beta J_{+}= e^{ -i {x^2  \over 4t}} I_{+} (t,X,x)$ with
		$$I_{+}(t,X,x)= \int_{0}^{+\infty} e^{it (k-X)^2} T(k) \jx^\beta(m_+(x,k)-1) \widetilde h(k)\, dk .$$ where $X=  - {x \over 2 t}$.
		Thanks to \eqref{m+}-\eqref{Tk2}, we can thus use Lemma \ref{lemstat} with $x \in I= \mathbb{R}_{+}$,
		and  $a(x,k) = T(k)\jx^\beta( m_+(x,k)-1)$.
		This yields
		$$ |J_{+}| \lesssim   \frac{1}{\sqrt t} \| \widetilde{h}(t) \|_{L^\infty} 
		+ \frac{1}{t^{\frac{3}{4}}} \| \partial_k \widetilde{h} \|_{L^2}$$
		as desired. The term $J_-$ can be estimated in the same way.
	\end{proof}

	\subsection{Improved Local $L^\infty$ decay}
	When $\hat{h}(k)$ vanishes at $k=0$, one should expect (locally in space)
	better decay rates than the one in \eqref{eq:linearpoinwise0} since, 
	intuitively, this eliminates $0$ velocity particles. 
	Indeed, we have the following:
	
	\begin{lem}\label{lem:freeimp}
		Suppose $\hat{h}\in H^{1}$ and $\hat{h}\left(0\right)=0$,
		then
		\begin{align}\label{freeimpconc}
			\left\Vert \left\langle x\right\rangle^{-1}
			e^{-it\partial_{xx}}h\right\Vert _{L^\infty_x}
			=\left\Vert \jx^{-1}\int e^{ixk+itk^{2}}\hat{h}(k)\,dk\right\Vert _{L^\infty_x}
			\lesssim \frac{1}{|t|^{\frac{3}{4}}} {\big\| \hat{h} \big\|}_{H^{1}}.
		\end{align}
	\end{lem}
	
	\begin{proof}
		It suffices to consider $t\geq 0$. We split the integral into two pieces
		\begin{align*}
			\int_{-\infty}^{\infty}e^{ixk}e^{itk^{2}}\hat{h}\,dk 
			& = \int_{\left|k\right|\leq\frac{1}{\sqrt{t}}}e^{ixk}e^{itk^{2}}\hat{h}\,dk
			+ \int_{\left|k\right|>\frac{1}{\sqrt{t}}}e^{ixk}e^{itk^{2}}\hat{h}\,dk
			=\text{I}+\text{II}.
		\end{align*}
		Since $\hat{h}(0)=0$, using the fundamental theorem of calculus and H\"older's inequality
		\begin{align}\label{h0}
			\big| \hat{h}(k) \big| = \Big| \int_0^k \partial_{k'}\hat{h}(k')\,dk' \Big|
			\lesssim |k|^{\frac{1}{2}} {\big\| \partial_k \hat{h} \big\|}_{L^2},
		\end{align}
		Then
		\begin{align*}
			|\text{I}| \lesssim 
			\int_{|k|\leq \frac{1}{\sqrt{t}}} |k|^\frac{1}{2}  {\big\| \hat{h} \big\|}_{H^{1}} \, dk
			\lesssim \frac{1}{t^\frac{3}{4}} {\big\| \hat{h} \big\|}_{H^{1}}. 
		\end{align*}
		
		
		Next we consider $\text{II}$ and the case $k<-t^{-\frac{1}{2}}$; the case $k>t^{-\frac{1}{2}}$ is similar. 
		Integrating by parts 
		one has 
		\begin{align*}
			\int_{k<-\frac{1}{\sqrt{t}}} e^{ixk}e^{itk^{2}}\hat{h}\,dk & 
			= \int_{k<-\frac{1}{\sqrt{t}}} \frac{1}{2itk} e^{ixk}\hat{h}(k)\, de^{itk^{2}}
			\\
			& = - \frac{1}{2i\sqrt{t}}\hat{h}\left(-\frac{1}{\sqrt{t}}\right)e^{i}e^{-\frac{1}{\sqrt{t}}ix}
			- \int_{k<-\frac{1}{\sqrt{t}}} e^{itk^{2}}\partial_{k}\left(\frac{1}{2itk} e^{ixk}\hat{h} \right)
			= \
			\text{II}_1 + \text{II}_2.
		\end{align*}
		The term $\text{II}_1$ is handled directly using \eqref{h0}
		For the term $\text{II}_2$
		we can bound the contribution when the differentiation hits $e^{ixk}\hat{h}$ as follows:
		\begin{align}\label{freeimp1}
			\int_{-\infty}^{-\frac{1}{\sqrt{t}}} \left|\frac{1}{k} \partial_{k}\left(e^{ixk}\hat{h}\right) \right|\,dk 
			& \lesssim \langle x \rangle \left(\int_{-\infty}^{-\frac{1}{\sqrt{t}}} 
			\frac{1}{k^{2}}\right)^{\frac{1}{2}} {\big\| \hat{h} \big\|}_{H^{1}}
			\lesssim \langle x \rangle \, t^{\frac{1}{4}} {\big\| \hat{h} \big\|}_{H^{1}}.
		\end{align}
		When the differentiation hits $\frac{1}{k}$, using \eqref{h0} one has
		\begin{align}\label{freeimp2}
			\int_{-\infty}^{-\frac{1}{\sqrt{t}}}\left|\frac{1}{k^{2}}e^{ixk}\hat{h}\right|\,dk 
			& \lesssim {\big\| \hat{h} \big\|}_{H^{1}} \int_{-\infty}^{-\frac{1}{\sqrt{t}}}\left|k\right|^{-\frac{3}{2}}dk
			\lesssim t^{\frac{1}{4}} {\big\| \hat{h} \big\|}_{H^{1}}.
		\end{align}
		Putting together \eqref{freeimp1} and \eqref{freeimp2} we obtain \eqref{freeimpconc}.
	\end{proof}
	
	\begin{rem}\label{remcoeff}
		Given the coefficients $a^\epsilon_\mu$, $\epsilon,\mu \in \{+,-\}$ from \eqref{Ksing+'}-\eqref{Ksing-'} we have
		\begin{align}\label{remcoeffest1}
			{\big\| a^\epsilon_\mu \tilde{f} \big\|}_{L^\infty_k} \lesssim {\big\| \tilde{f} \big\|}_{L^\infty_k},
			\qquad {\big\| a^\epsilon_\mu \tilde{f} \big\|}_{L^2_k} \lesssim {\big\| \tilde{f} \big\|}_{L^2_k}
			\qquad {\big\| a^\epsilon_\mu \tilde{f} \big\|}_{H^1_k} \lesssim {\big\| \tilde{f} \big\|}_{H^1_k}.
		\end{align}
		For the third estimate above we use that $\tilde{f}(0)=0$ and Lemma \ref{estiTR}.
		In particular we have that $a^\epsilon_\mu \tilde{f}$ satisfies the same a priori bounds satisfied by $\tilde{f}$,
		and therefore will enjoy the same decay and local decay estimates.
		More precisely, from \eqref{eq:linearpoinwise0} and \eqref{remcoeffest1} we see that, for $t \geq 0$,
		\begin{equation}\label{remcoeffest2}
			{\big\| e^{-it\partial_{xx}} \widehat{\mathcal{F}}^{-1}(a^\epsilon_\mu \tilde{f}) \big\|}_{L^\infty} 
			\lesssim\frac{1}{\sqrt{t}}{\big\| \tilde{f} \big\|}_{L^\infty}
			+ \frac{1}{t^{\frac{3}{4}}}{\big\| \tilde{f} \big\|}_{H^1},
		\end{equation}
		and, similarly from \eqref{freeimpconc} (recall $\tilde{f}(0)=0$) and \eqref{remcoeffest1}, we have
		\begin{align}\label{remcoeffest3}
			{\big\| \jx^{-1} e^{-it\partial_{xx}} \widehat{\mathcal{F}}^{-1}(a^\epsilon_\mu \tilde{f}) \big\|}_{L^\infty} 
			\lesssim \frac{1}{t^{\frac{3}{4}}}{\big\| \tilde{f} \big\|}_{H^1}.
		\end{align}
	\end{rem}
	
	\medskip

	Next, we consider the perturbed flow $e^{iHt} h$. 
	We do not assume that the potential is generic ($0$ could be a resonance)
	but only the condition $\tilde{h}(0)=0$. See Remark \ref{rem0}.
	
	
	\begin{lem}[Improved $L^\infty$ local decay]\label{lemlocdecinftyng}
		Suppose $\jx^\gamma V(x)\in L^1$ with $\gamma \geq 2$. Assume that $\tilde{h}(0)=0$.
		With the notation \eqref{decompphi}-\eqref{decompphinot}, one has
		\begin{align}\label{locdecinftyng}
			\left\Vert \jx^{-1} \big( e^{iHt} h \big)_A \right\Vert_{L^{\infty}_x}
			= \left\Vert \jx^{-1} \int\mathcal{K}_A(x,k)e^{ik^{2}t}\tilde{h}(k)\,dk\right\Vert_{L^{\infty}_x}
			\lesssim |t|^{-\frac{3}{4}}{\big\| \tilde{h} \big\|}_{H^1_k}, \qquad A \in \{0,S\},
		\end{align}
		and, for $\gamma\geq \beta+2$,
		\begin{align}\label{locdecinftyRng}
			\left\Vert \jx^\beta \big( e^{iHt} h \big)_R \right\Vert_{L^{\infty}_x}
			= \left\Vert \jx^\beta \int\mathcal{K}_R(x,k)e^{ik^{2}t}\tilde{h}(k)\,dk\right\Vert_{L^{\infty}_x}
			\lesssim |t|^{-\frac{3}{4}}{\big\| \tilde{h} \big\|}_{H^1_k}. 
		\end{align}
		As a consequence, for $\gamma \geq 2$ and $\zeta \in [0,1]$ we have
		\begin{align}\label{eq:interp}
			\left\Vert \jx^{-\zeta}\left(e^{-itH}h\right)_{A}\right\Vert_{L^{\infty}}
			\lesssim t^{-\left(\frac{1}{2}+\frac{\zeta}{4}\right)}{\big\| \tilde{h} \big\|}_{H^1_k},
			\qquad A \in \{0,S,R\}.
		\end{align}
	\end{lem}
	
	
	\begin{rem}
		Under the assumptions that the potential $V$ is generic,
		(and with slightly stronger decay assumptions), 
		one can get faster decay in the $L^\infty$ norm at the rate $|t|^{-1}$.
		See Lemma \ref{lemlocdecinfty}.We also refer Egorova-Kopylova-Marchenko-Teschl \cite{EKMT}, Goldberg \cite{Gold}, Krieger-Schlag \cite{KS}, Schlag \cite{Sch} where the authors prove localized improved $L^1\rightarrow L^\infty$ decay. 
	\end{rem}

	\begin{proof}[Proof of Lemma \ref{lemlocdecinftyng}]
		Recall the explicit formula for the flow:
		\begin{align}\label{formula}
			\Big( e^{it\left(-\partial_{xx}+V\right)}h \Big)_A
			= 
			\int\mathcal{K}_A(x,k)e^{ik^{2}t} \tilde{h}(k)\,dk, \qquad A \in \{0,S,R\}.
		\end{align}
		Let us first look at the case $A=0$, that is, we consider the full evolution $e^{it\left(-\partial_{xx}+V\right)}h$.
		It suffices to prove the desired estimate for the contribution with $k\geq0$;
		for the case with $k<0$ the analysis is similar.
		Using the definition of $\mathcal{K}$, see \eqref{matKk>0}, we have
		\begin{align}\label{locdecpr11ng} 
			\begin{split}
				& \sqrt{2\pi}\int_{k\geq0} \mathcal{K}(x,k)e^{ik^2 t} \tilde{h}(k)\,dk = I_+(x) + I_-(x),
				\\
				& I_+(x) := 
				\chi_+(x)\int_{k\geq0}e^{ik^2t} T(k) m_+(x,k)e^{ikx} \tilde{h}(k)\,dk,
				\\
				& I_-(x) := \chi_-(x)\int_{k\geq0} e^{ik^2t} \big(m_-(x,-k)e^{ikx} + R_-(k)m_-(x,k)e^{-ikx}\big) \tilde{h}(k)\,dk,
			\end{split}
		\end{align}
		
		Let us first look at $I_+$.
		Since $\tilde{h}\left(0\right)=0$, we can apply Lemma \ref{lem:freeimp} 
		with the expression $\chi_+(x)T(k)m_+(x,k)\tilde{h}(k)$ playing the role of $\hat{h}$, see \eqref{freeimpconc}.
		It then suffices to bound, uniformly in $x$,
		\[
		{\big\| \chi_+(x)T(k)m_+(x,k)\tilde{h}(k) \big\|}_{H^1_k} \lesssim {\big\| \tilde{h} \big\|}_{H^1_k}.
		\]
		This follows from the estimates for the Jost functions in Lemma \ref{lem:Mestimates},
		and the bounds on transmission coefficient in Lemma \ref{estiTR}.
		For the integral $I_-$ we can similarly apply Lemma \ref{lem:freeimp} to 
		$\chi_-(x)m_-(x,-k)\tilde{h}(k)$ and $\chi_-(x)R_-(k)m_-(x,k)\tilde{h}(k)$ in place of $\hat{h}$.
		
		The bound \eqref{locdecinftyng} for $A=S$ can be obtained in the same way (it is in fact easier since there
		is no $m_+$ to deal with). It can also be obtained by difference after we establish \eqref{locdecinftyRng},
		to which we now turn.
		
		We write
		\begin{align}\label{locdecpr12ng} 
			\begin{split}
				& \int_{k\geq0} \mathcal{K}_R(x,k)e^{ik^2 t} \tilde{h}(k)\,dk = J_+(x) + J_-(x),
			\end{split}
		\end{align}
		as in \eqref{locdecpr12ng0}.
		For $J_+$ we can use the same argument used for \eqref{locdecpr11ng} above 
		with the symbol $\jx^{\beta+1} (m_+(x,k)-1)$ instead of $m_+(x,k)$:
		from Lemma \ref{lem:Mestimates} and Lemma \ref{estiTR} we have
		\[
		{\big\| \chi_+(x)\jx^{\beta+1} T(k)(m_+(x,k)-1)\tilde{h}(k) \big\|}_{H^1_k} \lesssim {\big\| \tilde{h} \big\|}_{H^1_k}.
		\]
		so that applying \eqref{freeimpconc} to $J_+$ gives a bound consistent with \eqref{locdecinftyRng}. 
		The term $J_-$ can be treated in the same way.
		
		\eqref{eq:interp} follows by interpolating between the standard estimate \eqref{eq:linearpoinwiseH}
		and the localized decay estimate \eqref{locdecinftyng}.
	\end{proof}


	\medskip
	\subsection{Local smoothing estimates}
	First of all, with the definition of the regular part, see \eqref{eq:decomK} and \eqref{decompphi}, 
	applying the PDOs estimate from Lemma \ref{lem:pseeasy} 
	we immediately obtain the following $L^2$ estimates for the regular part of the Schr\"odinger flow.
	
	\begin{lem}\label{lem:regularL2}
		Suppose $\jx^\gamma V(x)\in L^1$ for some $\gamma \geq 3/2 + \beta$, with $\beta \geq 0$.
		Then
		\begin{align}\label{regularL2}
			\left\Vert \jx^\beta \big( e^{iHt} h \big)_R \right\Vert_{L^{2}_x}
			\lesssim {\big\| \tilde{h} \big\|}_{L^2_k}.
		\end{align}
	\end{lem}
	
	\begin{proof}
		One just needs to apply \eqref{eq:mPDO+} from Lemma \ref{lem:pseeasy}
		to the definition \eqref{decompphi} and \eqref{KR}. 
		Note that we can replace $e^{i\lambda x}$ with $e^{-i\lambda x}$ in \eqref{eq:mPDO+} by taking complex conjugates.
	\end{proof}

	Next we will show some improved $L^2$ local/smoothing decays estimate for the (derivative) of the linear evolution.
	We will present separately the generic and non-generic cases for which we have different time-decay rates.
	
	\medskip
	\subsubsection{The non-generic case} 
	First of all, we do not assume that the potential is generic but impose that $\tilde{h}=0$. 
	In this case we get a decay rate which is similar to that in Lemma \ref{lemlocdecinftyng}.
	
	\begin{lem}\label{lem:locl2ng}
		Suppose $\jx^{\gamma}V(x)\in L^1$ with $\gamma\geq2$ and $\tilde{h}\left(0\right)=0$.
		Then, for $t\geq0$,
		\begin{equation}\label{eq:locl2ngf}
			\left\Vert \jx^{-1}t\partial_x\left(e^{iHt}h\right)_A\right\Vert _{L^{2}}
			\lesssim t^{\frac{1}{4}} {\big\| \tilde{h} \big\|}_{H^1}\qquad A \in \{0,S\},
		\end{equation}
		and for $\gamma>\beta+5/2$, $\beta \geq 0$,
		\begin{equation}\label{eq:loc2ngR}
			\left\Vert \jx^{\beta}t\partial_x\left(e^{iHt}h\right)_{R}\right\Vert_{L^{2}}
			\lesssim t^{\frac{1}{4}} {\big\| \tilde{h} \big\|}_{H^1}.
		\end{equation}
	\end{lem}

	\begin{proof}[Proof of \eqref{eq:locl2ngf}]
		Our starting point is the same as the proof of Lemma \ref{lemlocdecinftyng}.
		We look at the formula \eqref{formula} and analyze the case $A=0$; the proof in the case $A=S$ is simpler.
		As before, it suffices to check the contribution for $k\geq0$.
		We split the relevant integral as the sum of $I_+$ and $I_-$ as in \eqref{locdecpr11ng}, 
		and look first at the contributions coming from $t \partial_x I_+$, that is
		\begin{align}
			\label{ng1}
			t\partial_x I_+(x) & = t \chi_+(x)\int_{k\geq0} ikT(k)e^{ikx}m_+(x,k)e^{ik^{2}t}\tilde{h}(k)\,dk
			\\
			\label{ng2}
			& + t \chi_+(x)\int_{k\geq0}T(k)e^{ikx}\partial_x m_+(x,k)e^{ik^{2}t}\tilde{h}(k)\,dk
			\\
			\label{ng3}
			& + t \, \partial_x\chi_+(x)\int_{k\geq0} T(k)e^{ikx}m_+(x,k)e^{ik^{2}t}\tilde{h}(k)\,dk.
		\end{align}
		
		For \eqref{ng1} we integrate by parts using $\tilde{h}\left(0\right)=0$:
		\begin{align}
			\nonumber
			2t \int_{k\geq0} ikT(k)e^{ikx}m_+(x,k)e^{ik^{2}t}\tilde{h}(k)\,dk & 
			= \int_{k\geq0}T(k)e^{ikx}m_+(x,k)\tilde{h}(k)\,de^{ik^{2}t}
			\\
			\label{ng11}
			& =-\int_{k\geq0}e^{ik^{2}t}ixe^{ikx}T(k)m_+(x,k)\tilde{h}(k)\,dk
			\\
			\label{ng12}
			& -\int_{k\geq0}e^{ik^{2}t}e^{ikx}T(k)\partial_{k}m_+(x,k)\tilde{h}(k)\,dk
			\\
			\label{ng13}
			& -\int_{k\geq0}e^{ik^{2}t}e^{ikx}m_+(x,k)\partial_{k}\left(T(k)\tilde{h}(k)\right)\,dk.
		\end{align}
		We can estimate the first and last terms above by applying directly \eqref{eq:pseeasy} 
		and using $|T| + |\partial_kT| \lesssim 1$ (see Lemma \ref{estiTR}):
		\begin{equation}\label{eq:locall2ng1}
			\left\Vert \jx^{-1}\chi_+(x)
			\int_{k\geq0}e^{ik^{2}t}ixe^{ikx}T(k)m_+(x,k)\tilde{h}(k)\,dk\right\Vert_{L^{2}}
			\lesssim {\big\| \tilde{h} \big\|}_{L^2},
		\end{equation}
		and
		\begin{equation}\label{eq:locl2ng3}
			\left\Vert 
			\chi_+(x)\int_{k\geq0}e^{ik^{2}t}e^{ikx}m_+(x,k)
			\partial_{k}\left(T(k)\tilde{h}(k)\right)\,dk\right\Vert _{L^{2}}
			\lesssim {\big\| \tilde{h} \big\|}_{H^1}.
		\end{equation}
		For the remaining term we use the $L^2$ estimate \eqref{eq:pdomk} from Lemma \ref{lem:m_k} to conclude that
		\begin{equation}\label{eq:locl2ng2}
			\left\Vert \jx^{-1}\chi_+(x)\int_{k\geq0}e^{ik^{2}t}e^{ikx}T(k)
			\partial_{k}m_+(x,k)\tilde{h}(k)\,dk\right\Vert_{L^{2}} 
			\lesssim{\big\| \tilde{h} \big\|}_{L^2}
		\end{equation}

		For the term in \eqref{ng2}, using the localized pointwise
		decay with $\tilde{h}\left(0\right)=0$, see Lemma \ref{lem:freeimp}, one has that
		\begin{equation}\label{eq:locl2ng4}
			\left| \chi_+(x) \, t \,\int_{k\geq0}T(k)e^{ikx}\partial_xm_+(x,k)e^{ik^{2}t}\tilde{h}(k)\,dk\right|
			\lesssim\left|t\right|\left|t\right|^{-\frac{3}{4}}\sup_{x\geq-1}
			\left\Vert\jx\partial_xm_+(x,k)\tilde{h}(k)\right\Vert_{H_{k}^{1}}.
		\end{equation}
		Note that for $j=0,\,1$, by Lemma \ref{lem:Mestimates}, we have, for $x\geq-1$,
		\[
		\left|\jx\chi_+(x)\partial_{k}^{j}\partial_xm_+(x,k)\right|\lesssim\jx\int_x^\infty 
		\jy V(y)\,dy\lesssim\jx^{-\gamma+2}.
		\]
		Hence, for any $\alpha \geq 0$ and $\gamma > 3/2 + \alpha$, we have
		\begin{equation}\label{eq:loc2ng5}
			\left\Vert \jx^{-1+\alpha}\chi_+(x)\int_{k\geq0}tiT(k)e^{ikx}\partial_xm_+(x,k)e^{ik^{2}t}\tilde{h}(k)\,dk\right\Vert_{L^{2}}
			\lesssim\left|t\right|^{\frac{1}{4}} {\big\| \tilde{h} \big\|}_{H^1},
		\end{equation}
		which, for $\alpha =0$, is the bound we want.
		The term \eqref{ng3} can also be estimated directly using Lemma \ref{lem:freeimp}.
		We then have obtained the desired bound for for $I_+$
		
		The estimate for $\jx^{-1}t\partial_xI_-$, see the definition in \eqref{locdecpr11ng}, can be done in the same way
		using the properties of $m_{-}(x,k)$ for $x \leq 1$, and $|R_{-}|+|\partial_kR_-|\lesssim 1$.
		

		\medskip
		\noindent
		{\it Proof of \eqref{eq:loc2ngR}}.
		To prove the estimate on the regular part we begin as in \eqref{locdecpr12ng0}, and write
		again
		\begin{align}\label{locdecpr12ng'}
			\begin{split}
				& \int_{k\geq0} \mathcal{K}_R(x,k)e^{ik^2 t} \tilde{h}(k)\,dk = J_+(x) + J_-(x),
			\end{split}
		\end{align}
		where, for ease of reference we recall
		\begin{align}
			\label{J_+}
			& J_+(x) := 
			\chi_+(x)\int_{k\geq0}e^{ik^2t} T(k) \big( m_+(x,k) - 1) e^{ikx} \tilde{h}(k)\,dk.
		\end{align}
		The estimate for $J_-$ can be done in the same way as the one for $J_+$ so we disregard it.
		
		We first look at $t \partial_x J_+$ and proceed in the same manner as in the proof of \eqref{eq:locl2ngf} above.
		Applying $\partial_x$ we obtain three expressions like \eqref{ng1}-\eqref{ng3} with $m_+-1$ instead of $m_+$.
		
		Integrating by parts in $k$ in the first of these three terms (the one where $\partial_x$ hits $e^{ixk}$)
		we obtain expressions like \eqref{ng11}-\eqref{ng13}, again with $m_+-1$ instead of $m_+$
		Using \eqref{eq:mPDO+} from Lemma \ref{lem:pseeasy}, for $x\geq-1$, under the condition $\gamma>\beta+5/2$,
		we have
		\begin{equation}\label{eq:locall2ng1'}
			\left\Vert \jx^\beta\chi_+(x)
			\int_{k\geq0}e^{ik^{2}t}ixe^{ikx}T(k)(m_+(x,k)-1)\tilde{h}(k)\,dk\right\Vert_{L^{2}}
			\lesssim{\big\| \tilde{h} \big\|}_{L^2},
		\end{equation}
		and by the regularity of $T(k)$ and \eqref{eq:mPDO+},
		\begin{equation}\label{eq:locl2ng3'}
			\left\Vert \jx^\beta\chi_+(x)\int_{k\geq0}
			e^{ik^{2}t}e^{ikx}(m_+(x,k)-1)\partial_{k}\left(T(k)\tilde{h}(k)\right)\,dk\right\Vert_{L^{2}}
			\lesssim {\big\| \tilde{h} \big\|}_{H^1}.
		\end{equation}
		Applying the estimate \eqref{eq:pdomk} from Lemma \ref{lem:m_k} with $\gamma > \beta + 5/2$, we have
		\begin{equation}\label{eq:locl2ng2'}
			\left\Vert \jx^\beta \chi_+(x)
			\int_{k\geq0}e^{ik^{2}t}e^{ikx}T(k)\partial_{k}m_+(x,k)\tilde{h}(k)\,dk\right\Vert_{L^{2}}
			\lesssim{\big\| \tilde{h} \big\|}_{L^2}
		\end{equation}
		
		The term where $\partial_x$ hits $m_+-1$ in \eqref{J_+} is exactly the same as \eqref{ng12},
		and \eqref{eq:loc2ng5} with $\alpha = \beta + 1$ gives us the desired bound.
		The term like \eqref{ng3} with $m_+-1$ instead of $m_+$ can also be estimated identically (since $\partial_x\chi_+$
		is compactly supported).
		This concludes the proof.
		%
		%
	\end{proof}

	\medskip
	\subsubsection{The generic case}
	We now look at the generic case. The key difference in this case is that we can use the vanishing
	of $T(k)$ and $R_\pm(k) + 1$ at $k=0$, see Lemma \ref{lem:estiTRTaylor}, to obtain faster time decay rates.
	Eventually this will allow us to carry out the nonlinear analysis under weaker assumption on $V$
	than in the non-generic case.

	\begin{lem}[Improved $L^2$ local decay]\label{lem:localEn}
		Assuming $\jx^\gamma V(x)\in L^1$, $\gamma\geq2$,
		we have
		\begin{align}
			\label{locdecL2}
			{\big\| \jx^{-1} \partial_x
				\big(e^{it(-\partial_{xx}+V)}h\big)_A \big\|}_{L^2_x}\lesssim |t|^{-1} {\big\| \tilde{h} \big\|}_{H^1_k},
			\qquad A \in \{0,S\},
		\end{align}
		and, for $\gamma>\beta+5/2$,
		\begin{align}
			\label{locdecL2R}
			\big{\| \jx^\beta \partial_x
				\big(e^{it(-\partial_{xx}+V)}h\big)_R \big\|}_{L^2_x} \lesssim |t|^{-1} {\big\| \tilde{h} \big\|}_{H^1_k}.
		\end{align}
	\end{lem}

	\begin{proof}
		
		{\it Proof of \eqref{locdecL2}}.
		We start again with the case $A=0$, look at the formula in \eqref{locdecpr11ng} 
		and focus on the contribution with $k\geq0$ given by
		\begin{align}
			& \sqrt{2\pi} \partial_x\int_{k\geq0}\mathcal{K}(x,k)e^{ik^{2}t}\tilde{h}(k)\,dk = I_+ + II_+ + I_- + II_- + III_-,
		\end{align}
		where
		\begin{align}\label{locdecpr21}
			\begin{split}
				& I_+ := \int_{k\geq0} T(k) \partial_x\big(\chi_+(x)m_+(x,k)\big)e^{ixk + ik^{2}t}\tilde{h}(k)\,dk,
				\\
				& II_+ := \int_{k\geq0} ikT(k) \chi_+(x)m_+(x,k)e^{ixk + ik^{2}t} \tilde{h}(k)\,dk,
			\end{split}
		\end{align}
		and
		\begin{align}\label{locdecpr22}
			\begin{split}
				& I_- := \partial_x\chi_-(x) \int_{k\geq0} \big(m_-(x,-k)e^{ikx} + R_-(k)m_-(x,k)e^{-ikx}\big) 
				e^{ik^{2}t}\tilde{h}(k)\,dk,
				\\
				& II_- := \chi_-(x) \int_{k\geq0} \big(\partial_x m_-(x,-k)e^{ikx} + R_-(k)\partial_x m_-(x,k)e^{-ikx}\big) 
				e^{ik^{2}t}\tilde{h}(k)\,dk,
				\\
				& III_- := \int_{k\geq0} ik \, \chi_-(x)m_-(x,-k)e^{ixk + ik^{2}t} \tilde{h}(k)\,dk,
				\\
				& IV_- := \int_{k\geq0} - ik \, \chi_-(x)R_-(k) m_-(x,k)e^{-ixk + ik^{2}t} \tilde{h}(k)\,dk,
			\end{split}
		\end{align}
		
		For the first term in \eqref{locdecpr21} we integrate by parts in $k$ and write
		\begin{align*}
			2it \, I_+ 
			& = \int_{k\geq0} \frac{T(k)}{k} e^{ixk}\partial_x(\chi_+(x)m_+(x,k)) \tilde{h}(k)\,de^{ik^{2}t} = A_1 + A_2 + A_3,
		\end{align*}
		where
		\begin{align*}
			A_1 & := -\int_{k\geq0} e^{ixk}\partial_x(\chi_+(x)m_+(x,k))e^{ik^{2}t}
			\partial_{k}\big(T(k)k^{-1} \tilde{h} \big)\,dk,
			\\
			A_2 & := -\int_{k\geq0} \frac{T(k)}{k} e^{ixk}\partial_{k}\partial_x((\chi_+(x)m_+(x,k))e^{ik^{2}t}\tilde{h}(k)\,dk,
			\\
			A_3 & := -\int_{k\geq0} \frac{T(k)}{k} ixe^{ixk}\partial_x((\chi_+(x)m_+(x,k)) e^{ik^{2}t}\tilde{h}(k)\,dk.
		\end{align*}
		Using Lemma \ref{lem:psehard1} and Lemma \ref{lem:pseeasy}
		we see that the pseudo-differential operator
		with the symbol $\partial_x(\chi_+(x)m_+(x,k))$ is bounded on $L^2$. Then using also $|T(k)| \lesssim |k|$, and Hardy's inequality (recall $\tilde{h}(0)=0$), we have
		\begin{align*}
			{\| A_1 \|}_{L^2} \lesssim {\big\| \partial_k \big(T(k)k^{-1} \tilde{h} \big)\big\|}_{L^2}
			\lesssim {\big\| k^{-1} \tilde{h} \big\|}_{L^2} + {\big\| \partial_k \tilde{h} \big\|}_{L^2} 
			\lesssim {\big\| \tilde{h} \big\|}_{H^1}.
		\end{align*}
		Using that the symbols
		\[\jx^{-1}\partial_{k}\partial_x\big(\chi_+(x)m_+(x,k)\big) 
		\qquad \mbox{and} \qquad \jx^{-1} x \partial_x\big(\chi_+(x)m_+(x,k)) \]
		also give rise to bounded pseudo-differential operators on $L^2$, 
		see Lemmas \ref{lem:pseeasy}, \ref{lem:psehard1}, \ref{lem:psehardk1} and \ref{lem:m_k}, 
		we have
		\begin{align*}
			{\| \jx^{-1} A_2 \|}_{L^2} + {\| \jx^{-1} A_3 \|}_{L^2} \lesssim {\big\| \tilde{h} \big\|}_{L^2}.
		\end{align*}
		The term $II_+$ can be estimated similarly, and it is in fact easier
		thanks to the extra power of $k$ in the integrand. More precisely we can write
		\begin{align*}
			2t \, II_+ 
			& = \int_{k\geq0} T(k) e^{ixk} \chi_+(x)m_+(x,k)\tilde{h}(k)\,de^{ik^{2}t} = B_1 + B_2 + B_3,
			\\
			B_1 & := -\chi_+(x)\int_{k\geq0} e^{ixk} m_+(x,k)e^{ik^{2}t} \partial_{k} \big( T(k)\tilde{h}(k)\big)\,dk,
			\\
			B_2 & := -\chi_+(x)\int_{k\geq0} T(k)e^{ixk}\partial_{k}m_+(x,k)e^{ik^{2}t}\tilde{h}(k)\,dk,
			\\
			B_3 & := -\chi_+(x)\int_{k\geq0} T(k) \, ixe^{ixk} \, m_+(x,k)e^{ik^{2}t}\tilde{h}(k)\,dk.
		\end{align*}
		All these terms can be handled using the boundedness of pseudo-differential operators
		with symbols $\chi_+(x)m_+(x,k)$, $\jx^{-1}\chi_+(x)\partial_km_+(x,k)$,
		and $\jx^{-1} x\chi_+(x)m_+(x,k)$, see Lemmas \ref{lem:pseeasy} and \ref{lem:m_k}, 
		and using $|T(k)|, |\partial_kT(k)| \lesssim 1$, see Lemma \ref{estiTR},
		so to obtain
		\begin{align*}
			|t| \, {\| \jx^{-1} II_+ \|}_{L^2} \lesssim {\big\| \tilde{h} \big\|}_{L^2}.
		\end{align*}
		
		Next, we look at the terms in \eqref{locdecpr22}. 
		$III_-$ and $IV_-$ are similar to $II_+$ and can be estimated in the same way, so we skip them.
		For the terms $I_-$ and $II_-$ we use arguments similar to those used for $I_+$ above,
		and $R_+(0)=-1$, see Lemma \ref{lem:estiTRTaylor}.
		We only look at $II_-$ since the term $I_-$ is easier to deal with.
		We distinguish between the cases $|k|\gtrsim 1$ and $|k|\lesssim 1$, 
		and then further rewrite the integral supported on $|k| \lesssim 1$ as a sum of three pieces:
		\begin{align}\label{locdecpr23}
			\begin{split}
				& II_- = C_1 + C_2 + C_3 + C_4,
				\\
				& C_1 = \chi_-(x) \int \varphi_{>1}(k) \big(\partial_x m_-(x,-k)e^{ikx} + R_-(k)\partial_x m_-(x,k)e^{-ikx}\big) 
				e^{ik^{2}t} \tilde{h}(k)\,dk,
				\\
				& C_2 = \chi_-(x) \int \varphi_{\leq 1}(k) \big(R_-(k) + 1\big) e^{-ikx} \partial_x m_-(x,k)
				e^{ik^{2}t} \tilde{h}(k)\,dk,
				\\
				& C_3 = \chi_-(x) \int \varphi_{\leq 1}(k) \big(\partial_x m_-(x,-k)-\partial_x m_-(x,k)\big)e^{ikx}  
				e^{ik^{2}t} \tilde{h}(k)\,dk,
				\\
				& C_4 = \chi_-(x) \int \varphi_{\leq 1}(k) \big(e^{ikx}-e^{-ikx}\big)\partial_x m_-(x,k) 
				e^{ik^{2}t} \tilde{h}(k)\,dk.
			\end{split}
		\end{align}
		
		On the support of $C_1$ we have $|k|\gtrsim 1$, so we can integrate by parts in $k$ 
		using $e^{itk^2} = (2it k)^{-1} \partial_k e^{itk^2}$, and gain a factor of $|t|^{-1}$ 
		without introducing any singularity in $k$.
		Using that the symbol $\jx^{-1}\chi_-(x)\partial_k \partial_x m_-(x,\pm k)$ defines a bounded PDO, 
		see Lemma \ref{lem:psehardk1}, we get the desired bound.
		
		On the support of $C_2$ we have $|k|\lesssim 1$ and, see Lemma \ref{lem:estiTRTaylor}, $R_-(k) + 1 = \mathcal{O}(k)$.
		In particular, this term is essentially the same as $I_+$ and can be treated identically.
		
		For $C_3$ we first notice that 
		\begin{align*}
			\partial_x m_-(x,-k)-\partial_x m_-(x,k) = -\int_{-1}^1 (\partial_x\partial_k m_-)(x,zk) \, dz \cdot k
			=: a(x,k) \cdot k.
		\end{align*}
		We can then use the $k$ factor in the above right-hand side to integrate by parts through the usual identity
		$e^{itk^2} = (2it k)^{-1} \partial_k e^{itk^2}$, and since for $x\leq1$, $\chi_-(x)\partial_k^p a(x,k)$, for $p=0,1$,
		are symbols of bounded PDOs, see Lemmas \eqref{lem:psehardk1} and \ref{lem:PDOmxkk} 
		(and recall that $|k|\lesssim 1$ here), 
		we get the bound ${\| C_3\|}_{L^2} \lesssim |t|^{-1}{\| \tilde{h} \|}_{H^1}$.

		The last term $C_4$ can be estimated using that, for $\gamma>\max(\beta/2+3/4,\beta)$,
		\begin{equation*}
			\left\Vert \jx^{\beta} \mathrm{1}_{x\leq1} 
			\int k^{-1} \jx^{-1} (e^{ikx}-e^{-ikx}) \partial_x m_-(x,k) \varphi_{\leq 1}(k)g(k)\,dk
			\right\Vert_{L^2_x}\lesssim{\| g \|}_{L^2}.
		\end{equation*}
		This is similar to the estimate \eqref{eq:psehard} and can be obtained as in the proof of Lemma \ref{lem:psehard1},
		using that $k^{-1} \jx^{-1} (e^{ikx}-e^{-ikx})$ is uniformly bounded, 
		and using $|k| \lesssim 1$ to bound the part corresponding to the last integral in \eqref{eq:split}.

		\medskip
		\noindent
		{\it Proof of \eqref{locdecL2R}}.
		As in Lemma \ref{lem:locl2ng}, the estimate involving $\mathcal{K}_R$ 
		follows from the same arguments used in the case of $\mathcal{K}$, 
		by replacing 
		$m_\pm$ with $m_\pm-1$, and using 
		the estimates \eqref{eq:mPDO+}, \eqref{eq:psehard}, 
		\eqref{eq:psehard-1}, \eqref{eq:pdomk} and \eqref{eq:mPDOhard+31}.
	\end{proof}


	\bigskip
	\section{Cubic terms analysis\label{sec:Cubic}}
	In this section, we carry out the analysis for
	\begin{equation}
		i\partial_{t}u-\partial_{xx}u\pm\left|u\right|^{2}u+Vu=0,
		\qquad u\left(0\right)=u_{0},\label{eq:NLSE-2}
	\end{equation}
	where $u_0 \in H^{1,1}(\mathbb{R})$. 
	We will consider the bootstrap space
	\begin{align}\label{X_T}
		X_{T}:=\left\{ u|\,{\| u \|} _{L_{t}^{\infty}\left(\left[0,T\right]:H_{x}^{1}\right)}
		+{\| \tilde{f} \|}_{L_{t}^{\infty}\left(\left[0,T\right]:L_{k}^{\infty}\right)}
		+{\| \langle t \rangle^{-\alpha}\tilde{f} \|}_{L_{t}^{\infty}\left(\left[0,T\right]:H_{k}^{1}\right)}\right\},
	\end{align}
	for $\alpha>0$ small enough depending on $\gamma$,
	where $f=e^{-itH}u$ is the profile of the solution. 
	
	It also possible to consider less regular data $u_{0}\in L^{2,1}(\mathbb{R})$,
	in which case 
	one would use instead
	\[
	X_{T}:=\left\{ u|\,{\| u \|}_{L_{t}^{\infty}\left(\left[0,T\right]:L^2_x\right)}
	+{\| \tilde{f} \|}_{L_{t}^{\infty}\left(\left[0,T\right]:L_{k}^{\infty}\right)}
	+{\| \langle t \rangle^{-\alpha}\tilde{f} \|}_{L_{t}^{\infty}\left(\left[0,T\right]:H_{k}^{1}\right)}\right\} .
	\]
	It is easy to establish the $L^{2}$ and $H^{1}$ bounds
	from the conservation laws and Strichartz estimates.  
	
	
	Using the local theory, and for the sake of convenience, we consider the solution starting at $t=1$.
	Taking the distorted Fourier transform and using Duhamel's formula, in terms of the profile we can write
	\begin{equation}
		\tilde{f}(t,k)=\tilde{f}(1,k)\pm i\int_{1}^{t}
		\iiint e^{is(-k^2+\ell^2-m^2+n^2)}\tilde{f}(s,\ell)\overline{\tilde{f}(s,m)}\tilde{f}(s,n)\mu(k,\ell,m,n)\,dndmd\ell ds\label{eq:profile}
	\end{equation}
	where
	\begin{equation}\label{eq:nonmea}
		\mu(k,\ell,m,n):=\int\overline{\mathcal{K}(x,k)}\mathcal{K}(x,\ell)\overline{\mathcal{K}(x,m)}\mathcal{K}(x,n)\,dx
	\end{equation}
	is what we call \emph{the nonlinear spectral distribution}.
	Our main purpose in this section is to obtain the $L_{k}^{2}$ bound for $\partial_{k}\tilde{f}$.
	We will show

	\begin{prop}\label{pro:weightmain}
		For $1\leq t\leq T$, one has that, for some $C>0$,
		\begin{align}\label{weightmainconc}
			{\big\| \partial_{k}\tilde{f}(t) \big\|}_{L_{k}^{2}}
			\leq {\big\| \partial_{k}\tilde{f}(1) \big\|}_{L_k^2}
			+ C t^{\alpha} {\| u \|}_{X_{T}}^{3}.
		\end{align}
	\end{prop}

	\begin{rem}
		Observe that under our assumptions on $V$ (genericity, or \eqref{nongenericresonance})
		we can split our computations into $k\geq0$ or $k<0$, 
		using the indicator functions $\mathrm{1}_{+} = \mathrm{1}_{[0,\infty)}$ 
		and $\mathrm{1}_{-} = \mathrm{1}_{(-\infty,0)}$. 
		Then, the contribution from $\partial_k$ hitting $\mathrm{1}_{\pm}$, 
		results in a delta measure at $k=0$ which, due to our assumptions 
		vanishes. The same observation also holds for the input variables $\ell,m,n$ in \eqref{eq:profile}.
	\end{rem}

	The first key step is the analysis of the nonlinear spectral distribution \eqref{eq:nonmea}.
	

	\medskip
	\subsection{Decomposition of the nonlinear spectral distribution and the nonlinearity}\label{subsec:DecM}
	The decomposition of the generalized eigenfunctions \eqref{eq:decomK} 
	can used to decompose the spectral distribution \eqref{eq:nonmea}
	as
	\begin{align}\label{decM1}
		(2\pi)^{2}\mu(k,\ell,m,n)=\mu_{S}(k,\ell,m,n) + \mu_{R}(k,\ell,m,n)
	\end{align}
	where $\mu_{S}$ is a singular part and $\mu_R$ is a regular part.
	Here is a detailed descriptions of each component. 
	
	First, one can write $(2\pi)^{2}\mu$ as the sum
	\begin{align}\label{musum}
		\sum_{(A,B,C,D)\in\left\{S,R\right\}}
		\int\overline{\mathcal{K}_{A}(x,k)}\mathcal{K}_{B}(x,\ell) \overline{\mathcal{K}_{C}(x,m)} \mathcal{K}_{D}(x,n)\,dx.
	\end{align}
	The singular part $\mu_{S}$ is obtained from the case when all the four pieces of the generalized eigenfunctions
	are either $\chi_+\K_+$ or $\chi_-\K_-$, see \eqref{Ksing}; that is, we let
	\begin{align}\label{muSpm}
		\begin{split}
			\mu_S(k,\ell,m,n) & = \mu_+(k,\ell,m,n) + \mu_-(k,\ell,m,n)
			\\
			\mu_{\ast}(k,\ell,m,n) & := \int \varphi_{\ast}(x)
			\overline{\K_\ast(x,k)} \K_\ast(x,\ell) \overline{\K_\ast(x,m)} \K_\ast(x,n) \, dx, 
			\qquad \varphi_{\ast}=\chi_{\ast}^{4}, \quad
			\ast \in \{+,-\},
		\end{split}
	\end{align}

	For the regular part, $\mu_{R}$, we let 
	\[
	\mathcal{X}_{R}=\left\{ \left(A_{0},A_{1},A_{2},A_{3}\right):\,\exists j=0,1,2,3:\ A_{j}=R\right\} 
	\]
	and write
	\begin{equation}\label{eq:muR}
		\mu_{R}(k,\ell,m,n) = \mu_{R,1}(k,\ell,m,n) + \mu_{R,2}(k,\ell,m,n)
	\end{equation}
	where
	\begin{equation}\label{muR1}
		\mu_{R,1}(k,\ell,m,n):=\sum_{(A,B,C,D)\in\mathcal{X}_{R}}\int\overline{\mathcal{K}_{A}(x,k)}\mathcal{K}_{B}(x,\ell)\overline{\mathcal{K}_{C}(x,m)}\mathcal{K}_{D}(x,n)\,dx
	\end{equation}
	and 
	\begin{equation}\label{muR2}
		\mu_{R,2}(k,\ell,m,n) :=
		\int\overline{\mathcal{K}_S(x,k)}\mathcal{K}_S(x,\ell)\overline{\mathcal{K}_S(x,m)}\mathcal{K}_S(x,n)\,dx
		- \mu_{S}(k,\ell,m,n) 
	\end{equation}

	\medskip
	According to the decomposition of $\mu$, from \eqref{eq:profile} we can write
	\begin{align}\label{eq:expandf}
		\begin{split}
			\tilde{f}(t,k) & =\tilde{f}(1,k)
			\pm i\int_{1}^{t} \iiint e^{is(-k^2+\ell^2-m^2+n^2)} \tilde{f}(s,\ell)\overline{\tilde{f}(s,m)}
			\tilde{f}(s,n)\mu(k,\ell,m,n)\,dndmd\ell ds
			\\
			& = \tilde{f}(1,k)\pm i\int_{1}^{t} \big(\mathcal{N}_{+} + \mathcal{N}_{-}
			+ \mathcal{N}_{R,1} + \mathcal{N}_{R,2} \big)\,ds
		\end{split}
	\end{align}
	where
	\begin{align}\label{Nast}
		\mathcal{N}_{\ast}(s,k)=\frac{1}{(2\pi)^{2}}
		\iiint e^{is(-k^2+\ell^2-m^2+n^2)}
		\tilde{f}(s,\ell)\overline{\tilde{f}(s,m)}\tilde{f}(s,n)\mu_{\ast}(k,\ell,m,n)\,dndmd\ell.
	\end{align}
	To prove the main Proposition \ref{pro:weightmain} it then suffices to show
	\begin{align}
		\label{weightmain1}
		& {\| \partial_k \mathcal{N}_{+} \|}_{L^2} 
		+ {\| \partial_k \mathcal{N}_{-} \|}_{L^2} \lesssim t^{\alpha} {\| u \|}_{X_{T}}^{3},
		\\
		\label{weightmain2}
		& {\| \partial_k \mathcal{N}_{R,1} \|}_{L^2} 
		+ {\| \partial_k \mathcal{N}_{R,2} \|}_{L^2} \lesssim t^{\alpha} {\| u \|}_{X_{T}}^{3}.
	\end{align}
	The proof of \eqref{weightmain1} is given in Subsection \ref{ssecsing} below,
	and the proof of \eqref{weightmain2} in Section \ref{sec:estimateregular}.

	\medskip
	\subsection{Estimates for the singular part}\label{ssecsing}
	Before getting to the nonlinear estimates,
	we first need formulas for the Fourier transform of the cutoff $\varphi_{\pm}=\chi_{\pm}^{4}$ in \eqref{muSpm}
	which are smoothed out versions of the indicator functions of $\pm x>0$.
	From standard formulas for the (regular) Fourier transform, we can write
	\begin{equation}\label{eq:hatphi+-}
		\widehat{\varphi_+}(k)=\sqrt{\frac{\pi}{2}}\delta_{0}+\frac{\hat{\zeta}(k)}{ik}+\widehat{\varpi}(k)
		\qquad \widehat{\varphi_-}=\sqrt{\frac{\pi}{2}}\delta_{0}-\frac{\hat{\zeta}(k)}{ik}+\widehat{\varpi}(k)
	\end{equation}
	where $\zeta$ is an even $C_{c}^{\infty}$ function with integral
	$1$ and $\varpi$ denotes a (generic) $C_{c}^{\infty}$ function. 
	See \cite[page 23]{GPR} for the details.

	\medskip
	\subsubsection{Structure of $\mu_S$ and approximate commutation}
	Recall that $\K_\pm$ can be written as a linear combination of exponentials
	\eqref{Ksing+'}-\eqref{Ksing-'}. 
	Then, from \eqref{muSpm} we have
	\begin{align}\label{muSstr1}
		\begin{split}
			\mu_\ast(k,\ell,m,n) & = \sum_{\e_0,\e_1,\e_2,\e_3 \in \{+,-\}} \int \varphi_{\ast}(x) \,
			\overline{a_\ast^{\e_0}(k) e^{\e_0ikx}} \,a_\ast^{\e_1}(\ell) e^{\e_1i\ell x} \,
			\overline{a_\ast^{\e_2}(m) e^{\e_2imx}} \,a_\ast^{\e_3}(n) e^{\e_3inx} \, dx, 
			\\
			& = \sum_{\e_0,\e_1,\e_2,\e_3 \in \{+,-\}}\overline{a_\ast^{\e_0}(k)} \,a_\ast^{\e_1}(\ell) \,
			\overline{a_\ast^{\e_2}(m)} \,a_\ast^{\e_3}(n) \, \sqrt{2\pi}\,
			\widehat{\varphi_\ast}(\e_0k-\e_1\ell+\e_2m-\e_3n).
		\end{split}
	\end{align}
	The formulae \eqref{eq:hatphi+-} can be used to express the last $\widehat{\varphi_\pm}$ above.

	We now establish a simple algebraic lemma which will allow us to establish
	a priori weighted estimates for $\mathcal{N}_\pm$, see \eqref{Nast}, in a straightforward fashion.
	
	\begin{lem}\label{lem:algetri}
		For $f_{j}\in\mathcal{S}$, define the trilinear form
		\begin{align}\label{eq:tril}
			\begin{split}
				\mathcal{T}_{\mathrm{b},\e_1,\e_2,\e_3}(f_1,f_2,f_3)(k)=
				\iiint e^{it(-k^2+\ell^2-m^2+n^2)}f_{1}(\ell) \overline{f_{2}(m)} f_{3}(n)
				\\ \times \mathrm{b}(k+\e_1\ell+\e_2m+\e_3n)\,d\ell dmdn 
			\end{split} 
		\end{align}
		where $\epsilon_{j}\in\left\{ \pm1\right\}$, and $\mathrm{b}$ is a distribution. 
		Then one has
		\begin{align}\label{eq:kdiffT}
			\begin{split}
				\partial_{k}\mathcal{T}_{\mathrm{b},\e_1,\e_2,\e_3}(f_1,f_2,f_3)(k) 
				= & -\e_1\mathcal{T}_{\mathrm{b},\e_1,\e_2,\e_3}
				\left(\partial_{\ell}f_{1},f_{2},f_{3}\right)(k)
				+\e_2\mathcal{T}_{\mathrm{b},\e_1,\e_2,\e_3}\left(f_{1},\partial_{m}f_{2},f_{3}\right)(k)
				\\
				& -\e_3\mathcal{T}_{\mathrm{b},\e_1,\e_2,\e_3}\left(f_{1},f_{2},\partial_{n}f_{3}\right)(k)
				-2it \, \mathcal{T}_{\mathrm{yb},\e_1,\e_2,\e_3}(f_1,f_2,f_3)(k).
			\end{split}
		\end{align}
	\end{lem}
	
	\begin{proof}
		Differentiating \eqref{eq:tril} with respect to $k$ gives
		\begin{align}\label{algetri1}
			\partial_k \mathcal{T}_{\mathrm{b},\e_1,\e_2,\e_3}(f_1,f_2,f_3)(k)
			= \mathcal{T}_{\partial_y\mathrm{b},\e_1,\e_2,\e_3}(f_1,f_2,f_3)(k)
			-2it k \, \mathcal{T}_{\mathrm{b},\e_1,\e_2,\e_3}(f_1,f_2,f_3)(k).
		\end{align}
		Denoting $p := k+\e_1\ell+\e_2m+\e_3n$ we have
		$\partial_{y}\mathrm{b}(p) = \partial_{k}\mathrm{b} (p)=
		\e_1\partial_{\ell}\mathrm{b}(p)=\e_2\partial_{m}\mathrm{b}(p)=\e_3\partial_{n}\mathrm{b}(p)$,
		and therefore
		\begin{equation}\label{eq:relation-1}
			\partial_{y}\mathrm{b}
			=\e_1\partial_{\ell}\mathrm{b}-\e_2\partial_{m}\mathrm{b}+\e_3\partial_{n}\mathrm{b}
		\end{equation}
		We use this identity to integrate by parts in $\ell,m$ and $n$ 
		in the expression for $\mathcal{T}_{\partial_y\mathrm{b},\e_1,\e_2,\e_3}$ 
		obtaining:
		\begin{align}\label{eq:phase1}
			\begin{split}
				\mathcal{T}_{\partial_y\mathrm{b},\e_1,\e_2,\e_3}&(f_1,f_2,f_3)(k)
				\\ & = -\e_1\mathcal{T}_{\mathrm{b}}\left(\partial_{\ell}f_{1},f_{2},f_{3}\right)(k)
				+\e_2\mathcal{T}_{\mathrm{b}}\left(f_{1},\partial_{m}f_{2},f_{3}\right)(k)
				-\e_3\mathcal{T}_{\mathrm{b}}\left(f_{1},f_{2},\partial_{n}f_{3}\right)(k)
				\\
				& -2it \iiint \e_1 \ell \, e^{it(-k^2+\ell^2-m^2+n^2)}f_{1}(\ell) \overline{f_{2}(m)} f_{3}(n)
				\mathrm{b}(k+\e_1\ell+\e_2m+\e_3n)\,d\ell dmdn
				\\
				& -2it \iiint \e_2 m \, e^{it(-k^2+\ell^2-m^2+n^2)}f_{1}(\ell) \overline{f_{2}(m)} f_{3}(n)
				\mathrm{b}(k+\e_1\ell+\e_2m+\e_3n)\,d\ell dmdn
				\\
				& -2it \iiint \e_3 n \, e^{it(-k^2+\ell^2-m^2+n^2)}f_{1}(\ell) \overline{f_{2}(m)} f_{3}(n)
				\mathrm{b}(k+\e_1\ell+\e_2m+\e_3n)\,d\ell dmdn.
			\end{split}
		\end{align}
		The first three terms on the right-hand side of \eqref{eq:phase1} and \eqref{eq:tril} are the same. 
		Moreover, the sum of last three terms in \eqref{eq:phase1} with the last term on
		the right-hand side of \eqref{algetri1} gives 
		\begin{align*}
			-2it \iiint (k+\e_1\ell+\e_2m+\e_3n)\, e^{it(-k^2+\ell^2-m^2+n^2)}f_{1}(\ell)\overline{f_{2}(m)}f_{3}(n)
			\\ \times \mathrm{b}(k+\e_1\ell+\e_2m+\e_3n)\,d\ell dmdn
		\end{align*}
		which is the last term in \eqref{eq:kdiffT}.
	\end{proof}

	To estimate the cubic terms $\mathcal{N}_\pm$, 
	we will apply Lemma \ref{lem:algetri} with $\mathrm{b}=\widehat{\varphi_\pm}$
	and $f_j = a_\pm^{\e_j} \tilde{f}$,
	%
	and the following identity 
	\begin{lem}
		\label{lem:inverseFD}
		Given $f_{j}\in\mathcal{S}$ and $\epsilon_{j}\in \{ +,-\}$, $j=1,2,3$,
		then
		\begin{align}\label{eq:inverseFD}
			\begin{split}
				& \mathcal{F}^{-1} \big[e^{itk^2} {\mathcal{T}}_{\mathrm{b},\e_1,\e_2,\e_3}(f_1,f_2,f_3)\big]
				=-\e_1u_{1}(t,-\e_1x)\overline{u_{2}(t,-\e_2x)}u_{3}(t,-\e_3x)\mathcal{F}^{-1}\left[\mathrm{b}\right],
				\\
				& \qquad \mbox{where} \qquad u_{j}:=e^{-it\partial_{xx}}\check{f}_{j}.
			\end{split}
		\end{align}
	\end{lem}

	\begin{proof}
		The proof is an explicit computation which is left to the reader.
	\end{proof}

\medskip
\subsubsection{Nonlinear estimates for $\partial_{k}\mathcal{N}_{\pm}$}\label{ssecNpm}

Let us work on $\mathcal{N}_+$; the case of $\mathcal{N}_-$ is identical.
From \eqref{Nast} and \eqref{muSstr1} we write
\begin{align}\label{N+0}
	\begin{split}
		(2\pi)^{3/2} \mathcal{N}_{+}(s,k) = \sum_{\e_0,\e_1,\e_2,\e_3 \in \{+,-\}}\overline{a_+^{\e_0}(k)} \,
		\iiint e^{is(-k^2+\ell^2-m^2+n^2)} \, a_+^{\e_1}(\ell)\tilde{f}(s,\ell) 
		\\ \times \, \overline{a_+^{\e_2}(m)\tilde{f}(s,m)} \,
		a_+^{\e_3}(n)\tilde{f}(t,n)
		\, \widehat{\varphi_+}(\e_0k-\e_1\ell+\e_2m-\e_3n) \,dndmd\ell.
	\end{split}
\end{align}
We let
\begin{align}\label{N+fj}
	f_j(s,\cdot) := a_+^{\e_j}(\cdot)\tilde{f}(s,-\e_j \,\cdot)  \qquad j=1,2,3,
\end{align}
where the coefficients are defined in \eqref{Ksing+'}-\eqref{Ksing-'},
and using the notation \eqref{eq:tril} we write
\begin{align}\label{N+1}
	\begin{split}
		(2\pi)^{3/2} \mathcal{N}_{+}(s,k) = \sum_{\e_0,\e_1,\e_2,\e_3 \in \{+,-\}}
		\overline{a_+^{\e_0}(k)} \,
		\mathcal{T}_{\mathrm{b},\e_1,\e_2,\e_3}(f_1,f_2,f_3)(\e_0k), \qquad \mbox{with} \quad \mathrm{b}=\widehat{\varphi_+}.
	\end{split}
\end{align}

To estimate $\mathcal{N}_{+}$ it suffices to estimate a generic term in the sum in \eqref{N+1}. 
To estimate $\mathcal{N}_{+}$ it suffices to estimate a generic term in the sum in \eqref{N+1}. 
We first note that, from the formulas \eqref{Ksing+'} and \eqref{Ksing-'}, and Lemmas 
\ref{estiTR} and \ref{lem:estiTRTaylor},
one has $\partial_k a_+^{-}(k) = \delta_0(k) + \mathrm{1}_{-}(k) \partial_k R_+(-k)$,
and
\begin{align*}
	\partial_k \sum_{\e_0 \in \{+,-\}} a_+^{\e_0}(k) & = 
	\partial_k\big( \mathrm{1}_{+}(k)T(k) + \mathrm{1}_{-}(k) + \mathrm{1}_{-}(k)R_+(-k) \big)
	\\
	& = \mathrm{1}_{+}(k) \partial_k T(k) + \mathrm{1}_{-}(k) \partial_k R_+(-k).
\end{align*}
Then, writing
\begin{align*}
	\sum_{\e_0 \in \{+,-\}} \overline{a_+^{\e_0}(k) }\,
	\mathcal{T}_{\mathrm{b},\e_1,\e_2,\e_3}(f_1,f_2,f_3)(\e_0 k)
	= \big(\overline{a_+^{+}(k) }+ \overline{a^-_+(k)} \big)\,
	\mathcal{T}_{\mathrm{b},\e_1,\e_2,\e_3}(f_1,f_2,f_3)(k)
	\\
	+\overline{a^-_+(k)} \big(
	\mathcal{T}_{\mathrm{b},\e_1,\e_2,\e_3}(f_1,f_2,f_3)(-k)-
	\mathcal{T}_{\mathrm{b},\e_1,\e_2,\e_3}(f_1,f_2,f_3)(k)\big),
\end{align*}
and using Lemma \ref{estiTR}, we see that
\begin{align}\label{N+2}
	\begin{split}
		{\left\Vert \partial_k \big(\sum_{\e_0 \in \{+,-\} \overline{a_+^{\e_0}}(\cdot)} \,
			\mathcal{T}_{\mathrm{b},\e_1,\e_2,\e_3}(f_1,f_2,f_3)(\e_0 \cdot) \big) \right\Vert_{L^2}}
		& \lesssim {\| \mathcal{T}_{\mathrm{b},\e_1,\e_2,\e_3}(f_1,f_2,f_3)\|}_{L^2}
		\\
		& + {\|\partial_k\mathcal{T}_{\mathrm{b},\e_1,\e_2,\e_3}(f_1,f_2,f_3)\|}_{L^2}.
	\end{split}
\end{align}

The first term  is easily estimated using Lemma \ref{lem:inverseFD}: letting
\begin{align}\label{N+uj}
	u_{j} := e^{-it\partial_{xx}}\check{f}_{j} 
	= e^{-it\partial_{xx}} \widehat{\mathcal{F}}^{-1}\big( a_+^{\e_j}\tilde{f}(-\e_j \,\cdot) \big), \qquad j=1,2,3,
\end{align}
we have
\begin{align*}
	{\| \mathcal{T}_{\mathrm{b},\e_1,\e_2,\e_3}(f_1,f_2,f_3)\|}_{L^2} \lesssim 
	{\| u_1(s,-\e_1\cdot) \, u_2(s,-\e_2\cdot)\, u_3(s,-\e_3\cdot) \, \varphi_+ \|}_{L^2}
	\\ \lesssim {\| u_1(s) \|}_{L^2} {\| u_2(s) \|}_{L^\infty} {\| u_3(s) \|}_{L^\infty}
	\lesssim s^{-1} {\big\| u \big\|}_{X_T}^3,
\end{align*}
where we used \eqref{remcoeffest1} in the last inequality.
Upon integration over time $s$, see \eqref{eq:expandf} we see that
this is consistent with the desired bound \eqref{weightmainconc}.

For the second term on the right-hand side of \eqref{N+2} we first use Lemma \ref{lem:algetri} to obtain
\begin{align}
	\nonumber & {\|\partial_k\mathcal{T}_{\mathrm{b},\e_1,\e_2,\e_3}(f_1,f_2,f_3)\|}_{L^2}
	\\ 
	\label{N+3} 
	& \lesssim {\|\mathcal{T}_{\mathrm{b},\e_1,\e_2,\e_3}(\partial_kf_1,f_2,f_3)\|}_{L^2}
	+ {\|\mathcal{T}_{\mathrm{b},\e_1,\e_2,\e_3}(f_1,\partial_kf_2,f_3)\|}_{L^2}
	+ {\|\mathcal{T}_{\mathrm{b},\e_1,\e_2,\e_3}(f_1,f_2,\partial_kf_3)\|}_{L^2}
	\\ 
	\label{N+4}
	& + s \,{\|\mathcal{T}_{\mathrm{yb},\e_1,\e_2,\e_3}(f_1,f_2,f_3)\|}_{L^2}.
\end{align}

The three terms in \eqref{N+3} are similar and can be estimated using again Lemma \ref{lem:inverseFD},
and \eqref{N+fj} with the estimate \eqref{remcoeffest1}:
\begin{align*}
	{\|\mathcal{T}_{\mathrm{b},\e_1,\e_2,\e_3}(\partial_kf_1,f_2,f_3)\|}_{L^2}
	\lesssim {\| \partial_kf_1(s) \|}_{L^2} {\| u_2(s) \|}_{L^\infty} {\| u_3(s) \|}_{L^\infty}
	\\
	\lesssim s^\alpha {\big\| u \big\|}_{X_T} \Big( \frac{1}{\sqrt{s}} {\big\| u \big\|}_{X_T} \Big)^2,
\end{align*}
which suffices for \eqref{weightmainconc}.

For the last term \eqref{N+4} we note that, for $\mathrm{b}(y)=\widehat{\varphi_+}$, 
in view of the formulae \eqref{eq:hatphi+-},
we have 
\begin{align}
	y\mathrm{b}(y) = -i\hat{\zeta}(y) + y\widehat{\varpi}(y)=\widehat{\Psi}(y)
\end{align}
with $\Psi$ a Schwartz function. 
Then, Lemma \ref{lem:inverseFD} gives
\begin{align*}
	s \,{\|\mathcal{T}_{\mathrm{yb},\e_1,\e_2,\e_3}(f_1,f_2,f_3)\|}_{L^2}
	& \lesssim s {\| u_1(s,-\e_1) \, u_2(s,-\e_2)\, u_3(s,-\e_3) \, \Psi \|}_{L^2}
	\\ 
	& \lesssim s {\| \jx^{-1} u_1(s) \|}_{L^\infty} {\| \jx^{-1} u_3(s) \|}_{L^\infty} {\| \jx^{-1} u_3(s) \|}_{L^\infty}.
\end{align*}
Recalling \eqref{N+fj} and using \eqref{remcoeffest3} we get
\begin{align*}
	s \,{\|\mathcal{T}_{\mathrm{yb},\e_1,\e_2,\e_3}(f_1,f_2,f_3)\|}_{L^2} 
	& \lesssim s \Big( \frac{1}{s^\frac{3}{4}} {\big\| \tilde{f} \big\|}_{H^1} \Big)^3
	\lesssim s^{-\frac{5}{4} + 3\alpha} {\big\| u \big\|}_{X_T}^3.
\end{align*}
Upon integration in $s$ we get a bound of  ${\| u \|}_{X_T}^3$,
which is stronger than the desired \eqref{weightmainconc}.

\bigskip
\section{Estimates for the regular part}\label{sec:estimateregular}

For this part we need to treat differently the generic and non-generic cases.

\subsection{Generic case}
We first discuss the estimates for the regular part in the generic case.
Throughout, we assume $\jx^\gamma V(x)\in L^1$ with $\gamma>5/2$ and let $\zeta=:\gamma-5/2>0$.

For convenience let us recall here the definitions \eqref{muR1} and \eqref{muR2}:
\begin{align}\label{muR1'}
	\begin{split}
		\mu_{R,1}(k,\ell,m,n):=\sum_{(A,B,C,D)\in\mathcal{X}_{R}}\int\overline{\mathcal{K}_{A}(x,k)}\mathcal{K}_{B}(x,\ell)\overline{\mathcal{K}_{C}(x,m)}\mathcal{K}_{D}(x,n)\,dx
		\\
		\mathcal{X}_{R}=\left\{ \left(A_{0},A_{1},A_{2},A_{3}\right):\,\exists j=0,1,2,3:\ A_{j}=R\right\},
	\end{split}
\end{align}
and
\begin{equation}\label{muR2'}
	\mu_{R,2}(k,\ell,m,n):=
	\int\overline{\mathcal{K}_S(x,k)}\mathcal{K}_S(x,\ell)\overline{\mathcal{K}_S(x,m)}\mathcal{K}_S(x,n)\,dx
	- \mu_{S}(k,\ell,m,n), 
\end{equation}
see also \eqref{muSpm}.
By their definitions the regular parts of the nonlinearity $\mathcal{N}_{R,1}$ and $\mathcal{N}_{R,2}$,
see the notation \eqref{Nast} 
have strong localization properties and can be estimated using
the improved local decay estimates 
from Lemmas \ref{lemlocdecinftyng}, \ref{lem:regularL2}, \ref{lem:locl2ng} and \ref{lem:localEn}.

\subsubsection{Estimate of $\mathcal{N}_{R,1}$}
Each element of the sum defining $\mu_{R,1}$ has at least one of the indexes $(A,B,C,D)$ equal to $R$.
It suffices to consider two situations: $A=R$ and $D=R$; all the others terms can be treated identically. 
We then define
\begin{align}
	\label{eq:mur11}
	\mu_{R,1}^{(1)} &:=\int\overline{\mathcal{K}}_{R}(x,k)\mathcal{K}_{M_{1}}(x,\ell)\overline{\mathcal{K}_{M_{2}}(x,n)}\mathcal{K}_{M_{3}}(x,m)\,dx
	\\
	\label{eq:mur12}
	\mu_{R,1}^{(2)} &:=\int\overline{\mathcal{K}}_{S}(x,k)\mathcal{K}_{M_{1}}(x,\ell)\overline{\mathcal{K}_{M_{2}}(x,n)}\mathcal{K}_{R}(x,m)\,dx,
	\qquad M_{i}\in\{S,R\}.
\end{align}
and see that, in order to estimate the $L_{k}^{2}$ norm of $\partial_{k}\mathcal{N}_{R,1}$,
it suffices to estimate the $L_{k}^{2}$ norms of
\begin{align}\label{NR11}
	A_1(k) := & isk\iiint\tilde{u}(\ell)\overline{\tilde{u}}(n)\tilde{u}(m) \, \mu_{R,1}^{(1)}\left(k,\ell,n,m\right)
	\,d\ell dmdn,
	\\
	\label{NR11'}
	A_2(k) := & \iiint\tilde{u}(\ell)\overline{\tilde{u}}(n)\tilde{u}(m) \, \partial_{k}\mu_{R,1}^{(1)}\left(k,\ell,n,m\right)
	\,d\ell dmdn,
\end{align}
and 
\begin{align}\label{NR12}
	B_1(k) & := isk\iiint\tilde{u}(\ell)\overline{\tilde{u}}(n)\tilde{u}(m) \, \mu_{R,1}^{(2)}\left(k,\ell,n,m\right)\,d\ell dmdn,
	\\
	\label{NR12'}
	B_2(k) & := \iiint\tilde{u}(\ell)\overline{\tilde{u}}(n)\tilde{u}(m) \, \partial_{k}\mu_{R,1}^{(2)}\left(k,\ell,n,m\right)
	\,d\ell dmdn.
\end{align}

\medskip
\subsubsection*{Estimate of \eqref{NR11}-\eqref{NR11'}}
Recalling the definition of $\mathcal{K}_R$ in \eqref{KR} we write
\begin{align}\label{KRid}
	ik\mathcal{K}_{R}(x,k) = \partial_x\mathcal{K}_{R}'(x,k) - \mathcal{K}_{R}''(x,k) 
\end{align}
where
\begin{align}\label{KR1}
	\mathcal{K}_{R}':=\begin{cases}
		\chi_+(x)T(k)\left(m_+(x,k)-1\right)e^{ixk}\\
		+\chi_-(x)\left[\left(m_-(x,-k)-1\right)e^{ixk}-R_-(k)\left(m_-(x,k)-1\right)e^{-ixk}\right] & k\geq0
		\\
		\chi_-(x)T(-k)\left(m_-(x,-k)-1\right)e^{ixk}\\
		+\chi_+(x)\left[\left(m_+(x,k)-1\right)e^{ixk}-R_+(-k)\left(m_+(x,-k)-1\right)e^{-ixk}\right] & k<0
	\end{cases}
\end{align}
and
\begin{align}\label{KR2}
	\mathcal{K}_{R}'':=\begin{cases}
		T(k)\partial_x\left[\chi_+(x)\left(m_+(x,k)-1\right)\right]e^{ixk}
		+\partial_x\left[\chi_-(x)\left(m_-(x,-k)-1\right)\right]e^{ixk}
		\\
		-R_-(k)\partial_x\left[\chi_-(x)\left(m_-(x,k)-1\right)\right]e^{-ixk} & k\geq0
		\\
		T(-k)\partial_x\left[\chi_-(x)\left(m_-(x,-k)-1\right)\right]e^{ixk}
		+\partial_x\left[\chi_+(x)\left(m_+(x,k)-1\right)\right]e^{ixk}
		\\
		-R_-(k)\partial_x\left[\chi_-(x)\left(m_-(x,k)-1\right)\right]e^{-ixk} & k<0
	\end{cases}.
\end{align}

Then, with the definition \eqref{eq:mur11}, integrating by parts in $x$ we have
\begin{align}\label{kmuR11}
	\begin{split}
		- ik \, \mu_{R,1}^{(1)} 
		& = \int\overline{\mathcal{K}_{R}'}(x,k) 
		\partial_x\big[ \mathcal{K}_{M_{1}}(x,\ell)\overline{\mathcal{K}_{M_{2}}(x,n)}\mathcal{K}_{M_{3}}(x,m) \big]\,dx
		\\
		& + \int\overline{\mathcal{K}_{R}''}(x,k)\mathcal{K}_{M_{1}}(x,\ell)
		\overline{\mathcal{K}_{M_{2}}(x,n)}\mathcal{K}_{M_3}(x,m)\,dx
	\end{split}
\end{align}
and therefore
\begin{align}\label{A_1}
	\begin{split}
		-A_1 & = (\sqrt{2\pi})^3s \int\overline{\mathcal{K}}_{R}'(x,k) 
		\partial_x \big[ u_{M_1}(x)\overline{u_{M_2}(x)}u_{M_3}(x) \big]\,dx
		\\
		& +  (\sqrt{2\pi})^3s \int\overline{\mathcal{K}}_{R}''(x,k)u_{M_1}(x)\overline{u_{M_2}(x)}u_{M_3}(x)\,dx 
		=: A_{1,1} + A_{1,2}, 
	\end{split}
\end{align}
where we are using the notation \eqref{decompphi}.

Given $\kappa> 1$ (this condition ensures that we can apply 
the improved local decay estimates for two factors in the cubic term), 
we now observe that by Lemma \ref{lem:pseeasy-1} and, respectively, Lemma \ref{lem:psehard2}
the symbol
\begin{align*}
	\jx^\kappa \mathcal{K}_{R}'(x,k), \qquad \mbox{and, respectively,} \qquad \jx^\kappa \mathcal{K}_{R}''(x,k)
\end{align*}
give rise to $L^2$-bounded operators if $\gamma > \kappa + 3/2$ (so $\gamma>\frac{5}{2}$),
respectively, $\gamma > \max(3/2,\kappa)$.
Note that we also need to use \eqref{estiTR} and the fact that $\partial_x\chi_-$ is compactly supported.
Then, for the term $A_{1,1}$, in the case where $\partial_x$ hits $u_{M_1}$ (the other cases are identical), 
we pick $\kappa= 1+\zeta/2$ (recall $\zeta=\gamma-5/2$). 
and estimate using the standard decay estimates \eqref{eq:linearpoinwiseH},
and the improved local decay estimates \eqref{locdecL2}, \eqref{locdecinftyng} and \eqref{eq:interp}:
\begin{align*}
	& s {\big\| \jx^{-\kappa}
		\partial_x\overline{u}_{M_{1}}(s)u_{M_{2}}(s)u_{M_{3}}(s) \big\|}_{L^{2}}
	\\
	& \lesssim {\| \jx^{-1}s\partial_xu_{M_{1}}(s) \|}_{L^2_x}
	{\| \jx^{-\zeta/2}u_{M_{2}}(s) \|}_{L^\infty_x} {\| u_{M_{3}}(s) \|}_{L^\infty_x}
	\\
	& \lesssim s^{-1-\frac{1}{8}\zeta} { {\| \tilde{f} \|}_{H^1}^{3} 
		\lesssim s^{-1-\frac{1}{8}\zeta+3\alpha} {\| u \|}_{X_T}^{3}.}
\end{align*}
Upon integration in $s$, 
this is controlled by the right-hand side of \eqref{weightmainconc} as desired provided $2\alpha\leq \zeta/8$.
%
%
To estimate $A_{1,2}$ we pick $\kappa=2+\zeta$ and use again \eqref{locdecinftyng} and \eqref{eq:interp}:
\begin{align*}
	{\| A_{1,2} \|}_{L^2} & \lesssim s {\big\| \jx^{-\kappa} \overline{u}_{M_{1}}(s)u_{M_{2}}(s)u_{M_{3}}(s) \big\|}_{L^{2}}
	\\
	& \lesssim s {\| \jx^{-1} u_{M_{1}}(s) \|}_{L^\infty_x}
	{\| \jx^{-1}u_{M_{2}}(s) \|}_{L^\infty_x} {\| \jx^{-\zeta}u_{M_{3}}(s) \|}_{L^\infty_x}
	\\
	& \lesssim s s^{-\frac{3}{2}} s^{-\frac{1}{2}-\frac{\zeta}{4}} {\| \tilde{f} \|}_{H^{1}}^{3} 
	\lesssim s^{-1-\frac{\zeta}{4}+3\alpha} {\| u \|}_{X_T}^{3},
\end{align*}
which suffices, again provided $\alpha\leq \zeta/8$.

We can estimate in a similar way also the term $A_2$ in \eqref{NR11'}.
Since $\jx^{\zeta/2} \partial_k \mathcal{K}_{R}(x,k)$ gives a bounded 
pseudo-differential operator on $L^2$
by \eqref{eq:pdomkd} and \eqref{eq:mPDO+-1}, and the decay \eqref{eq:interp}, we have
\begin{align*}
	\left\Vert A_{2}\right\Vert _{L^{2}} & \lesssim\left\Vert \jx^{-\zeta}u_{M_{1}}\overline{u}_{M_{2}}u_{M_{3}}(s)
	\right\Vert_{L^2_x}
	\\
	& \lesssim \left\Vert \jx^{-\zeta}u_{M_{1}}(s)\right\Vert_{L^\infty_x}
	\left\Vert u_{M_{2}}(s)\right\Vert _{L^\infty_x} \left\Vert u_{M_{3}}(s)\right\Vert_{L^2_x}
	\\
	& { \lesssim s^{-1-\frac{\zeta}{8}} {\| \tilde{f} \|}_{L^2} {\| \tilde{f} \|}_{H^1}^2
	}
	\lesssim s^{-1-\frac{1}{8}\zeta+2\alpha} {\| u \|}_{X_T}^{3}
\end{align*}
which is again sufficient under the same condition above.

\medskip
\subsubsection*{Estimate of \eqref{NR12}-\eqref{NR12'}}
For this term we are going to use the localization in $x$ provided by $\mathcal{K}_R(x,n)$,
similarly to how we used the localization of $\mathcal{K}_R'$ and $\mathcal{K}_R''$ above.

First, from \eqref{Ksing}-\eqref{Ksing-}, we have the following analogue of \eqref{KRid} for $\mathcal{K}_S$:
\begin{align}\label{KSid}
	ik\mathcal{K}_{S}(x,k) = \partial_x\mathcal{K}_{S}'(x,k) - \mathcal{K}_{S}''(x,k) 
\end{align}
with
\begin{align}\label{KS1}
	\mathcal{K}_S'(x,k) & = 
	\begin{cases}
		\chi_+(x) T(k) e^{ikx} + \chi_-(x) (e^{ikx} - R_-(k) e^{-ikx} ) & k\geq0
		\\
		\chi_+(x) (e^{ikx} - R_+(-k)e^{-ikx}) + \chi_-(x) (T(-k) e^{ikx}) & k<0
	\end{cases},
	\\
	\label{KS2}
	\mathcal{K}_S''(x,k) &:= 
	\begin{cases}
		\partial_x \chi_+(x) T(k) e^{ikx} + \partial_x \chi_-(x) (e^{ikx} - R_-(k) e^{-ikx} ) & k\geq0
		\\
		\partial_x \chi_+(x) (e^{ikx} - R_+(-k)e^{-ikx}) + \partial_x \chi_-(x) (T(-k) e^{ikx}) & k<0
	\end{cases}.
\end{align}
Then, in analogy with \eqref{kmuR11} and \eqref{A_1} we have
\begin{align}\label{kmuR12a}
	- ik \, \mu_{R,1}^{(2)} 
	&:= \int \overline{\mathcal{K}_{S}'}(x,k) 
	\partial_x\big[ \mathcal{K}_{M_{1}}(x,\ell)\overline{\mathcal{K}_{M_{2}}(x,n)}\mathcal{K}_{R}(x,m) \big]\,dx
	\\
	\label{kmuR12b}
	& + \int \overline{\mathcal{K}_{S}''}(x,k)\mathcal{K}_{M_{1}}(x,\ell)
	\overline{\mathcal{K}_{M_{2}}(x,n)} \mathcal{K}_{R}(x,m)\,dx,
\end{align}
and \eqref{NR12} can be written as 
\begin{align*}
	-B_1 & = s (\sqrt{2\pi})^3\int \overline{\mathcal{K}_{S}'}(x,k) 
	\partial_x \big[u_{M_1}(x) \overline{u_{M_2}}(x) u_R(x)\big] \, dx
	\\
	& + s(\sqrt{2\pi})^3 \int \overline{\mathcal{K}_{S}''}(x,k) 
	u_{M_1}(x) \overline{u_{M_2}}(x) u_R(x) \, dx =: B_{1,1} + B_{1,2}.
\end{align*}
Observe that
\begin{align}\label{mapKS}
	{\Big\| \int \overline{\mathcal{K}_{S}'}(x,k) F(x) \, dx \Big\|}_{L^2_k} \lesssim {\| F \|}_{L^2},
	\qquad 
	{\Big\| \int \overline{\mathcal{K}_{S}''}(x,k) F(x) \, dx \Big\|}_{L^2_k} \lesssim {\| \jx^{-\kappa} F \|}_{L^2},
\end{align}
for any $\kappa$ since $\partial_x\chi_\pm$ is compactly supported.
Then, we can estimate
\begin{align*}
	& {\| B_{1,1} \|}_{L^2_k}  \lesssim 
	{\| s \partial_x \big[u_{M_1} \overline{u_{M_2}} \big] u_R \|}_{L^2} 
	+ {\| u_{M_1} \overline{u_{M_2}} \, s\partial_x u_R \|}_{L^2} 
	\\
	& \lesssim \Big( {\| \jx^{-1} s \partial_x u_{M_1} \|}_{L^2} {\| \jx^{-\frac{\zeta}{2}}u_{M_2} \|}_{L^\infty}
	+ {\| \jx^{-\frac{\zeta}{2}} u_{M_1} \|}_{L^\infty} {\| \jx^{-1} s \partial_x u_{M_2} \|}_{L^2}\Big)
	{\| \jx^{\frac{\zeta}{2}+1} u_R \|}_{L^\infty} 
	\\
	& + {\| u_{M_1} \|}_{L^\infty} {\| \jx^{-\frac{\zeta}{2}} u_{M_2} \|}_{L^\infty} 
	{\|  s\jx^{\frac{\zeta}{2}} \partial_x u_R \|}_{L^2} 
	\\
	& \lesssim s^{-1-\frac{1}{8}\zeta} 
	{ {\| \tilde{f} \|}_{H^1}^3}
	\lesssim s^{-1 - \frac{1}{8}\zeta + 3\alpha} {\| u \|}_{X_T}^{3} 
\end{align*}
having used the $L^\infty$ estimates \eqref{eq:linearpoinwiseH}, \eqref{eq:interp} and \eqref{regularinfty},
and the $L^2$ estimates \eqref{locdecL2}-\eqref{locdecL2R}.
Similarly, using \eqref{mapKS}, \eqref{locdecinftyng} and \eqref{locdecinftyRng}, we can estimate
\begin{align*}
	{\| B_{1,2} \|}_{L^2_k} & \lesssim s {\| \jx^{-3} u_{M_1} \overline{u_{M_2}} u_R \|}_{L^2}
	\\ & \lesssim s {\| \jx^{-1} u_{M_1} \|}_{L^\infty} {\| \jx^{-1} u_{M_2} \|}_{L^\infty} {\|\jx^{-1} u_R \|}_{L^\infty} 
	\lesssim s^{-\frac{5}{4}} {\| \tilde{f} \|}_{H^1}^3
\end{align*}
which again suffices.

The last term \eqref{NR12'} is easier to treat since there is no extra power of $s$ in front 
of the nonlinearity. One can use the analogue of \eqref{mapKS} 
\begin{align}\label{mapKS'}
	{\Big\| \int \partial_k \overline{\mathcal{K}_{S}}(x,k) F(x) \, dx \Big\|}_{L^2} \lesssim {\| \jx F \|}_{L^2},
\end{align}
with $F = u_{M_1}u_{M_2}u_R$, the $L^2$ estimate for the regular part \eqref{regularL2},
and \eqref{eq:interp}, to bound
\begin{align*}
	{\| \jx u_{M_1}u_{M_2}u_R \|}_{L^2_x} \lesssim {\|\jx^{-\frac{\zeta}{2}} u_{M_1}\|}_{L^\infty} {\| u_{M_2} \|}_{L^\infty}
	{\| \jx^{\frac{\zeta}{2}+1} u_R \|}_{L^2} 
	\lesssim  s^{-1-\frac{1}{8}\zeta} { {\| \tilde{f} \|}_{H^1}^3}.
\end{align*}

\medskip
\subsubsection{Estimate of $\mathcal{N}_{R,2}$}
Finally, we estimate $\mathcal{N}_{R,2}$ which is the nonlinear term associated with the measure in \eqref{muR2'}.
The key observation here is that, according to the decomposition 
$\mathcal{K}_S = \chi_+ \mathcal{K}_+ + \chi_- \mathcal{K}_-$ in \eqref{Ksing},
and the definition of $\mu_S$ in \eqref{muSpm} we can write \eqref{muR2'} as
\begin{align}\label{muR2''}
	\begin{split}
		\mu_{R,2}(k,\ell,m,n)= \sum 
		\int \overline{\chi_{\e_0}(x)\mathcal{K}_{\e_0}(x,k)}
		\chi_{\e_1}(x)\mathcal{K}_{\e_1}(x,\ell) 
		\overline{\chi_{\e_2}(x)\mathcal{K}_{\e_2}(x,m)}\chi_{\e_3}(x)\mathcal{K}_{\e_3}(x,n)\,dx
	\end{split}
\end{align}
where the sum is over
\begin{align*}
	\e_0,\e_1,\e_2,\e_3 \in \{+,-\} \quad \mbox{with} \quad (\e_0,\e_1,\e_2,\e_3) \neq (+,+,+,+), \,(-,-,-,-).
\end{align*}
Note that, since $\chi_+\chi_-$ is compactly supported, the integrand in \eqref{muR2''} is compactly supported in $x$.

Without loss of generality, it suffices to estimate the contribution to $\mathcal{N}_{R,2}$ from the measure
\begin{align}\label{muR2est0}
	\mu_{R,2}^{(0)} := \int \overline{\chi_+(x)\mathcal{K}_+(x,k)}
	\chi_-(x)\mathcal{K}_-(x,\ell)\overline{\mathcal{K}_S(x,m)}\mathcal{K}_S(x,n)\,dx;
\end{align}
all other contributions from the sum in \eqref{muR2''} can be estimated identically.
We then want to estimate the $L^2$ norm of $\partial_k$ of the expression
\begin{align}\label{muR2est1}
	\begin{split}
		\iiint e^{is(-k^2+\ell^2-m^2+n^2)}
		\tilde{f}(s,\ell)\overline{\tilde{f}(s,m)}\tilde{f}(s,n)\mu_{R,2}^{(0)}(k,\ell,m,n)\,dndmd\ell
		\\ 
		= e^{-isk^2} (\sqrt{2\pi})^3\int \chi_+(x)\mathcal{K}_+(x,k) \chi_-(x) u_{-}(x) \overline{u_S(x)} u_S(x)\, dx
	\end{split}
\end{align}
where we use the same notation as in \eqref{decompphi} to define 
\begin{align*} 
	u_{-}(x) := \frac{1}{\sqrt{2\pi}}\int \mathcal{K}_-(x,k) \tilde{u}(k) \, dk. 
\end{align*}
%

%

Applying $\partial_k$ to \eqref{muR2est1} will give two contributions, one when the derivative hits the exponential
and the other one when it hits $\mathcal{K}_+$.
We disregard this second one since it is much easier to estimate.
We are then left with estimating the $L^2$ norm of 
\begin{align}\label{muR2est2}
	isk \int \chi_+(x)\mathcal{K}_+(x,k) \chi_-(x) u_{-}(x) \overline{u_S(x)} u_S(x)\, dx.
\end{align}
Similarly to what we did before in \eqref{KRid} and \eqref{KSid} we 
first convert the factor of $k$ into an $x$ derivative by writing
\begin{align}\label{Ksing+id}
	ik \mathcal{K}_+(x,k) = \partial_x\mathcal{K}_{+}'(x,k), 
	\qquad
	\mathcal{K}_+'(x,k) = 
	\begin{cases}
		T(k) e^{ikx} & k\geq0
		\\
		e^{ikx} - R_+(-k)e^{-ikx} & k<0
	\end{cases};
\end{align}
then, we integrate by parts in $x$ in \eqref{muR2est2} to get
\begin{align}\label{muR2est3}
	s\int \mathcal{K}_+'(x,k) \partial_x\big[ \chi_+(x)\chi_-(x) u_{-}(x) \overline{u_S(x)} u_S(x) \big]\, dx,
\end{align}
and eventually estimate this in $L^2_k$:
using that $\chi_+\chi_-$ is compactly supported and distributing the derivative we have
\begin{align}\label{eq:lastterm}
	\begin{split}
		& s {\Big\| \int \mathcal{K}_+'(x,k) \partial_x\big[ \chi_+(x)\chi_-(x)u_{-}(x)\overline{u_S(x)}u_S(x) \big]\,dx\Big\|}_{L^2}
		\\
		& \lesssim s{\| \partial_x\big[ \chi_+\chi_- u_- \overline{u_S} u_S \big] \|}_{L^2}
		\\
		& \lesssim s {\| \jx^{-1} u_{-}\|}_{L^\infty} {\| \jx^{-1} u_S \|}_{L^\infty}^2
		+ {\| \jx^{-1}  s \partial_x u_{-}\|}_{L^2} {\| \jx^{-1} u_S \|}_{L^\infty}^2
		\\ 
		& + {\| \jx^{-1} u_{-}\|}_{L^\infty} {\| \jx^{-1} u_S \|}_{L^\infty} {\| \jx^{-1} s\partial_x u_S\|}_{L^2}
		\\
		& \lesssim \big(s \cdot s^{-\frac{9}{4}} + s^{-\frac{3}{2}} \big) {\big\| \tilde{f} \big\|}_{H^1}^3
	\end{split}
\end{align}
where the last inequality follows from  estimates \eqref{locdecinftyng} and \eqref{locdecL2},
also applied to $u_-$.
This concludes the estimate of $\mathcal{N}_{R,2}$, thus of the regular part of the nonlinearity $\mathcal{N}_R$,
and gives Lemma \ref{pro:weightmain} in the case of a generic potential.

\medskip
\subsection{Non-generic case}
Next, we look the non-generic case. We assume that $\tilde{f}(0)=0$.
and assume $\jx^\gamma V(x)\in L^1$ with $\gamma>7/2$. Let $\zeta=\gamma-7/2$.
The significant difference from the generic case is that here we can only apply Lemma \ref{lem:locl2ng},
instead of Lemma \ref{lem:localEn},
to handle the localized $L^2$ estimates for the derivatives. 
This will introduce extra growing factors of $t^{1/4}$
which need to be compensated by additional localization (giving better improved $L^\infty$ decay estimates),
thus requiring the stronger assumption on $\gamma$.

Most of the estimates can be done exactly in the same way as in the generic case
with only a few exceptions; we detail the differences in these cases below. 
We use the same notation above, starting from \eqref{muSpm}-\eqref{Nast}

%

\subsubsection{Estimate of $\mathcal{N}_{R,1}$}

\subsubsection*{Estimate of \eqref{NR11}-\eqref{NR11'}}
Recall the notation \eqref{KR1}-\eqref{KR2} and \eqref{A_1}. 
By Lemma \ref{lem:pseeasy-1} and \ref{lem:psehard2}, we have that the symbol
$\jx^\kappa \mathcal{K}_{R}'(x,k)$
with $\gamma>\kappa+3/2$, and the symbol
$\jx^\kappa \mathcal{K}_{R}''(x,k)$
for $\gamma> \max(3/2,\kappa)$, 
give rise to $L^2$-bounded pseudo-differential operators. Since $\gamma>7/2$, we can pick $\kappa=2+\zeta/2$.
Then, for the term $A_{1,1}$ in the case where $\partial_x$ hits $u_{M_1}$ 
(the other cases are identical), 
we can use 
\eqref{locdecinftyng}, \eqref{eq:interp} and 
\eqref{eq:locl2ngf}, to obtain the bound
\begin{align*}
	& s {\big\| \jx^{-\kappa}
		\partial_x\overline{u}_{M_{1}}(s)u_{M_{2}}(s)u_{M_{3}}(s) \big\|}_{L^{2}}
	\\
	& \lesssim {\| \jx^{-1}s\partial_xu_{M_{1}}(s) \|}_{L^2_x}
	{\| \jx^{-1}u_{M_{2}}(s) \|}_{L^\infty_x} {\|\jx^{-\zeta/2} u_{M_{3}}(s) \|}_{L^\infty_x}
	\\
	& \lesssim s^{-1-\zeta/8} {\| \tilde{f} \|}_{H^{1}}^{3}
	\lesssim s^{-1-\zeta/8+3\alpha} {\| u \|}_{X_T}^{3},
\end{align*}
which, after integrating in $s$, is controlled by the right-hand side of \eqref{weightmainconc} 
as desired provided $2\alpha \leq \zeta/8$.
The estimates for the terms $A_{1,2}$ and  $A_{2}$ done above for the generic case 
apply verbatim here.



\subsubsection*{Estimate of \eqref{NR12}-\eqref{NR12'}}
As the generic case,  we can estimate
\begin{align*}
	& {\| B_{1,1} \|}_{L^2_k}  \lesssim 
	{\| s \partial_x \big[u_{M_1} \overline{u_{M_2}} \big] u_R \|}_{L^2} 
	+ {\| u_{M_1} \overline{u_{M_2}} \, s\partial_x u_R \|}_{L^2} 
	\\
	& \lesssim \Big( {\| \jx^{-1} s \partial_x u_{M_1} \|}_{L^2} {\|\jx^{-\frac{\zeta}{2}} u_{M_2} \|}_{L^\infty}
	+ {\|\jx^{-\frac{\zeta}{2}} u_{M_1} \|}_{L^\infty} {\| \jx^{-1} s \partial_x u_{M_2} \|}_{L^2}\Big)
	{\| \jx^{1+\frac{\zeta}{2}} u_R \|}_{L^\infty} 
	\\
	& + {\| \jx^{-1}u_{M_1} \|}_{L^\infty} {\|  \jx^{-\frac{\zeta}{2}}u_{M_2} \|}_{L^\infty} 
	{\|\jx^{1+\frac{\zeta}{2}} s \partial_x u_R \|}_{L^2} 
	\lesssim s^{-1-\frac{\zeta}{8}+3\alpha} {\| u \|}_{X_T}^{3} 
\end{align*}
having used \eqref{eq:linearpoinwiseH}, \eqref{eq:interp}, \eqref{locdecinftyng}, 
\eqref{regularinfty}, \eqref{eq:locl2ngf} and \eqref{eq:loc2ngR} with $\gamma > 7/2 + \zeta/2$.

\subsubsection{Estimate of $\mathcal{N}_{R,2}$}
We proceed as in the generic case and reduce things to estimate \eqref{muR2est3}.
From \eqref{eq:lastterm} we get
\begin{align*}
	{\| \eqref{muR2est3} \|}_{L^2} 
	& \lesssim s {\| \jx^{-1} u_{-}\|}_{L^\infty} {\| \jx^{-1} u_S \|}_{L^\infty}^2
	+ {\| \jx^{-1}  s \partial_x u_{-}\|}_{L^2} {\| \jx^{-1} u_S \|}_{L^\infty}^2
	\\ & + {\| \jx^{-1} u_{-}\|}_{L^\infty} {\| \jx^{-1} u_S \|}_{L^\infty} {\| \jx^{-1} s\partial_x u_S\|}_{L^2}
	\\ 
	& \lesssim \big(s \cdot s^{-\frac{9}{4}} + s^{-\frac{5}{4}} \big) {\big\| \tilde{f} \big\|}_{H^1}^3
\end{align*}
where the last inequality follows from the  estimates \eqref{locdecinftyng} and \eqref{eq:locl2ngf},
also applied to $u_-$.
This concludes the estimate of $\mathcal{N}_{R,2}$, thus of the regular part of the nonlinearity $\mathcal{N}_R$,
and gives Proposition \ref{pro:weightmain} in the non-generic case.

\bigskip
\section{Pointwise bound for the profile}\label{sec:pointwisebound}
The aim of this section is 
to show how to obtain the $L^{\infty}$ bound for the distorted Fourier transform of the profile,
and the asymptotics \eqref{mainasy}.
The main result is the following:

\begin{prop}\label{proasy}
	For $1 \leq t \leq T$, and $|k|\gtrsim t^{-3\alpha}$, we have 
	\begin{align}\label{secasODE}
		i \partial_t \wt{f}(t,k) = \frac{1}{2t}{|\wt{f}(t,k)|}^2 \wt{f}(t,k) + 
		\mathcal{O}({\| u \|}_{X_T}^3 t^{-1-\rho}) ,
	\end{align}
	for some $\rho>0$.
	Defining the modified profile
	\begin{align}
		\label{secasmod}
		w (t,k):= \exp\Big(\frac{i}{2} \int_0^t |\wt{f}(s,k)|^2 \, \frac{ds}{1+s} \Big) \wt{f}(t,k),
	\end{align}
	for every  $1<t_1 < t_2 < T$ we have, see \eqref{datasmall},
	\begin{align}
		\label{secas10}
		\big| w(t_1,k) - w(t_2,k) | \lesssim ( \eta  + {\| u \|}_{X_T}^{3}) \, t_1^{-\rho/2}.
	\end{align}
\end{prop}

Before getting to the proof of the  proposition above, we explain how to use it to close 
our bootstrap argument, obtain a global solution, and complete the proof of the main theorem.

\smallskip
\subsection*{Bootstrap argument and proof of Theorem \ref{thm:main1}}
Recall the definition of $X_T$ in \eqref{X_T}, 
and assume that for $\eta_1 := \eta^{2/3}$. We make the a priori assumption
\begin{align}\label{apriori0}
	{\| u \|}_{X_T} \leq \eta_1.
\end{align}
Proposition \ref{pro:weightmain} implies, see also \eqref{eq:weiF},
\begin{align}\label{apriori1}
	{\big\| \partial_{k}\tilde{f}(t) \big\|}_{L_{k}^{2}}
	\leq \eta + C |t|^{\alpha} \eta_1^3 \leq 2\eta |t|^{\alpha},
\end{align}
provided $\eta\leq \epsilon_0$ small enough.
Then observe that, by our assumptions, 
$\tilde{f}\left(t,0\right)=0$ for all $t \in [0,T]$ so that, for $\left|k\right|\lesssim\left|t\right|^{-3\alpha}$, 
we have
\begin{align}\label{apriori2}
	|\tilde{f}\left(t,k\right)| \leq \left|\int_{0}^{k} \partial_{\eta}\tilde{f}\left(t,\eta\right)\,d\eta\right|
	\leq \sqrt{\left|k\right|} \,
	{\| \partial_{k}\tilde{f} \|}_{L^{2}} \lesssim |t|^{-\frac{\alpha}{2}} 2\eta . 
\end{align}
In particular, the low-frequency part of $\tilde{f}(t,k)$ goes to zero as $t\rightarrow \infty$,
and we can reduce matters to considering only $|k| \geq |t|^{-3\alpha}$.
Under this latter condition, using \eqref{secasmod} and \eqref{secas10} in Proposition \ref{proasy}
we deduce that
\begin{align}\label{apriori3}
	|\tilde{f}\left(t,k\right)| = |w(t,k)| \leq |w(1,k)| + C \eta \lesssim \eta. 
\end{align}
\eqref{apriori1}-\eqref{apriori3} imply ${\| u \|}_{X_T} \leq \eta_1/2$, improving on \eqref{apriori0},
so that a standard continuation argument gives us a global solution which is bounded in the $X_\infty$ norm. 
Eventually, we obtain the following:

\begin{cor} 
	Let $u = e^{it(\partial_{xx}+V)}f$ be the global-in-time solution obtained above.
	With the same notation of Proposition \ref{proasy}, we have that
	$w(t)$ is a Cauchy sequence in time with values in $L^\infty$.
	Letting $W_{+\infty} := \lim_{t \rightarrow \infty} w(t)$ we obtain the asymptotics \eqref{mainasy}.
\end{cor}

\subsection*{Proof of Proposition \ref{proasy}}
Our starting point is again \eqref{eq:expandf}, which we rewrite in the form
\begin{equation}\label{ODE}
	\partial_{t}\widetilde{f}(t,k)=-i\iiint e^{it\left(-k^{2}+\ell^{2}-m^{2}+n^{2}\right)}
	\widetilde{f}(t,\ell)\overline{\tilde{f}(t,m)}\widetilde{f}(t,n)\,\mu(k,\ell,m,n) \, d\ell dmdn
\end{equation}
where $\mu$ is decomposed as in \eqref{decM1} so that, with the notation \eqref{Nast},
\begin{equation}\label{ODE'}
	\partial_{t}\widetilde{f}(t,k) = \mathcal{N}_{S}\left(t,k\right) + \mathcal{N}_{R}\left(t,k\right).
\end{equation}

\medskip
\subsubsection*{Asymptotics of the regular part}
We begin with the regular part which is straightforward. 
From the estimates in Section \ref{sec:estimateregular}, we have
\[
\left\Vert \mathcal{N}_{R}\left(t,k\right)\right\Vert _{L^{2}} 
+ \left\Vert \partial_{k} \mathcal{N}_{R}\left(t,k\right)\right\Vert _{L^{2}}	
\lesssim {\| u \|}_{X_T}^3 \, t^{-1+3\alpha-\frac{\zeta}{8}},
\]
so that, by Sobolev embedding,
\begin{equation}\label{ODEregular}
	\left| \mathcal{N}_{R}\left(t,k\right)\right|\lesssim 
	{\| u \|}_{X_T}^3 \, t^{-1+3\alpha-\frac{\zeta}{8}}.
\end{equation}
This can be absorbed into the remainder in \eqref{secasODE} provided $\alpha$  is small enough.

\medskip
\subsubsection*{Asymptotics of the singular part}
The first step is the following stationary phase-type lemma: 

\begin{lem}
	\label{AsLem1}
	For $k,t \in \R$, consider the integral expression
	\begin{align}\label{I1}
		\begin{split}
			I[g_1,g_2,g_3](t,k) & = \iiint e^{it \Phi(k,p,m,n)}
			g_1(\epsilon_1(\epsilon_0 k - p + \epsilon_2 m - \epsilon_3 n)) \overline{g_2(m)} 
			g_3(n) \frac{\widehat{\phi}(p)}{p} \, dm\, dn\, dp
			\\
			\Phi(k,p,m,n) & = -k^2 + (\epsilon_0 k - p + \epsilon_2 m - \epsilon_3 n)^2 - m^2 + n^2.
		\end{split}
	\end{align}
	for an even bump function $\phi \in C_0^\infty$ with integral $1$, 
	and with $g:=(g_1,g_2,g_3)$ satisfying
	\begin{align}
		\label{AsLem1as}
		{\| g(t) \|}_{L^\infty} + {\| \langle k \rangle g(t) \|}_{L^2} + \langle t \rangle^{-\alpha} {\| g'(t) \|}_{L^2} \leq 1,
	\end{align}
	for some  $\alpha > 0 $ small enough. 
	Then, for any $t \in \R$,
	\begin{align}\label{I2}
		\begin{split}
			I[g_1,g_2,g_3](t,k) = \frac{\pi}{|t|} e^{-itk^2} \int e^{it(-p+\epsilon_0 k)^2} g_1(\epsilon_1(-p+\epsilon_0 k))
			\overline{g_2(\epsilon_2(-p+\epsilon_0 k))}
			\\ \times  g_3(\epsilon_3 (-p+\epsilon_0 k)) \mathrm{p.v.} \frac{\widehat{\phi}(p)}{p} \,dp
			+ \mathcal{O}(|t|^{-1-\rho})
		\end{split}
	\end{align}
	for some $\rho>0$.
	
	For 
	$|k| \geq |t|^{-3\alpha}$ we can further simplify \eqref{I2} to
	\begin{align}\label{I3}
		\begin{split}
			I[g_1,g_2,g_3](t,k) & =-i\frac{\pi}{|t|}
			\sqrt{\frac{\pi}{2}} \cdot \mathrm{sign}(\e_0k t) \cdot g_1(\epsilon_1\epsilon_0 k) 
			\overline{g_2(\epsilon_2\epsilon_0 k))} g_3(\epsilon_3\epsilon_0 k)+  \mathcal{O}(|t|^{-1-\rho}).
		\end{split}
	\end{align}
\end{lem}

\begin{proof}
	For a proof of \eqref{I2} see Lemma 5.1 in \cite{GPR}.
	The further simplification to \eqref{I3} is obtained by using the following:
	for any function $g$ such that
	\begin{align*}
		{\| g \|}_{L^\infty} + \langle t \rangle^{-\alpha} {\| g^\prime \|}_{L^2} \leq 1,
	\end{align*}
	for $\alpha>0$ small enough, and $|K| > t^{-3\alpha}$, one has
	\begin{align}\label{I4}
		I(t,K) = e^{-itK^{2}} \mathrm{p.v.}  \int e^{it q^2} g(q) \frac{\hat{\phi}(q-K)}{q-K} \,dq
		= i \sqrt{\frac{\pi}{2}} \, \mathrm{sign}(tK) \, g(K) \, + \, \mathcal{O}(t^{-\rho}).
	\end{align}
	Changing variables $p \mapsto \e_0k - q$ in \eqref{I2}, taking $K=\e_0k$, using 
	that $\hat{\phi}(0) = 1/\sqrt{2\pi}$, and applying \eqref{I4} gives \eqref{I3}.
	
	To prove \eqref{I4} 
	we first decompose (smoothly) the support of $I(t,K)$ into the regions $|q-K| \geq t^{-1/3}$
	and $|q-K| \leq t^{-1/3}$.
	When $|q-K| \geq t^{-1/3}$ we obtain a remainder term by looking separately at the cases 
	$|q| \geq t^{-2/5}$ (where we can integrate by parts in $q$) and $|q| \leq t^{-2/5}$
	(where the size of the integral is bounded by $|K|^{-1} t^{-2/5}$).
	In the remaining region with $|q-K| \leq t^{-1/3}$ we first approximate $g(q)$ by $g(K)$,
	then change variables and apply integration by parts (note that $t|K|\gg t|q|$), 
	to see that
	\begin{align}
		\begin{split}
			I(t,K) & = g(K) \, \mathrm{p.v.}  \int e^{2i q tK}  e^{it q^2} \frac{\hat{\phi}(q)}{q} \varphi_{\leq 0}(qt^{1/3}) 
			\,dq + \mathcal{O}(t^{-\rho})
			\\ 
			& = g(K)\hat{\phi}(0) \int e^{2i q tK} \frac{1}{q} \, dq + \mathcal{O}(t^{-\rho})
			\\
			& = g(K)\hat{\phi}(0) \, i \pi \, \mathrm{sign}(2tK) + \mathcal{O}(t^{-\rho}) \, . 
		\end{split}
	\end{align}
	For the last identity we used the formula 
	$\widehat{(\mathrm{p.v.} 1/q)}(\xi) = -i\sqrt{\pi/2}\, \mathrm{sign}(\xi)$.
\end{proof}

\medskip
\begin{proof}[Proof of  \eqref{secasODE} and \eqref{secas10}]
	
	
	We recall here for convenience the formulas \eqref{Nast} and \eqref{muSstr1}: 
	$\mathcal{N}_S = \mathcal{N}_+ + \mathcal{N}_-$ with
	\begin{align}\label{NSasy}
		& \mathcal{N}_\iota(t,k) = \frac{1}{(2\pi)^{2}} \iiint e^{it(-k^2+\ell^2-m^2+n^2)} 
		\tilde{f}(t,\ell)\overline{\tilde{f}(t,m)}\tilde{f}(t,n) \mu_\iota(k,\ell,m,n)\,dndmd\ell,
		\\
		\label{muSstr1'}
		& \mu_\iota(k,\ell,m,n) 
		= \sum_{\e_0,\e_1,\e_2,\e_3 \in \{+,-\}}\overline{a_\iota^{\e_0}(k)} \,a_\iota^{\e_1}(\ell) \,
		\overline{a_\iota^{\e_2}(m)} \,a_\iota^{\e_3}(n) \, \sqrt{2\pi} \,
		\widehat{\varphi_\iota}(\e_0k-\e_1\ell+\e_2m-\e_3n),
	\end{align}
	with $\varphi_\iota$ defined in \eqref{eq:hatphi+-}.
	In what follows we disregard $\varpi$ which gives faster decaying terms like $\mathcal{N}_R$ above. 
	
	We change variables $\ell \mapsto p = \e_0k-\e_1\ell+\e_2m-\e_3n$,
	and apply Lemma \ref{AsLem1} to the $\mathrm{p.v.}$ contribution in \eqref{NSasy}, 
	which we denote $\mathcal{N}_\iota^{\mathrm{p.v.}}$, 
	obtaining
	\begin{align}
		\begin{split}
			& \mathcal{N}_\iota^{\mathrm{p.v.}}(t,k) = \frac{1}{(2\pi)^{3/2}} \sum_{\e_0,\e_1,\e_2,\e_3 \in \{+,-\}}
			\iiint e^{it(-k^2+(\e_0k-p+\e_2m-\e_3n)^2-m^2+n^2)} 
			\\ & \times \overline{a_\iota^{\e_0}(k)} 
			\, a_\iota^{\e_1}(\e_1(\e_0k-p+\e_2m-\e_3n)) \,
			\overline{a_\iota^{\e_2}(m)} \,a_\iota^{\e_3}(n) \,
			\\
			& \times
			\tilde{f}(t,\e_1(\e_0k-p+\e_2m-\e_3n)) 
			\overline{\tilde{f}(t,m)}\tilde{f}(t,n) \, 
			\mathrm{p.v.}\, \iota \frac{\hat{\zeta}(p)}{ip} \,dn dm dp
			\\
			& = - \iota \frac{1}{(2\pi)^{3/2}} \frac{\pi}{|t|}\sqrt{\frac{\pi}{2}}
			\sum_{\e_0,\e_1,\e_2,\e_3 \in \{+,-\}} \mathrm{sign}(\e_0k t) 
			\\
			& \times \overline{a_\iota^{\e_0}(k)} 
			\, a_\iota^{\e_1}(\e_1\e_0k) \,
			\overline{a_\iota^{\e_2}(\e_2\e_0k)} \,a_\iota^{\e_3}(\e_3\e_0k) \,
			\tilde{f}(t,\e_1\e_0k) \overline{\tilde{f}(t,\e_2\e_0k)}\tilde{f}(t,\e_3\e_0k) 
			+  \mathcal{O}( {\| u \|}_{X_T}^{3} s^{-1-\rho}).
		\end{split}
	\end{align}
	We set
	\begin{align}
		\label{N+}
		\mathcal{N}^+[f](k) := \big|T(k)\wt{f}(k) + R_+(k)\wt{f}(-k)\big|^2 \big( T(k)\wt{f}(k) + R_+(k)\wt{f}(-k) \big)
	\end{align}and
	\begin{align}
		\label{N-}
		\begin{split}
			\mathcal{N}^-[f](k) & := \big|T(k)\wt{f}(-k) + R_-(k)\wt{f}(k)\big|^2 \big( T(k)\wt{f}(-k) + R_-(k)\wt{f}(k) \big).
		\end{split}
	\end{align}
	Using the formulas for the coefficients in \eqref{Ksing+'}-\eqref{Ksing-'} and \eqref{TRid},
	we calculate, approximately up to $\mathcal{O}( {\| u \|}_{X_T}^{3} t^{-1-\rho})$,
	\begin{align}\label{Npv+}
		\begin{split}
			\mathcal{N}_+^{\mathrm{p.v.}}(t,k)  & \approx 
			\frac{\pi}{|t|} \frac{1}{(2\pi)^{3/2}} \Big[-\sqrt{\frac{\pi}{2}} {T(-k)} \mathcal{N}^+[f](k) \mathbf{1}_+(k)
			\\ & \Big(\sqrt{\frac{\pi}{2}} |\wt{f}(k)|^2 \wt{f}(k)
			- \sqrt{\frac{\pi}{2}}  R_+(k) \mathcal{N}^+[f](-k) \Big) \mathbf{1}_-(k) \Big] 
		\end{split}
	\end{align}
	and
	\begin{align}
		\begin{split}\label{Npv-}
			\mathcal{N}_-^{\mathrm{p.v.}}(t,k)  \approx- \frac{\pi}{|t|} \frac{1}{(2\pi)^{3/2}} \Big[ \Big(-\sqrt{\frac{\pi}{2}} |\wt{f}(k)|^2 \wt{f}(k)
			+ \sqrt{\frac{\pi}{2}} {R_-(-k)} \mathcal{N}^-[f](k) \Big) \mathbf{1}_+(k)
			\\
			+\sqrt{\frac{\pi}{2}}T(k) \mathcal{N}^-[f](-k) \mathbf{1}_-(k) \Big].
		\end{split}
	\end{align}
	
	For the $\delta$ part, the analysis is standard. 
	For $\iota\in\{+,-\}$ and
	\begin{align*}
		\begin{split}
			& \mathcal{N}_\iota^{\delta}(t,k) := \frac{1}{(2\pi)^{3/2}} \sum_{\e_0,\e_1,\e_2,\e_3 \in \{+,-\}}
			\iiint e^{it(-k^2+(\e_0k-p+\e_2m-\e_3n)^2-m^2+n^2)} 
			\\ & \times \overline{a_\iota^{\e_0}(k)} 
			\, a_\iota^{\e_1}(\e_1(\e_0k-p+\e_2m-\e_3n)) \,
			\overline{a_\iota^{\e_2}(m)} \, a_\iota^{\e_3}(n) \, 
			\\
			& \times
			\tilde{f}(t,\e_1(\e_0k-p+\e_2m-\e_3n)) \overline{\tilde{f}(t,m)}\tilde{f}(t,n) \,\sqrt{\frac{\pi}{2}\delta_{0}(p)}\,dn dm dp,
		\end{split}
	\end{align*} 
	we have, up to  $\mathcal{O}( {\| u \|}_{X_T}^{3}  t^{-1-\rho})$
	\begin{align}
		\label{Nd+}
		\begin{split}
			\mathcal{N}_+^{\delta} & \approx \frac{1}{(2\pi)^{3/2}}\frac{\pi}{|t|} \sqrt{\frac{\pi}{2}}
			\Big[ T(-k) \mathcal{N}^+[f](k) \mathbf{1}_+(k)
			+ \Big(|\wt{f}(k)|^2 \wt{f}(k) +  R_+(k) \mathcal{N}^+[f](-k) \Big) \mathbf{1}_-(k) \Big]
		\end{split},
	\end{align}
	and
	\begin{align}
		\label{Nd-}
		\begin{split}
			\mathcal{N}_-^{\delta} & \approx \frac{1}{(2\pi)^{3/2}}\frac{\pi}{|t|} \sqrt{\frac{\pi}{2}}
			\Big[\Big(|\wt{f}(k)|^2 \wt{f}(k) + {R_-(-k)} \mathcal{N}^-[f](k) \Big) \mathbf{1}_+(k)
			+ {T(k)} \mathcal{N}^-[f](-k) \mathbf{1}_-(k) \Big].
		\end{split}
	\end{align}
	
	Putting together \eqref{Npv+}-\eqref{Npv-} and \eqref{Nd+}-\eqref{Nd-}, we obtain
	\begin{align*}
		\begin{split}
			\mathcal{N}_S  & = \mathcal{N}_+ + \mathcal{N}_-
			\\& =\mathcal{N}_+^{\delta}(t,k)+\mathcal{N}_+^{\mathrm{p.v.}}(t,k)
			+\mathcal{N}_-^{\delta}(t,k)+\mathcal{N}_-^{\mathrm{p.v.}}(t,k)
			+  \mathcal{O}( {\| u \|}_{X_T}^{3} t^{-1-\rho})
			\\& =\frac{1}{(2\pi)^{3/2}}\frac{\pi}{|t|} \sqrt{\frac{\pi}{2}}\Big(2|\wt{f}(k)|^2 \wt{f}(k)\Big)
			+  \mathcal{O}( {\| u \|}_{X_T}^{3} t^{-1-\rho})
			\\
			& = \frac{1}{2t}|\wt{f}(k)|^2 \wt{f}(k) +  \mathcal{O}( {\| u \|}_{X_T}^{3} t^{-1-\rho}).
		\end{split}
	\end{align*}
	From \eqref{ODE'} and \eqref{ODEregular} we obtain \eqref{secasODE}.
	
	The estimate \eqref{secas10} can be obtained from \eqref{secasODE}, 
	using that $\partial_t w(t,k) =  \mathcal{O}({\| u \|}_{X_T}^{3} t^{-1-\rho})$ for $|k|\gtrsim t^{-3\alpha}$,
	and the inequality \eqref{apriori2} for $|w(t_j,k)|$, $j=1,2$, when $|k|\lesssim t_j^{-3\alpha}$.
\end{proof}

\medskip
\appendix
\section{Improved Localized decay for generic potentials}\label{sec:localdecay}

The following lemma is an improvement of Lemma \ref{lemlocdecinftyng} in the case of generic potentials.

\begin{lem}[Improved $L^\infty$ local decay]\label{lemlocdecinfty}
	Suppose $\jx^\gamma V(x)\in L^1$ with $\gamma \geq 3$ and $V$ is generic. 
	Then, for the perturbed flow one has
	\begin{align}\label{locdecinfty0}
		\left\Vert \jx^{-2} e^{iHt} h \right\Vert_{L^{\infty}_x}
		\lesssim |t|^{-1}{\big\| \tilde{h} \big\|}_{H^1_k}.
	\end{align}
\end{lem}

\begin{rem}
	The estimate \eqref{locdecinfty0} is consistent with the estimate of Krieger-Schlag \cite{KS} 
	and Schlag \cite{Sch} who proved (in more general scenarios as well)
	\begin{align}\label{KS}
		\left\Vert \jx^{-1}  e^{iHt} h  \right\Vert_{L^{\infty}_x}
		\lesssim |t|^{-3/2}{\big\| \jx h \big\|}_{L^1_x}.
	\end{align}
	In particular \eqref{KS} implies \eqref{locdecinfty0} by noticing that 
	the boundedness of wave operators for generic potentials, gives
	\[{\big\| \tilde{h} \big\|}_{H^1_k} \approx {\big\| \jx h \big\|}_{L^2_x}\]
	and then looking separately at the regions $|x| \leq |t|$ and $|x|>|t|$.
	The use of $L^2$ based norms, instead of $L^1$ as in the right-hand side of \eqref{KS}, 
	will be useful when considering the stability of solitons for cubic NLS
	(and similar models). 
	For the nonlinear analysis of this paper (Section \ref{sec:Cubic})
	the rate of decay $|t|^{-\frac{3}{4}}$ of Lemma \ref{lemlocdecinftyng} is enough,
	and we do not need to use Lemma \ref{lemlocdecinfty}.
	Nevertheless, since this is of independent interest, we give its proof below 
	(actually we prove the more general bounds \eqref{locdecinfty} and \eqref{locdecinftyR})
	and thus provide also an alternative approach to the proof of \eqref{KS}.
\end{rem}

\begin{proof}[Proof of Lemma \ref{lemlocdecinfty}]
	We follow the same spirit of the proof of Lemma \ref{lemlocdecinftyng}, and show the stronger estimates
	\begin{align}\label{locdecinfty}
		\left\Vert \jx^{-2} \big( e^{iHt} h \big)_A \right\Vert_{L^{\infty}_x}
		= \left\Vert \jx^{-2} \int\mathcal{K}_A(x,k)e^{ik^{2}t}\tilde{h}(k)\,dk\right\Vert_{L^{\infty}_x}
		\lesssim |t|^{-1}{\big\| \tilde{h} \big\|}_{H^1_k}, \qquad A \in \{0,S\},
	\end{align}
	and, for $\gamma-3\geq\beta$,
	\begin{align}\label{locdecinftyR}
		\left\Vert \jx^\beta \big( e^{iHt} h \big)_R \right\Vert_{L^{\infty}_x}
		= \left\Vert \jx^\beta \int\mathcal{K}_R(x,k)e^{ik^{2}t}\tilde{h}(k)\,dk\right\Vert_{L^{\infty}_x}
		\lesssim |t|^{-1}{\big\| \tilde{h} \big\|}_{H^1_k}. 
	\end{align}
	
	{\it Proof of \eqref{locdecinfty0}}.
	We start with the case $A=0$ and write down explicitly
	\[
	e^{it\left(-\partial_{xx}+V\right)}h
	=\frac{1}{\sqrt{2\pi}}\int\mathcal{K}(x,k)e^{ik^{2}t}\tilde{h}(k)\,dk.
	\]
	We can focus on the contribution to the integral with $k\geq0$, since the analysis for $k\leq 0$ is identical.
	We write, see \eqref{matKk>0}, 
	\begin{align}\label{locdecpr11}
		\begin{split}
			& \int_{k\geq0}\mathcal{K}(x,k)e^{ik^{2}t}\tilde{h}(k)\,dk = A_+ + A_-,
			\\
			& A_+ := \chi_+(x) \int_{k\geq0} T(k) m_+(x,k) e^{ixk + ik^{2}t}\tilde{h}(k)\,dk,
			\\
			& A_- := \chi_-(x) \int_{k\geq0} \big(m_-(x,-k)e^{ikx} + R_-(k)m_-(x,k)e^{-ikx}\big) 
			e^{ik^{2}t}\tilde{h}(k)\,dk,
		\end{split}
	\end{align}
	
	For the first term in \eqref{locdecpr21} we integrate by parts in $k$ 
	and write
	\begin{align*}
		2it \, A_+ 
		& = \chi_+(x) \int_{k\geq0} \frac{T(k)}{k} e^{ixk} m_+(x,k) \tilde{h}(k)\,de^{ik^{2}t} = A_1 + A_2 + A_3,
		\\
		A_1 & := -\chi_+(x) \int_{k\geq0} e^{ixk} m_+(x,k) e^{ik^{2}t}
		\partial_{k}\big(T(k)k^{-1} \tilde{h} \big)\,dk,
		\\
		A_2 & := -\chi_+(x)\int_{k\geq0} T(k)k^{-1} e^{ixk}\partial_{k}m_+(x,k) e^{ik^{2}t}\tilde{h}(k)\,dk,
		\\
		A_3 & := - \chi_+(x)ix \int_{k\geq0} T(k)k^{-1} e^{ixk} m_+(x,k) e^{ik^{2}t}\tilde{h}(k)\,dk.
	\end{align*}
	
	For the first term, using $\chi_+(x) | m_+(x,k) | \lesssim 1$, see \eqref{Mestimates1}, we have
	\begin{align*}
		\big|A_1(x)\big| 
		\lesssim \int_{k\geq0} \big| \partial_{k}\big(T(k)k^{-1} \tilde{h}\big) \big| \,dk 
	\end{align*}
	From \eqref{estiTR} and Lemma \ref{lem:estiTRTaylor} we have
	\begin{align}\label{T/k}
		\big|T(k)k^{-1}\big| + \big|\partial_{k}\big(T(k)k^{-1}\big)\big| \lesssim {\langle k \rangle}^{-1}.
	\end{align}
	In particular, the expressions above are in $L^2$ and applying Cauchy-Schwarz we obtain
	\begin{align*}
		\big|A_1(x)\big| \lesssim {\| \tilde{h} \|}_{H^1}
	\end{align*}
	which is consistent with the desired bound.
	$A_2$ and $A_3$ can be estimated similarly: 
	using $\chi_+(x) \jx^{-1}| \partial_k m_+(x,k) | \lesssim 1$, see \eqref{Mestimates2}, and \eqref{T/k} we have
	\begin{align*}
		\jx^{-1} \Big( \big|A_2(x)\big| + \big|A_3(x)\big| \Big) \lesssim {\| \tilde{h} \|}_{L^2}.
	\end{align*}
	
	Next, we look at $A_-$ in \eqref{locdecpr11}. We want to apply similar argument to those used for $A_+$ above,
	but we need to pay some more attention to the vanishing of the integrand at $k=0$
	since, in the proof of the bound for $A_1$ we used the vanishing of $T(k)$ at $k=0$.
	Using a smooth cutoff function to distinguish $|k|\gtrsim 1$ and $|k| \lesssim 1$ we write
	\begin{align}\label{locdecpr13}
		\begin{split}
			& A_- = A_4 + A_5 + A_6 + A_7,
			\\
			& A_4 = \chi_-(x) \int \varphi_{>1}(k) \big(m_-(x,-k)e^{ikx} + R_-(k)m_-(x,k)e^{-ikx}\big) 
			e^{ik^{2}t} \tilde{h}(k)\,dk,
			\\
			& A_5 = \chi_-(x) \int \varphi_{\leq 1}(k) \big(R_-(k) + 1\big) e^{-ikx}m_-(x,k)
			e^{ik^{2}t} \tilde{h}(k)\,dk,
			\\
			& A_6 = \chi_-(x) \int \varphi_{\leq 1}(k) \big(m_-(x,-k)-m_-(x,k)\big)e^{ikx}  
			e^{ik^{2}t} \tilde{h}(k)\,dk,
			\\
			& A_7 = \chi_-(x) \int \varphi_{\leq 1}(k) \big(e^{ikx}-e^{-ikx}\big) m_-(x,k) 
			e^{ik^{2}t} \tilde{h}(k)\,dk.
		\end{split}
	\end{align}
	
	On the support of $A_4$ we have $|k|\gtrsim 1$, so we can integrate by parts in $k$ 
	using $e^{itk^2} = (2it k)^{-1} \partial_k e^{itk^2}$ as before, and gain a factor of $|t|^{-1}$ 
	without introducing any singularity in $k$. We can then proceed exactly as in the estimate of $A_+$.
	
	On the support of $A_5$ we have $|k|\lesssim 1$ and, see Lemma \ref{lem:estiTRTaylor}, 
	$R_-(k) + 1 = \mathcal{O}(k)$.
	In particular we see from Lemma \ref{lem:estiTRTaylor} that $R_-(k) + 1$ behaves exactly like $T(k)$,
	so this term is essentially the same as $A_+$ and can be treated identically.
	
	For $A_6$ we first notice that 
	\begin{align*}
		m_-(x,-k)- m_-(x,k) = -\int_{-1}^1 (\partial_k m_-)(x,zk) \, dz \cdot k
		=: a(x,k) \cdot k.
	\end{align*}
	We can then use the $k$ factor in the above right-hand side to integrate by parts;
	and since $|\chi_-(x)\partial_k^p a(x,k)| \lesssim 1$, for $p=0,1$, see \eqref{Mestimates2},
	we get the bound 
	\begin{align*}
		\big| \jx^{-1}A_6(x) \big| \lesssim {\| \tilde{h} \|}_{H^1}.
	\end{align*}
	
	The last term $A_7$ is similar to $A_6$ by observing that, on the support of $\chi_-$,
	\begin{align*}
		\big| k^{-1} (e^{ikx}-e^{-ikx}) m_-(x,k) \big| 
		+ \Big| \partial_k\big( k^{-1} (e^{ikx}-e^{-ikx}) m_-(x,k) \big) \Big|  \lesssim \jx^2.
	\end{align*}

	\medskip
	\noindent
	{\it Proof of \eqref{locdecinftyR}}.
	The estimate involving $\mathcal{K}_R$ 
	follows from the same arguments used in the case of $\mathcal{K}$, 
	by essentially replacing 
	$m_\pm(x,\pm k)$ with $m_\pm(x,\pm k)-1$, 
	and using the estimates \eqref{Mestimates2} from Lemma \ref{lem:Mestimates}. 
	We just notice that if $\gamma-3\geq \beta$ then $|\jx^\beta\jx^2(m_\pm(x,\pm k)-1)|<\infty$.
\end{proof}

\bigskip

\end{document}